\documentclass[a4j, 12pt, draft, reqno]{amsart}
\usepackage{url}
\usepackage{amssymb}
\usepackage{enumerate}
\usepackage{braket}
\usepackage{ascmac}
\usepackage{amsmath}
\usepackage{mathrsfs}	
\usepackage{bm}
\usepackage{mathtools}
\usepackage[dvipdfmx]{hyperref}
\usepackage{color}
\usepackage{comment}

\allowdisplaybreaks[3]
\allowdisplaybreaks[4]

\theoremstyle{definition}
\newtheorem{definition}{Definition}[section]

\newtheorem{rem}[definition]{Remark}

\newtheorem*{ackname}{Acknowledgment}

\theoremstyle{theorem}
\newtheorem{theorem}{Theorem}
\newtheorem{corollary}[definition]{Corollary}
\newtheorem{lemma}[definition]{Lemma}
\newtheorem{proposition}[definition]{Proposition}

\newtheorem*{theorem*}{Theorem}

\numberwithin{equation}{section}

\mathtoolsset{showonlyrefs=true}

\let\sl\l
\renewcommand\l{%
	\leavevmode
  \ifmmode
    \left
  \else
    \sl
  \fi}

  \let\sL\L
  \renewcommand\L{%
    \leavevmode
    \ifmmode
      \mathscr{L}
    \else
      \sL
    \fi}

\renewcommand{\set}[2]{\left\{ #1 \; | \; #2 \right\}}
\renewcommand{\Set}[2]{\left\{ #1 \; \middle| \; #2 \right\}}

\newcommand{\inp}[1]{\left\langle #1 \right\rangle}
\newcommand{\EXP}[1]{\mathbb{E}\left[ #1 \right]}

\newcommand{\CC}{\mathbb{C}}
\newcommand{\RR}{\mathbb{R}}

\newcommand{\ZZ}{\mathbb{Z}}

\newcommand{\EE}{\mathbb{E}}
\newcommand{\PP}{\mathbb{P}}
\newcommand{\e}{\varepsilon}
\newcommand{\s}{\sigma}

\newcommand{\ms}{\mathscr}

\newcommand{\mca}{\mathcal}
\newcommand{\Lam}{\Lambda}
\newcommand{\lam}{\lambda}

\newcommand{\qquadf}{\qquad \qquad \qquad \qquad}

\newcommand{\te}{\tilde{\eta}}
\newcommand{\meas}{\operatorname{meas}}
\newcommand{\Li}{\operatorname{Li}}

\newcommand{\ol}{\overline}
\newcommand{\supp}{\operatorname{supp}}
\newcommand{\Log}{\operatorname{Log}}

\newcommand{\DF}{M}
\newcommand{\nDF}{N}

\renewcommand{\a}{\alpha}
\renewcommand{\b}{\beta}
\renewcommand{\r}{\right}

\renewcommand{\epsilon}{\varepsilon}

\renewcommand{\phi}{\varphi}
\renewcommand{\Re}{\operatorname{Re}}
\renewcommand{\Im}{\operatorname{Im}}

\begin{document}

\title[Probability density functions for $\te_{m}(s)$]{On the value-distribution of iterated integrals of the logarithm 
of the Riemann zeta-function II: Probabilistic aspects
}

\author[K. Endo]{Kenta Endo}

\address[Endo]{Graduate School of Mathematics, Nagoya University, Chikusa-ku, Nagoya 464-8602, Japan.}
\email{m16010c@math.nagoya-u.ac.jp}

\author[S. Inoue]{Sh\=ota Inoue}

\address[Inoue]{\phantom{a}Department of Mathematics, \phantom{a} Tokyo Institute of Technology, \phantom{a} 2-12-1 Ookayama, 
Meguro-ku, Tokyo 152-8551, Japan.}

\email{inoue.s.bd@m.titech.ac.jp}

\author[M. Mine]{Masahiro Mine}

\address[Mine]{Faculty of science and technology, Sophia University, 7-1 Kioi-cho, Chiyoda-ku, Tokyo 102-8554, Japan}
\email{m-mine@sophia.ac.jp}



\maketitle

\begin{abstract}
  In this paper, we discuss the value-distribution of the Riemann zeta-function.
  The authors give some results for the discrepancy estimate and large deviations in the limit theorem by Bohr and Jessen.
\end{abstract}


\section{\textbf{Introduction and statement of main results}}\label{sec1}

\subsection{Introduction}
The study for the behavior of the Riemann zeta-function $\zeta$ plays an important role to understand the distribution of zeros of $\zeta$ 
and prime numbers.
We in this paper discuss this theme based on the following celebrated result due to Bohr and Jessen.
Throughout this paper, $\meas(\cdot)$ stands for the Lebesgue measure on $\RR$, and $\PP_{T}$ is defined by
\begin{align*}
  \PP_{T}(f(t) \in A) 
  = \frac{1}{T} \meas \set{t \in [T, 2T]}{f(t) \in A}
\end{align*}
for any Lebesgue measurable set $A$ and any Lebesgue measurable function $f$.

\begin{theorem*}[Bohr and Jessen in \cite{BJ1930}]
    Let $\s > \frac{1}{2}$.
    There exists a probability measure $P_{\s}$ on $\CC$ such that for any rectangle $\mathcal{R} \subset \CC$ 
    with edges parallel to the axes
    \begin{align}   \label{LTBJ}
        \PP_{T}(\log{\zeta(\s + it)} \in \mathcal{R})
        \sim P_{\s}(\mathcal{R})
    \end{align}
    as $T \rightarrow + \infty$. 
    Moreover, the probability measure $P_{\s}$ has a probability density function.
\end{theorem*}

Here, we mention the branch of $\log{\zeta(s)}$. Define $\log{\zeta(\s + it)} = \int_{\infty}^{\s} \frac{\zeta'}{\zeta}(\a + it) d\a$ 
if $t$ is not zero or $t$ is not an imaginary part of a zero of $\zeta(s)$. 
If $t = 0$, then $\log{\zeta(\s)} = \lim_{\e \downarrow 0}\log{\zeta(\s + it)}$, 
and if $t$ is an imaginary part of a zero $\rho = \b + i\gamma$ of $\zeta(s)$, 
then $\log{\zeta(\s + i\gamma)} = \lim_{t \uparrow \gamma} \log{\zeta(\s + it)}$.

We can write the probability measure $P_{\s}$ explicitly as follows.
Let $\mathcal{P}$ be the set of prime numbers.
Let $\{ X(p) \}_{p \in \mathcal{P}}$ be a sequence of independent random variables on a probability space $(\Omega, \mathscr{A}, \PP)$ 
and uniformly distributed on the unit circle in the complex plane. Define $\zeta(\s, X) = \prod_{p}\l( 1 - X(p)p^{-\s} \r)^{-1}$.
Then, we have
\begin{align*}
  P_{\s}(A) = \PP( \log \zeta(\s, X) \in A)
\end{align*}
for any Borel set $A$ on $\CC$.

From the theorem of Bohr and Jessen and more recent results, 
we can guess the behavior of the Riemann zeta-function in the critical strip via $P_{\s}$.
Hattori and Matsumoto \cite{HM1999} showed that for $\mca{L} = [-\ell, \ell] \times i[-\ell, \ell]$,
\begin{align}   \label{HMMLD}
  P_{\s}(\CC \setminus \mca{L})
  = \exp\l( -(A(\s) + o(1))\l(\ell (\log{\ell})^{\s}\r)^{\frac{1}{1-\s}} \r)
\end{align}
for $\frac{1}{2} < \s < 1$ as $\ell \rightarrow + \infty$. Here, $A(\s)$ is expressed by 
\begin{align} \label{def_A(s)}
  A(\s) = \l( \frac{\s^{2\s}}{(1 - \s)^{2\s - 1} G(\s)^{\s}} \r)^{\frac{1}{1-\s}}, 
\end{align}
where $G(\s) = \int_{0}^{\infty}\log{I_{0}(u)}u^{-1-\frac{1}{\s}}du$, and $I_{0}$ is the modified $0$-th order Bessel function.
Moreover, the error term was improved by Lamzouri \cite{L2011}.
From these results, we can grasp the behavior of $P_{\s}$ well.
Therefore, the probability measure $P_{\s}$ is a good model for the distribution function of the Riemann zeta-function.

It is a natural question to ask how precise the limit theorem \eqref{LTBJ} is.
Lamzouri, Lester, and Radziwi\l\l \ \cite{LLR2019} studied this problem, 
and gave some results involving the discrepancy estimate and large deviations.
In particular, the large deviation result is related to the Lindel\"of Hypothesis, and hence interesting.
Precisely, they gave the result that the asymptotic formula
\begin{align}
  \label{LTBJRW}
  \PP_{T}(\log |\zeta(\s + it)| > \tau)
  \sim \PP(\log |\zeta(\s, X)| > \tau)
\end{align}
holds for $\frac{1}{2} < \s < 1$, $\tau = o\l( \frac{(\log{T})^{1-\s}}{(\log{\log{T}})^{1/\s}} \r)$\footnote{
  The range of $\tau$ in their paper is wider. 
  However, the range of $\tau$ such that asymptotic formula \eqref{LTBJRW} holds is $o\l( \frac{(\log{T})^{1 - \s}}{(\log{\log{T}})^{1/\s}} \r)$ 
  because their error term is not equal to $o(1)$ when $\tau$ is outside of the range.
} as $T \rightarrow + \infty$.
If the range of $\tau$ in their result could be extended to an appropriate range
(e.g.\ $\tau \leq C \frac{(\log{T})^{1 - \s}}{(\log{\log{T}})^{\s}}$ 
with $C$ a suitable positive constant), 
then it would hold that for $\frac{1}{2} < \s < 1$
\begin{align}   \label{CBRCS}
  \log{|\zeta(\s + it)|}
  \leq C(\s) \frac{(\log{t})^{1 - \s}}{(\log{\log{t}})^{\s}}
\end{align}
for any $t \geq 3$, and 
\begin{align}   \label{ORRM}
  \log{|\zeta(\s + it)|}
  = \Omega\l( \frac{(\log{t})^{1 - \s}}{(\log{\log{t}})^{\s}} \r)
\end{align}
as $t \rightarrow + \infty$.
In fact, the former estimate is conjectural, and the latter $\Omega$-estimate has been proved by Montgomery \cite{M1977}.
Hence, we are interested in the dependency of $\mca{R}$ with respect to $T$.
In this paper, we extend the range of $\tau$ due to Lamzouri, Lester, and Radziwi\l\l \ 
to $\tau = o\l( \frac{(\log{T})^{1-\s}}{\log{\log{T}}} \r)$ in Theorem \ref{thmLD} below.
Moreover, we generalize this problem to iterated integrals of $\log{\zeta(s)}$.

We define 
\begin{gather*}
\te_{m}(\s + it) = \int_{\s}^{\infty}\te_{m-1}(\a + it)d\a
\end{gather*}
for $m \in \ZZ_{\geq 1}$ recursively, where $\te_{0}(s) = \log{\zeta}(s)$.
We also define the corresponding random eta-function $\te_{m}(\s, X)$ by
\begin{align}\label{eqdefeta}
  \te_{m}(\s, X) = \sum_{p}\te_{m, p}(\s, X(p)), 
\end{align}
where
\begin{gather}\label{eqlocal}
  \te_{m, p}(\s, w)
  =\frac{\Li_{m + 1}(p^{-\s} w)}{(\log{p})^{m}}
  =\sum_{k=1}^\infty\frac{w^k}{kp^{k\s}(\log{p}^k)^{m}}
\end{gather}
for $w \in \CC$ with $|w|=1$. 
One can show that \eqref{eqdefeta} converges almost surely if $\s>1/2$ and $m\in\ZZ_{\geq0}$; see Lemma \ref{lemASConv}. 
The purpose of this paper is to study the value-distribution of $\te_{m}(\s+it)$ in $t$-aspect. 
Our work on the value-distribution of $\te_{m}$ is motivated in the following fact:

\begin{proposition} \label{GCGCS}
  The following statements are equivalent.
  \begin{itemize}
    \item[(A)] The Lindel\"of Hypothesis is true.
    \item[(B)] For any positive integer $m$, the estimate $\Re \te_{m}(\s + it) = o(\log{t})$ holds for any fixed $\s > \frac{1}{2}$ 
    as $t \rightarrow + \infty$.
    \item[(C)] There exists a positive integer $m$ such that
    the estimate $\Re \te_{m}(\s + it) = o(\log{t})$ holds for any fixed $\s > \frac{1}{2}$ as $t \rightarrow + \infty$. 
  \end{itemize}
\end{proposition}

This proposition is shown in the Appendix section.

\subsection{Statement of main results} \label{Sec_SMR}

The first main result presents a discrepancy estimate for the value-distribution of $\te_{m}(\sigma+it)$, that is, an upper bound for the quantity  
\begin{gather}\label{eq:DB}
D_{\sigma, m}(T)
=\sup_{\mca{R}}
\l|\PP_T(\te_{m}(\s+it)\in\mca{R})-\PP(\te_{m}(\s, X)\in\mca{R})\r|, 
\end{gather}
where $\mca{R}$ runs through all rectangle in $\CC$ with edges parallel to the axes.

\begin{theorem}\label{thmDB1}
Let $1/2 < \s < 1$ and $m\in\ZZ_{\geq0}$. 
Then we have 
\begin{align*}
D_{\sigma, m}(T)
\ll_{\s, m} 
\frac{1}{(\log{T})^{\s}(\log{\log{T}})^{m}}.
\end{align*}
\end{theorem}

This theorem is a generalization of \cite[Theorem 1]{LLR2019}.
Actually, the result in the case $m = 0$ coincides with theirs.

Next, we give the large deviation result
of $\Re e^{-i\a}\te_m(\s+it)$ with any angle $\a\in\RR$.

\begin{theorem}\label{thmLD}
Let $1/2 < \s < 1$, $\a \in \RR$, and $m \in \ZZ_{\geq 0}$. 
There exists a positive constant $a = a(\s, m)$ such that for large numbers $T$, $\tau$ 
with $\tau \leq a \frac{(\log{T})^{1-\s}}{(\log\log{T})^{m+1}}$ we have 
\begin{align*}
  \PP_T\l(\Re e^{-i\a}\te_m(\s+it)>\tau\r)
  =\PP\l(\Re e^{-i\a}\te_m(\s, X) > \tau\r)(1 + E), 
\end{align*}
where
\begin{align*}
  E \ll_{\s, m} \frac{\tau^{\frac{\s}{1-\s}}(\log{\tau})^{\frac{\s+m}{1-\s}}}{(\log{T})^{\s}(\log\log{T})^m}.
\end{align*}
\end{theorem}

This theorem in the case $m = 0$ is an improvement of work due to Lamzouri, Lester, and Radziwi\l\l \ \cite{LLR2019}.
Actually, our estimate of $E$ is sharper than theirs.

\begin{rem}
  Our result is to concern the case $1/2 < \s < 1$.
  The authors further expect that we can prove a similar result even for $\s = 1/2$ when $m$ is positive.
  Indeed, Proposition \ref{Key_Prop} of this paper remains available in that case. 
  However, we should remark that Lemmas \ref{CREP} and \ref{CREP2} give only weak error terms for $\s=1/2$. 
  In the authors' opinion, an appropriate goal is to prove a large deviation result in the same range as \cite[Theorem 1]{II2020}, 
  but they can prove only a weaker result at present.
  This is an ongoing problem.
\end{rem}

\begin{rem}
  Some people often consider the other situation by using $\PP_{T}^{*}(f(t) \in A) = \frac{1}{T} \meas \set{t \in [0, T]}{f(t) \in A}$ 
  in place of $\PP_{T}$.
  Our results also recover the situation because we can show the same results for $\PP_{T}^{*}$
  by dividing the range as $[0, T] = [0, 2^{-N} T] \cup \bigcup_{n = 1}^{N}(2^{-n}T, 2^{-n+1}T]$ 
  with $N = \l[ \frac{\log{T}}{2 \log{2}} \r] + 1$.
  The result for $\PP_{T}^{*}$ is obtained by applying our theorems to each partition $(2^{-n}T, 2^{-n+1}T]$.
\end{rem}

\subsection{Outline of the paper}

The following sections are organized as follows: 
In Section \ref{Sec_FMGF}, we introduce a Dirichlet polynomial $P_{m,Y}(\s+it)$ and a random Dirichlet polynomial $P_{m,Y}(\s, X)$ as finite truncations of $\te_m(\s+it)$ and $\te_m(\s,X)$, respectively.  
Then we state an asymptotic formula which relates the moment generating function of $P_{m,Y}(\s+it)$ to that of $P_{m,Y}(\s, X)$. 
The formula plays a key role in the proofs of both Theorem \ref{thmDB1} and Theorem \ref{thmLD}. 
In Section \ref{secRDS}, we prove the existence and some properties of the probability density function $\DF_{\s,m}$ attached to $\te_{m}(\s, X)$.
These properties are used in the proof of Theorem \ref{thmDB1}.
In Section \ref{Sec_ED_ZvDP}, we give the fact that the set of $t \in [T, 2T]$ 
for which the gap between $\te_{m}(\s+it)$ and $P_{m,Y}(\s+it)$ is large has a satisfactorily small measure. 
This fact is also used in the proofs of both Theorem \ref{thmDB1} and Theorem \ref{thmLD}.
In Section \ref{Sec_PDB}, we give the proof of Theorem \ref{thmDB1}. 
The main tool is the Beurling-Selberg function that gives a smooth approximation of the indicator function $\bm{1}_{\mathcal{R}}$. 
Moreover, it plays an important role even in improving the large deviation result due to Lamzouri, Lester, and Radziwi{\l\l}  \cite{LLR2019}.
The contents in Sections \ref{Sec_PThmLD}, \ref{Sec_FPDF}, \ref{Sec_PLD} are preparations for the proof of Theorem \ref{thmLD}.
We in Section \ref{Sec_PThmLD} give asymptotic formulas for the cumulant-generating function of $\te_{m}(\s, X)$ and its derivatives.
In Section \ref{Sec_FPDF}, we define another probability density function $\nDF^{\tau}_{\s,m}$ by using $\DF_{\s,m}$ of Section \ref{secRDS}. 
Then we show some properties of $\nDF^{\tau}_{\s,m}$, which are used in the proof of Theorem \ref{thmLD}.
In Section \ref{Sec_PLD}, we complete the proof of Theorem \ref{thmLD}.

\subsection{Some remarks}
Many parts in our paper are based on the method of Lamzouri, Lester, and Radziwi\l\l \ \cite{LLR2019}.
In particular, Theorem \ref{thmDB1} is shown similarly to Theorem 1.1 in their paper.
On the other hand, there are important changes in the proof of Theorem \ref{thmLD} from their method. 

\begin{itemize}
\item[$(\mathrm{i})$]
According to \cite[Section 7]{LLR2019}, we denote by $M(z)$ the cumulant-generating function $M(z) = \log \EXP{|\zeta(\s, X)|^{z}}$ for $\s>1/2$. 
Proposition 7.5 of \cite{LLR2019} asserts certain estimates of $M^{(n)}(k)$ for $n=0,1,2$ and of $M^{(3)}(k+it)$, 
where $k>0$ is a large real number and $|t| \leq k$. 
The details of the proof were omitted in their paper, but the authors suspect that it is not so trivial for the following reasons. 
Firstly, it appears that the non-vanishing of the moment-generating function $E_p(z)=\EXP{\l( 1 -X(p) p^{-\s} \r)^{-z}}$ 
is not clear off the real axis when $p$ is small. 
It causes the difficulty of ensuring the regularity of $M(z) = \sum_{p} \log E_p(z)$. 
Moreover, the estimates of $M^{(n)}(k)$ for $n \geq2$ do not follow directly along the same line as \cite[Proposition 3.2]{L2011}. 
Indeed, we need the upper bound $\frac{d^n}{dk^n}\log E_p(k)\ll k^{1-n} p^{-\sigma}$ for small $p$, which is nontrivial for $n \geq2$. 
To overcome these difficulties, we in Section \ref{SSec_RPL} discuss some properties of the polylogarithmic functions. 
Then we give an approximation formula of $E_p(z)$ as Proposition \ref{thmF} to ensure its non-vanishing for small $p$. 
This formula yields the regularity of $\log{E}_p(z)$ in the domain $\Re z > C$ and $|\Im z| \leq \Re z$ with $C > 0$ sufficiently large. 
As a result, we prove an approximation formula of $M(z)$ in this domain. 
The formula plays an important role in the proof of Proposition \ref{propCumu1}, which corresponds to Proposition 7.5 in \cite{LLR2019}. 
In our paper, the estimates of $M^{(n)}(k)$ are justified by the argument using Cauchy's integral formula.

\item[$(\mathrm{ii})$]
We observe a rotation of any angle $\a$ by considering the function $\Re e^{-i\a} \te_{m}(\s+it)$ instead of $\log|\zeta(\sigma+it)|$. 
It makes a new difficulty when $\a \neq0$; see Lemmas \ref{lemZeros}--\ref{lemLB}. 
If we consider the case $\a = 0$ only, then the arguments will be simpler because we do not need these lemmas.

\item[$(\mathrm{iii})$]
In the proof of Proposition 7.1 in \cite{LLR2019}, they used the formula 
\begin{gather}\label{eq:2233}
M(s)
=M(\kappa)+it M'(\kappa)-\frac{t^2}{2}M''(\kappa)
+O\left(\frac{\kappa^{\frac{1}{\sigma}-3}}{\log{\kappa}} |t|^3\right)
\end{gather}
for $s=\kappa+it$ with $|t| \leq \kappa$ which follows from Proposition 7.5 in their paper. 
Combining it with the condition $M'(\kappa)=\tau$ and the formula
\begin{gather*}
\frac{e^{\lambda s}-1}{\lambda s^2}
=\frac{1}{\kappa} \left(1-i \frac{t}{\kappa} +O\left(\lambda \kappa +\frac{t^2}{\kappa^2}\right)\right), 
\end{gather*}
they concluded that the estimate 
\begin{align}\label{eq:2246}
&\EXP{|\zeta(\s, X)|^{s}} e^{-\tau s} \frac{e^{\lambda s}-1}{\lambda s^2} \\
&=\frac{1}{\kappa} \EXP{|\zeta(\s, X)|^{\kappa}} e^{-\tau \kappa} \exp\left(-\frac{t^2}{2}M''(\kappa)\right) \\
&\qquad \times
\left(1-i \frac{t}{\kappa} +O\left(\lambda \kappa +\frac{t^2}{\kappa^2}
+\frac{\kappa^{\frac{1}{\sigma}-3}}{\log{\kappa}} |t|^3 \right)\right)
\end{align}
holds for $s=\kappa+it$ with $|t| \leq \kappa$. 
For deducing this, it appears that the expected value $\EXP{|\zeta(\s, X)|^{s}}$ is calculated as 
\begin{align*}
\EXP{|\zeta(\s, X)|^{s}} 
&=\exp\left(M(\kappa)+it M'(\kappa)-\frac{t^2}{2}M''(\kappa)
+O\left(\frac{\kappa^{\frac{1}{\sigma}-3}}{\log{\kappa}} |t|^3\right)\right) \\
&=\EXP{|\zeta(\s, X)|^{\kappa}} e^{-it \tau} \exp\left(-\frac{t^2}{2}M''(\kappa)\right)
\left(1+O\left(\frac{\kappa^{\frac{1}{\sigma}-3}}{\log{\kappa}} |t|^3\right)\right) 
\end{align*}
by using \eqref{eq:2233} and $M'(\kappa)=\tau$. 
However, we remark that the asymptotic formula 
\begin{gather*}
\exp\left(O\left(\frac{\kappa^{\frac{1}{\sigma}-3}}{\log{\kappa}} |t|^3\right)\right)
=1+O\left(\frac{\kappa^{\frac{1}{\sigma}-3}}{\log{\kappa}} |t|^3\right)
\end{gather*}
holds only for $|t| \ll \kappa^{1-\frac{1}{3\sigma}} (\log{\kappa})^{1/3}$ 
since $e^{f(z)}=1+O(|f(z)|)$ is valid when $f$ is bounded or $f(z) \in i\RR$. 
Therefore, estimate \eqref{eq:2246} should be valid only in $|t| \ll \kappa^{1-\frac{1}{3\sigma}} (\log{\kappa})^{1/3}$ as well 
since the error term is bounded only for $|t| \ll \kappa^{1-\frac{1}{3\sigma}} (\log{\kappa})^{1/3}$,
and further we can confirm that the error term does not always belong to $i \RR$. 
For this reason, it appears that the proof of Proposition 7.1 in \cite{LLR2019} is incomplete. 
To complete the proof, we need an additional argument of $\EXP{|\zeta(\s, X)|^{\kappa+it}}$ 
for $\kappa^{1-\frac{1}{3\sigma}} (\log{\kappa})^{1/3} \ll |t| \leq \kappa$. 
It will be described as \eqref{lemBound1td} in Lemma \ref{lemBound1td} of our paper. 
Remark that a similar trouble appears in the proof of \cite[Lemma 3.2]{W2007}, 
but it can be corrected by using the first inequality of Lemma 2.4 in the same paper. 
The proof of \cite[Lemma 6.3]{LRW2008} can be also corrected similarly.

\item[$(\mathrm{iv})$]
Let $N=[\log\log{T}]$ and $\lambda =c(\sigma)(\log{T})^{-\sigma}$ with some constant $c(\sigma)>0$ as in the proof of Theorem 1.3 in \cite{LLR2019}. 
We focus on the formula
\begin{gather}\label{eq:2337}
\exp\left(O\left(\lambda N (\tau \log{\tau})^{\frac{\sigma}{1-\sigma}}\right)\right)
=1+O\left(\frac{(\tau \log{\tau})^{\frac{\sigma}{1-\sigma}} \log\log{T}}{(\log{T})^\sigma}\right)
\end{gather}
which was used in \cite[Equation (7.20)]{LLR2019}. 
Since $\lambda N (\tau \log{\tau})^{\frac{\sigma}{1-\sigma}}$ is bounded only 
in the range $C \leq \tau \ll (\log{T})^{1-\sigma}(\log\log{T})^{-\frac{1}{\sigma}}$ with $C > 0$ large, 
we see that \eqref{eq:2337} is valid in this range only. 
This observation deduces that the admissible range of $\tau$ in Theorem 1.3 of \cite{LLR2019} should be corrected as 
\begin{gather}\label{eq:2351}
C \leq \tau \leq b(\sigma) (\log{T})^{1-\sigma}(\log\log{T})^{-\frac{1}{\sigma}}
\end{gather}
with some constant $b(\sigma)>0$, which was originally stated as 
\begin{gather*}
C \leq \tau \leq b(\sigma) (\log{T})^{1-\sigma}(\log\log{T})^{1-\frac{1}{\sigma}}. 
\end{gather*}
We note that the term $\log\log{T}$ in \eqref{eq:2337} comes from an inequality of Granville--Soundararajan \cite{GS2003}; 
see Lemma \ref{lemGS} of our paper. 
The inequality is used to approximate the indicator function $\bm{1}_{(-\infty, c]}$ by using certain Mellin transforms. 
Our main idea is to use the Beurling--Selberg function in the treatment of indicator functions. 
It enables us to extend \eqref{eq:2351} to  
\begin{gather*}
C \leq \tau \leq b(\sigma) (\log{T})^{1-\sigma}(\log\log{T})^{-1} 
\end{gather*}
which is completely the same as the range of Lamzouri \cite[Theorem 1.1]{L2011}. 
Additionally, we cannot also follow the first equation of (7.20) in \cite{LLR2019} from their reason.
To avoid this issue, we use some properties of probability density functions which are shown in Section \ref{Sec_FPDF}.

\item[$(\mathrm{v})$]
We use the Beurling--Selberg function to deduce our theorems. 
We note that the same result can be shown by using the Berry--Esseen inequality instead of the Beurling--Selberg function. 
See the forthcoming paper by the third author \cite{M2021} for a related work. 
\end{itemize}

\section{\textbf{An approximation formula for moment generating functions of Dirichlet polynomials}}  \label{Sec_FMGF}

Denote by $\mca{A}$ the set of pairs $(\s, m)$ such that
\begin{align}	\label{def_set_A}
\mca{A}=\set{(\s, m)}{\text{$\s>1/2$ and $m\in\ZZ_{\geq0}$}}\cup\set{(\s, m)}{\text{$\s\geq1/2$ and $m\in\ZZ_{\geq1}$}}.
\end{align}
For $1/2 \leq \s < 1$, we put 
\[
\tau(\s) 
= 
\begin{cases}
\s &\text{if $1/2 < \s < 1$}, \\ 
0 &\text{if $\s = 1/2$}. 
\end{cases}
\]
Then we define
\begin{align*}
P_{m, Y}(\s + it)
&= \sum_{p^{k} \leq Y}\frac{p^{-ikt}}{k p^{k\s}(\log{p^{k}})^{m}}, \\
P_{m, Y}(\s, X)
&= \sum_{p^{k} \leq Y}\frac{X(p)^{k}}{k p^{k\s} (\log{p^{k}})^{m}}
\end{align*}
for $(\s, m)\in\mca{A}$ and $Y\geq3$. 
The following result relating the moment-generating functions of $P_{m, Y}(\s + it)$ and $P_{m, Y}(\s, X)$ 
is useful to study the value-distribution of $\te_m(\sigma+it)$. 

\begin{proposition}	\label{Key_Prop}
Let $(\s, m) \in \mca{A}$ with $\sigma<1$. 
Let $T, V>0$ be large. 
Denote by $A_T=A_T(V, Y; \sigma, m)$ the set
\begin{align}	\label{def_A_Y}
A_T = \set{t \in [T, 2T]}{|P_{m, Y}(\s + it)| \leq V}
\end{align}
for $Y\geq3$. 
If we further suppose that 
\begin{gather}\label{eqY}
3 \leq Y \leq \exp\left(\frac{\log{T}}{V^{\frac{1}{1-\s}}(\log{V})^{\frac{m+\tau(\s)}{1-\s}}}\right) 
\end{gather}
holds, then there exist positive constants $b_{1} = b_{1}(\s, m)$ and $b_{2} = b_{2}(\s, m)$ such that for any complex numbers $z_1, z_2$ with $|z_{1}|, |z_{2}| \leq b_{1} V^{\frac{\s}{1-\s}}(\log{V})^{\frac{m+\tau(\s)}{1-\s}}$ we have 
\begin{align}
&\frac{1}{T}\int_{A_T}\exp\l( z_{1}P_{m, Y}(\s+it) + z_{2}\ol{P_{m, Y}(\s+it)} \r) \,dt\\
&= \EXP{\exp\l( z_{1}P_{m, Y}(\s, X) + z_{2}\ol{P_{m, Y}(\s, X)} \r)} +E, 
\end{align}
where $E$ is estimated as
\begin{align*}
E
&\ll\frac{1}{T}
\l( V^{\frac{\s}{1-\s}}(\log{V})^{\frac{m + \tau(\s)}{1-\s}}Y \r)^{V^{\frac{1}{1-\s}}(\log{V})^{\frac{m+\tau(\s)}{1-\s}}}\\
&\qquad+ \exp\l( -b_{2} V^{\frac{1}{1-\s}}(\log{V})^{\frac{m+\tau(\s)}{1-\s}} \r).
\end{align*}
\end{proposition}

\subsection{\textbf{Preliminaries for Proposition \ref{Key_Prop}}}\label{sec3}

\begin{lemma}	\label{CDR}
Let $(\s, m) \in \mca{A}$ with $\sigma<1$.  
Let $T \geq 5$ and $Y \geq 3$.
For any $k, \ell\in\ZZ_{\geq1}$, we have
\begin{align*}
&\frac{1}{T}\int_{T}^{2T}(P_{m, Y}(\s + it))^{k}\l( \ol{P_{m, Y}(\s + it)} \r)^{\ell}dt\\
&= \EXP{ \l(P_{m, Y}(\s, X)\r)^{k} \l( \ol{P_{m, Y}(\s, X)} \r)^{\ell} } 
+ O\l( \frac{(2^{m} Y)^{2(k+\ell)}}{T} \r). 
\end{align*}
Here, the implicit constant is absolute.
\end{lemma}

\begin{proof}
We see that
\begin{align*}
&\int_{T}^{2T}(P_{m, Y}(\s + it))^{k}\l( \ol{P_{m, Y}(\s + it)} \r)^{\ell}dt \\
&=
\sum_{\substack{p_{1}^{a_1}, \dots, p_{k}^{a_{k}} \leq Y \\ q_1^{b_1}, \dots, q_{\ell}^{b_{\ell}} \leq Y}}
\frac{1}
{a_1 p_{1}^{a_{1}\s} (\log{p_{1}^{a_{k}}})^{m} \cdots a_{k} p_{k}^{a_{k}\s} (\log{p_{k}^{a_{k}}})^{m}}\\
&\qquad\qquad\qquad
\times\frac{1}
{b_{1} q_{1}^{b_{1}\s} (\log{q_{1}^{b_{1}}})^{m} \cdots b_{\ell}q_{\ell}^{b_{\ell}\s} (\log{q_{\ell}^{b_{\ell}}})^{m}}
\int_{T}^{2T} \l(\frac{ q_{1}^{b_1} \cdots q_{\ell}^{b_{\ell}} }{ p_{1}^{a_1} \cdots p_{k}^{a_{k}}} \r)^{it}dt\\
&= S_{1} + S_{2}, 
\end{align*}
where $S_1$ is the sum over the terms with $p_{1}^{a_1} \cdots p_{k}^{a_{k}} = q_{1}^{b_1} \cdots q_{\ell}^{b_{\ell}}$, 
and $S_2$ is the sum over the other terms. 
Here, for $p_{1}^{a_1} \cdots p_{k}^{a_{k}} \neq q_{1}^{b_1} \cdots q_{\ell}^{b_{\ell}}$, 
it holds that
\begin{align*}
\int_{T}^{2T}\l(\frac{q_{1}^{b_1} \cdots q_{\ell}^{b_{\ell}}}{p_{1}^{a_1} \cdots p_{k}^{a_{k}}}\r)^{it}dt
\ll \frac{1}{\l|\log\l( q_{1}^{b_{1}} \cdots q_{\ell}^{b_{\ell}} / p_{1}^{a_{1}} \cdots p_{k}^{a_{k}} \r)\r|}
\ll Y^{k+\ell}, 
\end{align*}
and hence we have
\begin{align*}
S_{2} 
\ll Y^{k+\ell}
\l( \sum_{p^{a} \leq Y}\frac{1}{a p^{a\s} (\log{p^{a}})^{m}} \r)^{k+\ell}
\leq Y^{k + \ell} \l( \sum_{n \leq Y}\frac{1}{(\log{2})^{m}} \r)^{k + \ell}
\leq (2^{m} Y)^{2(k+\ell)}.
\end{align*}
We can also write
\begin{align}\label{eq_S1}
&\frac{1}{T}S_{1}= 
\sum_{\substack{p_{1}^{a_1}, \dots, p_{k}^{a_{k}} \leq Y \\ 
q_{1}^{b_1}, \dots, q_{\ell}^{b_{\ell}} \leq Y \\ 
p_{1}^{a_1} \cdots p_{k}^{a_{k}} = q_{1}^{b_1} \cdots q_{\ell}^{b_{\ell}}}}
\frac{1}
{a_1 p_{1}^{a_{1}\s} (\log{p_{1}^{a_{k}}})^{m} \cdots a_{k} p_{k}^{a_{k}\s} (\log{p_{k}^{a_{k}}})^{m}} \\
&\qquad\qquad\qquad\qquad\qquad
\times\frac{1}
{b_{1} q_{1}^{b_{1}\s} (\log{q_{1}^{b_1}})^{m} \cdots b_{\ell}q_{\ell}^{b_{\ell}\s} (\log{q_{\ell}^{b_{\ell}}})^{m}}.
\end{align}
On the other hand, it holds that
\begin{align*}
&\EXP{ \l(P_{m, Y}(\s, X)\r)^{k} \l( \ol{P_{m, Y}(\s, X)} \r)^{\ell} } \\
&=
\sum_{\substack{p_{1}^{a_1}, \dots, p_{k}^{a_{k}} \leq Y \\ q_1^{b_1}, \dots, q_{\ell}^{b_{\ell}} \leq Y}}
\frac{1}
{a_1 p_{1}^{a_{1}\s} (\log{p_{1}^{a_{k}}})^{m} \cdots a_{k} p_{k}^{a_{k}\s} (\log{p_{k}^{a_{k}}})^{m}} \\
&\qquad\qquad\qquad
\times\frac{1}
{b_{1} q_{1}^{b_{1}\s} (\log{q_{1}^{b_{1}}})^{m} \cdots b_{\ell}q_{\ell}^{b_{\ell}\s} (\log{q_{\ell}^{b_{\ell}}})^{m}}
\EXP{ \frac{X(p_1)^{a_{1}} \cdots X(p_{k})^{a_{k}}}{ X(q_1)^{b_1} \cdots X(q_{\ell})^{b_{\ell}} } }. 
\end{align*}
Since the random variables $X(p)$ are independent and uniformly distributed on the unit circle in $\CC$, 
we have
\begin{align}	\label{BEUIR}
\EXP{ \frac{X(p_1)^{a_{1}} \cdots X(p_{k})^{a_{k}}}{ X(q_1)^{b_1} \cdots X(q_{\ell})^{b_{\ell}} } }
=
\begin{cases}
1	&	\text{if $p_{1}^{a_1} \cdots p_{k}^{a_{k}} = q_{1}^{b_1} \cdots q_{\ell}^{b_\ell}$}, \\
0	&	\text{if $p_{1}^{a_1} \cdots p_{k}^{a_{k}} \neq q_{1}^{b_1} \cdots q_{\ell}^{b_\ell}$}. 
\end{cases}
\end{align}
Therefore, we deduce from \eqref{eq_S1} the equality
\begin{align*}
\EXP{ \l(P_{m, Y}(\s, X)\r)^{k} \l( \ol{P_{m, Y}(\s, X)} \r)^{\ell} } 
= T S_{1}, 
\end{align*}
which completes the proof of the lemma.
\end{proof}

\begin{lemma}	\label{VSLS}
Let $\{a(p)\}_{p\in\mca{P}}$ be any complex sequence. 
Let $T \geq 5$ and $Y \geq 3$. 
For $k \in \ZZ_{\geq 1}$ with $Y^{k} \leq T(\log{T})^{-1}$, we have
\begin{align}
\label{VSLS1}
\frac{1}{T}\int_{T}^{2T}\bigg| \sum_{p \leq Y}a(p)p^{-it} \bigg|^{2k}dt
\ll k! \l( \sum_{p \leq Y}|a(p)|^2 \r)^{k}.
\end{align}
Additionally, for any $k \in \ZZ_{\geq 1}$, we have
\begin{align}
\label{VSLS2}
\EXP{ \bigg| \sum_{p \leq Y}a(p) X(p) \bigg|^{2k} }
\leq k! \l( \sum_{p \leq Y}|a(p)|^2 \r)^{k}.
\end{align}
\end{lemma}

\begin{proof}
The former assertion directly follows from \cite[Lemma 8]{II2019}. 
We prove the latter assertion. By equation \eqref{BEUIR}, we see that
\begin{align*}
&\EXP{ \bigg| \sum_{p \leq Y}a(p) X(p) \bigg|^{2k} }\\
&= \sum_{\substack{p_1, \dots, p_{k} \leq Y \\ q_{1}, \dots, q_{k} \leq Y}}
a(p_1) \cdots a(p_{k}) \ol{a(q_{1}) \cdots a(q_{k})}
\EXP{\frac{X(p_1) \cdots X(p_{k}) }{ X(q_{1}) \cdots X(q_{k})} }\\
&\leq k!\sum_{p_{1}, \dots, p_{k} \leq Y}|a(p_{1})|^2 \cdots |a(p_{k})|^{2}
= k! \l( \sum_{p \leq Y}|a(p)|^2 \r)^{k}, 
\end{align*}
which completes the proof of the lemma.
\end{proof}

\begin{lemma}	\label{UBDRP}
Let $(\s, m) \in \mathcal{A}$ with $\sigma<1$.
Let $T>0$ be large and $Y \geq 3$.
There exists a positive constant $C = C(\s, m)$ such that
\begin{align}	\label{UBDRP1}
\frac{1}{T}\int_{T}^{2T} | P_{m, Y}(\s + it) |^{2k}dt
\ll \l( \frac{C k^{1-\s}}{(\log{2k})^{m + \tau(\s)}} \r)^{2k}
\end{align}
for $k \in \ZZ_{\geq 1}$ with $Y^{k} \leq T(\log{T})^{-1}$. 
Additionally, we have 
\begin{align}	\label{UBDRP2}
\EXP{|P_{m, Y}(\s, X)|^{2k}}
\ll \l( \frac{C k^{1-\s}}{(\log{2k})^{m + \tau(\s)}} \r)^{2k}
\end{align}
for any $k \in \ZZ_{\geq 1}$. 
\end{lemma}

\begin{proof}
Suppose that the inequality ${k}\log{2k}<Y$ holds. 
We can write
\begin{align*}
P_{m, Y}(\s + it)
&= \sum_{p \leq Y}\frac{1}{p^{\s + it}(\log{p})^{m}}
+ \sum_{\substack{p^{k} \leq Y \\ k \geq 2 }}\frac{1}{k p^{k(\s + it)} (\log{p^{k}})^{m}}\\
&= \sum_{p \leq Y}\frac{1}{p^{\s + it}(\log{p})^{m}} + O_{\s}(1).
\end{align*}
We then see that
\begin{align*}
&\int_{T}^{2T}|P_{m, Y}(\s + it)|^{2k}dt \\
&\leq9^{k}\Bigg( \int_{T}^{2T}\bigg| \sum_{p \leq k \log{2k}}\frac{1}{p^{\s+it}(\log{p})^{m}} \bigg|^{2k}dt \\
&\qquad\qquad
+ \int_{T}^{2T}\bigg| \sum_{k\log{2k} < p \leq Y}\frac{1}{p^{\s+it}(\log{p})^{m}} \bigg|^{2k}dt + C^{2k} T \Bigg), 
\end{align*}
where $C = C(\s)$ is a positive constant. 
By Lemma \ref{VSLS} and the prime number theorem, it holds that
\begin{align*}
\frac{1}{T}\int_{T}^{2T}\bigg| \sum_{k \log 2k < p \leq Y}\frac{1}{p^{\s+it}(\log{p})^{m}} \bigg|^{2k}dt
&\ll k! \l( \sum_{k \log 2k < p \leq Y}\frac{1}{p^{2\s}(\log{p})^{2m}} \r)^{k}\\
&\ll \l(  \frac{ C_1 k^{1 - \sigma} }{ (\log 2k)^{m + \tau(\s)}} \r)^{2k}, 
\end{align*}
where $C_{1}$ is a positive constant which may depend on $\s$ and $m$.
Furthermore, by the prime number theorem it follows that
\begin{align}\label{eq_FS}
\frac{1}{T}\int_{T}^{2T}\bigg| \sum_{p \leq k\log{2k}}\frac{1}{p^{\s+it}(\log{p})^{m}} \bigg|^{2k}dt
&\ll \l( \sum_{p \leq k\log{2k}}\frac{1}{p^{\s}(\log{p})^{m}} \r)^{2k}\\
&\ll \l(\frac{C_{2} k^{1 - \sigma}}{(\log{2k})^{m+\s}} \r)^{2k}, 
\end{align}
where $C_2$ is also a positive constant which may depend on $\s$ and $m$.
From the above estimates, we obtain estimate \eqref{UBDRP1}.
If the inequality $Y\leq{k}\log{2k}$ holds, then we have 
\begin{align*}
P_{m, Y}(\s + it)
\ll_{\s} \sum_{p\leq{Y}}\frac{1}{p^\sigma(\log{p})^m}
\ll_{m} \frac{k^{1 - \sigma}}{(\log{2k})^{m+\s}} 
\end{align*}
by the prime number theorem. 
Hence estimate \eqref{UBDRP1} follows in this case. 
Similarly, we can prove estimate \eqref{UBDRP2}.
\end{proof}

\begin{lemma}	\label{ESE}
Let $(\s, m) \in \mca{A}$ with $\sigma<1$. 
Let $T, V>0$ be large. 
There exists a small positive constant $c_{1} = c_{1}(\s, m)$ such that
\begin{align}	\label{ESE1}
\PP_T\l( |P_{m, Y}(\s+it)| > V \r)
\leq \exp\l( -c_{1} V^{\frac{1}{1-\s}}(\log{V})^{\frac{m+\tau(\s)}{1-\s}} \r),
\end{align}
if $Y \geq 3$ satisfies \eqref{eqY}. 
Additionally, we have
\begin{align}	\label{ESE2}
\PP\l( |P_{m, Y}(\s, X)| > V \r)
\leq \exp\l( -c_{1} V^{\frac{1}{1-\s}}(\log{V})^{\frac{m+\tau(\s)}{1-\s}} \r)
\end{align}
for any $Y \geq 3$. 
\end{lemma}

\begin{proof}
Let $k \in \ZZ_{\geq 1}$ with $Y^{k} \leq T(\log{T})^{-1}$.
Then we derive from \eqref{UBDRP1} the estimate
\begin{align*}
\PP_T\l( |P_{m, Y}(\s+it)| > V \r)
&\leq \frac{1}{V^{2k}} \frac{1}{T} \int_{T}^{2T} |P_{m, Y}(\s+it)|^{2k}dt\\
&\ll \frac{1}{V^{2k}} \l( \frac{C(\s, m) k^{1-\s}}{(\log{2k})^{m + \tau(\s)}} \r)^{2k}. 
\end{align*}
Hence, choosing $k = \l[c_1V^{\frac{1}{1-\s}}(\log{V})^{\frac{m+\tau(\s)}{1-\s}}\r]$ 
with $c_1$ a suitably small constant depending on $\s$ and $m$, 
we obtain inequality \eqref{ESE1}. 
Note that the inequality $Y^{k} \leq T(\log{T})^{-1}$ holds for this $k$ under condition \eqref{eqY}.
Similarly, by using \eqref{UBDRP2}, we see that
\begin{align*}
\PP\l( |P_{m, Y}(\s, X)| > V \r)
&\leq \frac{1}{V^{2k}}\EXP{|P_{m, Y}(\s, X)|^{2k}}\\
&\ll \frac{1}{V^{2k}} \l( \frac{C(\s, m) k^{1-\s}}{(\log{2k})^{m + \tau(\s)}} \r)^{2k}
\end{align*}
holds for any $Y \geq 3$.
Thus again choosing $k = \l[c_1V^{\frac{1}{1-\s}}(\log{V})^{\frac{m+\tau(\s)}{1-\s}}\r]$, 
we obtain inequality \eqref{ESE2}.
\end{proof}

\subsection{\textbf{Proof of Proposition \ref{Key_Prop}}}

\def\c{c_3}

\begin{proof}[Proof of Proposition \ref{Key_Prop}
]
Let $(\s, m) \in \mca{A}$ be fixed.
Let $T, V$ be large.
Let $3 \leq Y \leq \exp\l( \log{T} / V^{\frac{1}{1 - \s}}(\log{V})^{\frac{m + \tau(\s)}{1 - \s}} \r)$.
By the definition of the set $A_T=A_{T}(V, Y; \sigma, m)$, we find that
\begin{align}	\label{KPEQ1}
&\int_{A_{T}}\exp\l( z_{1}P_{m, Y}(\s+it) + z_{2} \ol{P_{m, Y}(\s+it)} \r)dt\\
&= \sum_{\substack{k + \ell \leq Z \\ k, \ell \in \ZZ_{\geq 0}}}\frac{z_{1}^{k} z_{2}^{\ell}}{k! \ell!}
\int_{A_{T}}P_{m, Y}(\s+it)^{k} \ol{P_{m, Y}(\s+it)}^{\ell}dt 
+ O\l( T\sum_{\substack{k + \ell > Z \\ k, \ell \in \ZZ_{\geq 0}}}\frac{|z_1|^{k} |z_{2}|^{\ell}}{k! \ell!}V^{k+\ell} \r), 
\end{align}
where $Z = \c V^{\frac{1}{1-\s}}(\log{V})^{\frac{m+\tau(\s)}{1-\s}}$, and $\c$ is a small constant decided later.
For $|z_{1}|, |z_{2}| \leq 2^{-1}e^{-2} \c V^{\frac{\s}{1-\s}}(\log{V})^{\frac{m+\tau(\s)}{1-\s}} = 2^{-1} e^{-2} V^{-1} Z$, 
it holds that
\begin{align}
\sum_{\substack{k + \ell > Z \\ k, \ell \in \ZZ_{\geq 0}}}\frac{|z_1|^{k} |z_{2}|^{\ell}}{k! l!}V^{k+\ell}
&\leq \sum_{n > Z}\frac{1}{n!}\sum_{k = 0}^{n}\binom{n}{k}
\l( 2^{-1}e^{-2} Z \r)^{n}
= \sum_{n > Z}\frac{1}{n!}
\l( e^{-2} Z  \r)^{n}\\
\label{KPE1}
&\ll \sum_{n > Z}e^{-n}
\ll \exp\l( - \c V^{\frac{1}{1-\s}}(\log{V})^{\frac{m+\tau(\s)}{1-\s}}  \r)
\end{align}
by the Stirling formula.
On the other hand, we can write
\begin{align*}
&\int_{A_{T}}P_{m, Y}(\s+it)^{k} \ol{P_{m, Y}(\s+it)}^{\ell}dt \\
&=\int_{T}^{2T}P_{m, Y}(\s+it)^{k} \ol{P_{m, Y}(\s+it)}^{\ell}dt
-\int_{[T, 2T] \setminus A_{T}}P_{m, Y}(\s+it)^{k} \ol{P_{m, Y}(\s+it)}^{\ell}dt.
\end{align*}
Recall that $Y^{k+\ell}\leq{T}(\log{T})^{-1}$ is satisfied for $k+\ell \leq Z$ if $c_3$ is sufficiently small. 
By using the Cauchy-Schwarz inequality and throughout estimates \eqref{UBDRP1}, \eqref{ESE1}, we have
\begin{align*}
&\frac{1}{T}\int_{[T, 2T] \setminus A_{T}}P_{m, Y}(\s+it)^{k} \ol{P_{m, Y}(\s+it)}^{\ell}dt\\
&\leq \l( \frac{1}{T} \meas([T, 2T] \setminus A_{T}) \r)^{1/2}
\l( \frac{1}{T} \int_{T}^{2T} \l| P_{m, Y}(\s+it) \r|^{2(k+\ell)}dt \r)^{1/2}\\
&\ll \exp\l( -\frac{c_1}{2}V^{\frac{1}{1-\s}}(\log{V})^{\frac{m+\tau(\s)}{1- \s}} \r)
\l( \frac{C(m, \s) (k+\ell)^{1-\s}}{(\log{2(k+\ell)})^{m + \tau(\s)}} \r)^{k+\ell}\\
&\ll \exp\l( -\frac{c_1}{2}V^{\frac{1}{1-\s}}(\log{V})^{\frac{m+\tau(\s)}{1- \s}} \r)
\l( \frac{C(m, \s) Z^{1-\s}}{(\log{2Z})^{m + \tau(\s)}} \r)^{k+\ell}
\end{align*}
for $1\leq k+\ell \leq Z$.
We note that the same is true for $k=\ell=0$ by estimate \eqref{ESE1}. 
Therefore, we have
\begin{align*}
&\frac{1}{T}\sum_{\substack{k + \ell \leq Z \\ k, \ell \in \ZZ_{\geq 0}}}
\frac{z_{1}^{k} z_{2}^{\ell}}{k! \ell!}\int_{[T, 2T] \setminus A_{T}}P_{m, Y}(\s+it)^{k} \ol{P_{m, Y}(\s+it)}^{\ell}dt\\
&\ll \exp\l( -\frac{c_1}{2}V^{\frac{1}{1-\s}}(\log{V})^{\frac{m+\tau(\s)}{1-\s}} \r)
\sum_{\substack{0 \leq k + \ell \leq Z \\ k, \ell \in \ZZ_{\geq 0}}} 
\frac{1}{ k! \ell!} \l( 2^{-1} C' \c V^{\frac{1}{1-\s}}(\log{V})^{\frac{m+\tau(\s)}{1-\s}} \r)^{k +\ell}\\
&\ll \exp\l( -\frac{c_1}{2}V^{\frac{1}{1-\s}}(\log{V})^{\frac{m+\tau(\s)}{1-\s}} \r) 
\exp\l( C' \c V^{\frac{1}{1-\s}}(\log{V})^{\frac{m+\tau(\s)}{1-\s}} \r), 
\end{align*}
where $C' > C(m, \s) + 1$ is a positive constant not depending on $V$ and $\c$.
Hence, choosing $\c = c_1/4C'$, we obtain
\begin{align}
&\frac{1}{T}\sum_{\substack{k + \ell \leq Z \\ k, \ell \in \ZZ_{\geq 0}}}
\frac{z_{1}^{k} z_{2}^{\ell}}{k! \ell!}\int_{[T, 2T] \setminus A_{T}}P_{m, Y}(\s+it)^{k} \ol{P_{m, Y}(\s+it)}^{\ell}dt\\ \label{KPE2}
&\ll \exp\l( -\frac{c_1}{4}V^{\frac{1}{1-\s}}(\log{V})^{\frac{m+\tau(\s)}{1-\s}} \r)
\ll \exp\l( - \c V^{\frac{1}{1-\s}}(\log{V})^{\frac{m+\tau(\s)}{1-\s}} \r).
\end{align}
Thus, by \eqref{KPEQ1}, \eqref{KPE1}, and \eqref{KPE2}, we have
\begin{align}	\label{KPEQ2}
&\frac{1}{T}\int_{A_{T}}\exp\l( z_{1}P_{m, Y}(\s+it) + z_{2} \ol{P_{m, Y}(\s+it)} \r)dt\\
&= \frac{1}{T}\sum_{\substack{k + \ell \leq Z \\ k, \ell \in \ZZ_{\geq 0}}}\frac{z_{1}^{k} z_{2}^{\ell}}{k! \ell!}
\int_{T}^{2T}P_{m, Y}(\s+it)^{k} \ol{P_{m, Y}(\s+it)}^{\ell}dt \\
&\qquad +O\l( \exp\l( - \c V^{\frac{1}{1-\s}}(\log{V})^{\frac{m+\tau(\s)}{1-\s}} \r) \r). 
\end{align}
Applying Lemma \ref{CDR} to the integral on the right-hand side, we see that its first term is equal to
\begin{align*}
&\sum_{\substack{k + \ell \leq Z \\ k, \ell \in \ZZ_{\geq 0}}}
\EXP{ \frac{z_{1}^{k} z_{2}^{\ell}}{k! \ell!}P_{m, Y}(\s, X)^{k} \ol{P_{m, Y}(\s, X)}^{\ell} }
+ O\l( \frac{1}{T}|z_{1}|^{Z} |z_{2}|^{Z}\l(2^{m} Y\r)^{2Z} \r)\\
&= \EXP{ \exp\l( z_1P_{m, Y}(\s, X) + z_{2}\ol{P_{m, Y}(\s, X)} \r) }\\
&\qquad-\sum_{\substack{k + \ell > Z \\ k, \ell \in \ZZ_{\geq 0}}}\frac{z_{1}^{k} z_{2}^{\ell}}{k! \ell!}
\EXP{ P_{m, Y}(\s, X)^{k} \ol{P_{m, Y}(\s, X)}^{\ell} }\\
&\qquad+ O\l( \frac{1}{T}\l(V^{\frac{\s}{1-\s}}(\log{V})^{\frac{m+\tau(\s)}{1-\s}} Y\r)^{2Z} \r).
\end{align*}
Here, we chose $c_{3}$ such that $c_{3} \leq 2^{-m}$.
By using the Cauchy--Schwarz inequality and estimate \eqref{UBDRP2}, we obtain
\begin{align*}
\EXP{ P_{m, Y}(\s, X)^{k} \ol{P_{m, Y}(\s, X)}^{\ell} }
\ll \l( \frac{C(\sigma, m)(k + \ell)^{1-\s}}{(\log2(k+\ell))^{m + \tau(\s)}} \r)^{k+\ell}.
\end{align*}
By using this estimate and the Stirling formula, we find that
\begin{align*}
&\sum_{\substack{k + \ell > Z \\ k, \ell \in \ZZ_{\geq 0}}}\frac{z_{1}^{k} z_{2}^{\ell}}{k! \ell!}
\EXP{ P_{m, Y}(\s, X)^{k} \ol{P_{m, Y}(\s, X)}^{\ell} }\\
&\ll \sum_{n > Z}\frac{1}{n!}\sum_{k = 0}^{n}
\begin{pmatrix}
n\\
k
\end{pmatrix}
\l(\frac{C(m, \s) \max\{ |z_{1}| , |z_{2}| \} n^{1-\s} }{(\log{n})^{m+\tau(\s)}}\r)^{n}\\
&\ll \sum_{n > Z}
\l(\frac{2 e C(m, \s) \max\{ |z_{1}| , |z_{2}| \} }{ n^{\sigma} (\log{n})^{m+\tau(\s)}}\r)^{n}\\
&\ll \sum_{n > Z}
\l(\frac{c_3 C' V^{\frac{\s}{1-\s}}(\log{V})^{\frac{m+\s}{1-\s}}}{e n^{\s}(\log{n})^{m+\tau(\s)}}\r)^{n}
\ll \exp\l(- \c V^{\frac{1}{1-\s}}(\log{V})^{\frac{m+\tau(\s)}{1-\s}}\r).
\end{align*}
Hence, the left-hand side of \eqref{KPEQ2} is equal to
\begin{multline*}
\EXP{ \exp\l( z_1P_{m, Y}(\s, X) + z_{2}\ol{P_{m, Y}(\s, X)} \r) } \\
+ O\l( \frac{1}{T}\l(V^{\frac{\s}{1-\s}}(\log{V})^{\frac{m+\tau(\s)}{1-\s}} Y\r)^{2Z}
+ \exp\l( - \c V^{\frac{1}{1-\s}}(\log{V})^{\frac{m+\tau(\s)}{1-\s}} \r) \r), 
\end{multline*}
which completes the proof of Proposition \ref{Key_Prop}.
\end{proof}

\section{\textbf{Probability density function for $\te_m(\s, X)$}}\label{secRDS}
To begin with, we check that $\te_m(\s, X)$ defined by \eqref{eqdefeta} presents a $\mathbb{C}$-valued random variable if $(\s, m) \in \mca{A}$. 
For this, it is sufficient to show the following. 

\begin{lemma}\label{lemASConv}
Let $(\s, m) \in \mca{A}$. 
Then \eqref{eqdefeta} converges almost surely. 
\end{lemma}

\begin{proof}
For any prime number $p$, we have
\begin{gather*}
\EXP{\te_{m, p}(\s, X(p))}
= \sum_{k = 1}^\infty \frac{ \EXP{X(p)^k} }{kp^{k \sigma}(\log{p}^k)^m}
=0, 
\end{gather*}
where the change of the sum and expectation is justified by Fubini's theorem.
By the prime number theorem, we further obtain
\begin{align} \label{plemASConv1}
\sum_p \EXP{\l|\te_{m, p}(\s, X(p))\r|^2}
\ll \sum_p \frac{1}{p^{2\s} (\log p)^{2m}}
< \infty
\end{align}
since $(\s, m) \in \mca{A}$. 
Thus the assertion follows from \cite[Theorem 1.4.2]{S2011}. 
\end{proof}

Then, we prove that $\te_m(\s, X)$ has a continuous probability density function. 

\begin{proposition}\label{propPDF}
Let $(\s, m)\in\mca{A}$. 
Then there exists a continuous function $\DF_{\s, m}:\CC\to\RR_{\geq0}$ such that 
\begin{gather}\label{eqMine12}
\PP(\te_{m}(\s, X)\in{A})
=\int_{A} \DF_{\s, m}(z)\, |dz|
\end{gather}
for any Borel set $A$ on $\CC$, where $|dz|=(2\pi)^{-1}dxdy$ for $z=x+iy$ with $x,y \in \mathbb{R}$.  
\end{proposition}

\begin{proof}
The characteristic function of the random variable $\te_m(\s, X)$ is defined as
\begin{gather*}
\Lam_{\s, m}(w)
=\EXP{\exp(i \inp{\te_{m}(\s, X), w})}
\end{gather*}
for $w \in \mathbb{C}$, where the inner product $\inp{z_1, z_2}$ stands for $\Re z_1 \Re z_2 + \Im z_1 \Im z_2$ for $z_1, z_2 \in \CC$. 
By the independence of $X=\{X(p)\}_{p \in \mathcal{P}}$, we have 
\begin{align}\label{eqMine3}
\Lam_{\s, m}(w)
&=\EXP{\prod_{p} \exp(i \inp{\te_{m,p}(\s, X(p)), w})} \\
&=\prod_{p} \Lam_{\s, m, p}(w),  
\end{align}
where $\Lam_{\s, m, p}(w)=\EXP{\exp(i \inp{\te_{m, p}(\s, X(p)), w})}$. 
The inequality $|\Lam_{\s, m, p}(w)|\leq1$ holds for every $p$ by the definition. 
Therefore, we obtain
\begin{gather}\label{eqMine5}
\l|\Lam_{\s, m}(w)\r|
\leq\prod_{p\in\ms{P}} \l|\Lam_{\s, m, p}(w)\r|
\end{gather}
for any subset $\ms{P}\subset\mca{P}$. 
Write $w=u+iv$ with $u,v \in \mathbb{R}$ and $P(M)={M}|w|^{1/\s}$ for $M\geq1$. 
By the Taylor expansion of $\exp(z)$, we obtain 
\begin{align*}
&\exp(i \inp{\te_{m, p}(\s, X(p)), w})\\
&=\exp(iu\Re\te_{m, p}(\s, X(p))+iv\Im\te_{m, p}(\s, X(p)))\\
&=1+iu\Re\te_{m, p}(\s, X(p))+iv\Im\te_{m, p}(\s, X(p))\\
&\qquad+\frac{1}{2}\{iu\Re\te_{m, p}(\s, X(p))+iv\Im\te_{m, p}(\s, X(p))\}^2+O\l(\frac{(|u|+|v|)^3}{p^{3\s}(\log{p})^{3m}}\r)
\end{align*}
for $p>P(M_1)$ with some $M_1\geq1$. 
By the definition of $\te_{m, p}(\s, X(p))$, we can easily see that 
\[
\EXP{\Re\te_{m, p}(\s, X(p))}
=\EXP{\Im\te_{m, p}(\s, X(p))}=0, 
\]
\[
\EXP{\Re\te_{m, p}(\s, X(p))\Im\te_{m, p}(\s, X(p))}
=0, \]
and
\[\EXP{\l(\Re\te_{m, p}(\s, X(p))\r)^2}
=\EXP{\l(\Im\te_{m, p}(\s, X(p))\r)^2}
=\frac{1}{2}\frac{\Li_{2m+2}(p^{-2\s})}{(\log{p})^{2m}}. \]
Therefore, the characteristic function $\Lam_{\s, m, p}(w)$ is calculated as 
\[\Lam_{\s, m, p}(w)
=1-\frac{|w|^2}{4}\frac{\Li_{2m+2}(p^{-2\s})}{(\log{p})^{2m}}+O\l(\frac{|w|^3}{p^{3\s}(\log{p})^{3m}}\r).\]
Furthermore, we have 
\[\Log\Lam_{\s, m, p}(w)
=-\frac{|w|^2}{4}\frac{\Li_{2m+2}(p^{-2\s})}{(\log{p})^{2m}}+O\l(\frac{|w|^3}{p^{3\s}(\log{p})^{3m}}\r)\]
if $p>P(M_2)$ with some $M_2\geq{M}_1$, where $\Log(z)$ is the principal blanch of logarithm. 
Hence we deduce the asymptotic formula
\[\log\l|\Lam_{\s, m, p}(w)\r|
=-\frac{|w|^2}{4}\frac{\Li_{2m+2}(p^{-2\s})}{(\log{p})^{2m}}+O\l(\frac{|w|^3}{p^{3\s}(\log{p})^{3m}}\r)\]
by the equality $\log|z|=\Re\Log(z)$. 
We notice that the inequalities
\begin{gather*}
\frac{|w|^2}{4}\frac{\Li_{2m+2}(p^{-2\s})}{(\log{p})^{2m}}\geq\frac{1}{4}\frac{|w|^2}{p^{2\s}(\log{p})^{2m}}, \\
\frac{|w|^3}{p^{3\s}(\log{p})^{3m}}\leq\frac{1}{M}\frac{|w|^2}{p^{2\s}(\log{p})^{2m}} 
\end{gather*}
are satisfied for $p>P(M)$ with any $M\geq1$. 
Hence there exists an absolute constant $M_3\geq{M}_2$ such that the inequality
\[\log\l|\Lam_{\s, m, p}(w)\r|
\leq-\frac{1}{8}\frac{|w|^2}{p^{2\s}(\log{p})^{2m}} \]
holds for $p>P(M_3)$. 
Therefore, taking $\ms{P} = \set{p \in \mathcal{P}}{p > P(M_{3})}$ in \eqref{eqMine5}, we deduce from the prime number theorem that
\begin{gather}\label{eqMine9}
\l|\Lam_{\s, m}(w)\r|
\leq\exp\l(-\frac{|w|^2}{8}\sum_{p>P(M_3)}\frac{1}{p^{2\s}(\log{p})^{2m}}\r)
\leq\exp\l(-|w|^{1/(2\s)}\r)
\end{gather}
if $|w|>c(\sigma, m)$ with some large constant $c(\s, m)>0$. 
This implies that
\begin{align} \label{neqMine1}
  \int_\CC|\Lam_{\s, m}(w)|\, |dw|<\infty,  
\end{align}
and one can define a function
\begin{gather}\label{eqMine14}
  \DF_{\s, m}(z)=\int_\CC\Lam_{\s, m}(w)\exp(-i\inp{z, w})\, |dw|
\end{gather}
for $z \in \mathbb{C}$. 
The function $\DF_{\s, m}(z)$ is continuous by the dominated convergence theorem. 
By Levy's inversion formula, we see that $\DF_{\s, m}(z)$ is a probability density function of $\te_m(\s, X)$. 
Thus the proof is completed. 
\end{proof}

We further prove the following property of $\DF_{\s, m}(z)$. 
We will prove more properties according to the method of Jessen--Wintner \cite{JW1935} in Appendix section later.

\begin{proposition}\label{prop:MGF}
Let $(\s, m)\in\mca{A}$. 
Then the integral
\begin{gather*}
\int_\CC e^{a|z|} \DF_{\s, m}(z)\, |dz|
\end{gather*}
is finite for any $a>0$. 
\end{proposition}

\begin{proof}
We use Proposition \ref{propPDF} to deduce
\begin{gather*}
\int_\CC e^{a|z|}\DF_{\s, m}(z)\, |dz|
=\EXP{\exp\l(a|\te_m(\s, X)|\r)}. 
\end{gather*}
Thus it is sufficient to prove that $\EXP{\exp\l(a|\te_m(\s, X)|\r)}$ is finite for any $a \in \mathbb{R}$. 
Recall that $\sum_{p<y}\te_{m, p}(\s, X(p))\to\te_m(\s, X)$ in law as $y\to\infty$. 
Hence we deduce from Fatou's lemma that 
\begin{gather}\label{eqMine1}
\EXP{\Phi(\te_{m}(\s, X))}
\leq\liminf_{y\to\infty}
\EXP{\Phi\l(\sum_{p<y}\te_{m, p}(\s, X(p))\r)}
\end{gather}
for any continuous function $\Phi:\CC\to\RR_{\geq0}$. 
In particular, we can take the function $\Phi(z)=\exp(\pm{a}\Re{z})$. 
In this case, we have $\Phi(z+w)=\Phi(z)\Phi(w)$, and therefore the identity
\begin{gather}\label{eqMine1'}
\EXP{\Phi\l(\sum_{p<y}\te_{m, p}(\s, X(p))\r)}
=\prod_{p<y}\EXP{\Phi\l( \te_{m, p}(\s, X(p)) \r)}
\end{gather}
holds since $\{X(p)\}_{p\in\mca{P}}$ is independent. 
If we suppose $p\geq{a}^{1/\s}$, then the estimate 
\[\Phi\l( \te_{m, p}(\s, X(p)) \r)
=1\pm{a}\Re\te_{m, p}(\s, X(p))
+O\l(\frac{a^2}{p^{2\s}(\log{p})^{2m}}\r)\]
follows by the Taylor expansion. 
It implies
\[\EXP{\Phi\l( \te_{m, p}(\s, X(p)) \r)}
=1+O\l(\frac{a^2}{p^{2\s}(\log{p})^{2m}}\r) \]
since $\EXP{\Re\te_{m, p}(\s, X(p))}$ vanishes. 
Note that the series $\sum_{p}p^{-2\s}(\log{p})^{-2m}$ is finite if $(\s, m) \in \mca{A}$. 
Hence we conclude that the infinite product
\[\prod_{p}\EXP{\Phi\l( \te_{m, p}(\s, X(p)) \r)}\]
converges, and that $\EXP{\Phi(\te_{m}(\s, X))}$ is finite by \eqref{eqMine1} and \eqref{eqMine1'}. 
From the above, we deduce
\begin{align}
&\EXP{\exp(a|\Re\te_m(\s, X)|)}\nonumber\\
&\leq\EXP{\exp(a\Re\te_m(\s, X))}+\EXP{\exp(-a\Re\te_m(\s, X))}
<\infty. 
\end{align}
One can prove that $\EXP{\exp(a|\Im\te_m(\s, X)|)}$ is also finite by replacing the above function $\Phi$ with $\Phi(z)=\exp(\pm{a}\Im{z})$. 
Note that the inequality 
\begin{gather*}
\exp\l(a|\te_m(\s, X)|\r)
\leq \exp(a |\Re \te_{m}(\s, X)|) \exp(a |\Im \te_{m}(\s, X)|)
\end{gather*}
holds. 
Therefore, by the Cauchy--Schwarz inequality, we obtain
\begin{align}\label{eqMine13}
&\EXP{\exp\l(a|\te_m(\s, X)|\r)}\\
&\leq\sqrt{\EXP{\exp\l(2a|\Re\te_m(\s, X)|\r)}}
\sqrt{\EXP{\exp\l(2a|\Im\te_m(\s, X)|\r)}}<\infty
\end{align}
as desired. 
\end{proof}

Let $(\s, m) \in \mca{A}$, and denote by  $P_{m, Y}(\s, X)$ the random Dirichlet polynomial introduced in Section \ref{Sec_FMGF}. 
Then the following lemmas are used in the proofs of Theorems \ref{thmDB1} and \ref{thmLD}. 

\begin{lemma}	\label{CREP} 
Let $(\s, m) \in \mca{A}$.
For $Y \geq 3$ and $w \in \CC$, 
we have
\begin{align*}
\Lam_{\s, m}(w)
= \EXP{ \exp\l( i\inp{P_{m, Y}(\s, X), w} \r) }
+ O_{\s, m} \l( \frac{|w|}{Y^{\s-\frac{1}{2}}(\log{Y})^{m}} \r).
\end{align*}
\end{lemma}

\begin{proof} 
By the definition of $\te_{m}(\s, X)$, we see that
\begin{align}
  \hspace{-9mm}
\begin{aligned}
\label{CREP1}
\Lam_{\s, m}(w)
= \EXP{ \exp\l(i\inp{P_{m, Y}(\s, X), w} + i\inp{\sum_{p^{k} > Y}\frac{X(p)^{k}}{k p^{k\s} (\log{p^{k}})^{m}}, w}\r) }.
\end{aligned}
\end{align}
Since the estimate
\begin{align*}
\sum_{\substack{p^{k} > Y \\ k \geq 2}}\frac{X(p)^{k}}{k p^{k\s} (\log{p^{k}})^{m}}
\ll_{\s, m}\frac{1}{Y^{\s - \frac{1}{2}}(\log{Y})^{m}}
\end{align*}
holds, we have
\begin{align*}
\inp{\sum_{p^{k} > Y}\frac{X(p)^{k}}{k p^{k\s} (\log{p^{k}})^{m}}, w}
= \inp{\sum_{p > Y}\frac{X(p)}{p^{\s} (\log{p})^{m}}, w}
+ O_{\s, m}\l( \frac{|w|}{Y^{\s-\frac{1}{2}}(\log{Y})^{m}} \r)
\end{align*}
by applying the Cauchy--Schwarz inequality $|\inp{z, w}|\leq|z||w|$. 
Furthermore, by the inequality $|e^{i b} - e^{i a} | \leq | b -a |$ for $a, b \in \mathbb{R}$, 
the left-hand side of \eqref{CREP1} is equal to
\begin{align*}
\EXP{ \exp\l(i\inp{P_{m, Y}(\s, X), w} + i\inp{\sum_{p > Y}\frac{X(p)}{p^{\s} (\log{p})^{m}}, w}\r) }
+ O_{\s, m}\l( \frac{|w|}{Y^{\s - \frac{1}{2}} (\log{Y})^{m}} \r).
\end{align*}
From the independence of $X(p)$'s, the above expectation is equal to
\begin{align}	\label{eq:CREP2}
\EXP{ \exp\l(i\inp{P_{m, Y}(\s, X), w} \r)} \times \EXP {\exp\l(i\inp{\sum_{p > Y}\frac{X(p)}{p^{\s} (\log{p})^{m}}, w}\r) }.
\end{align}
Moreover, by the Cauchy--Schwarz inequality and the inequality $|e^{i x} - 1| \leq | x |$ for $x \in \mathbb{R}$, 
we find that
\begin{align*}
&\Bigg|\EXP{ \exp\l(i\inp{\sum_{Y < p \leq Z}\frac{X(p)}{p^{\s}(\log{p})^{m}}, w} \r) } - 1\Bigg|\\
&\leq  | w | \EXP{\bigg| \sum_{Y < p \leq Z}\frac{X(p)}{p^{\s}(\log{p})^{m}} \bigg| }
\leq | w | \l( \EXP{\bigg| \sum_{Y < p \leq Z}\frac{X(p)}{p^{\s}(\log{p})^{m}} \bigg|^2 } \r)^{1/2}\\
&= |w| \l( \EXP{\sum_{Y < p_{1}, p_{2} \leq Z}
\frac{X(p_{1}) \ol{X(p_{2})}}{(p_{1} p_{2})^{\s}(\log{p_{1}} \log{p_{2}})^{m}} } \r)^{1/2}
\end{align*}
for any $Z > Y$.
By equation \eqref{BEUIR}, the last is equal to
\begin{align*}
| w | \l( \sum_{Y < p \leq Z}\frac{1}{p^{2\s}(\log{p})^{2m}} \r)^{1/2}
\ll_{\s, m} \frac{|w|}{Y^{\s-\frac{1}{2}}(\log{Y})^{m}}.
\end{align*}
Therefore, by Lebesgue's dominated convergence theorem, it holds that
\begin{align*}
&\EXP{ \exp\l(i\inp{\sum_{p > Y}\frac{X(p)}{p^{\s}(\log{p})^{m}}, w} \r) }\\
&= \lim_{Z \rightarrow \infty}\EXP{ \exp\l(i\inp{\sum_{Y < p \leq Z}\frac{X(p)}{p^{\s}(\log{p})^{m}}, w} \r) }
= 1 + O_{\s, m}\l(\frac{|w|}{Y^{\s-\frac{1}{2}}(\log{Y})^{m}}\r), 
\end{align*}
and hence, \eqref{eq:CREP2} is equal to
\begin{align*}
\EXP{ \exp\l(i\inp{P_{m, Y}(\s, X), w} \r)}
+ O_{\s, m}\l( \frac{|w|}{Y^{\s-\frac{1}{2}}(\log{Y})^{m}} \r).
\end{align*}
Thus, the left-hand side of \eqref{CREP1} is also equal to above.
\end{proof}

Let $(\s, m) \in \mca{A}$. 
By Proposition \ref{prop:MGF}, we see that the moment-generating function 
\begin{align} \label{eqFY}
F_{\sigma,m}(s;\alpha)
:=\EXP{ \exp(s\Re e^{-i\a}\te_{m}(\s, X))}
\end{align}
exists for any $s\in\CC$ and $\alpha \in \mathbb{R}$. 

\begin{lemma}	\label{CREP2} 
Let $(\s, m) \in \mca{A}$.
Let $Y$ be a large parameter.
For $s=\kappa+it$ with $|s|\leq{Y}^{\sigma-\frac{1}{2}}(\log{Y})^m$, we have
\begin{align*}
F_{\sigma,m}(s;\alpha)
= \EXP{ \exp\l( s\Re {e^{-i\a} P_{m, Y}(\s, X)} \r) }
+ O_{\s, m}\l(F_{\sigma,m}(\kappa;\alpha) \frac{|s|}{Y^{\s-\frac{1}{2}}(\log{Y})^{m}} \r).
\end{align*}
\end{lemma}

\begin{proof}
Since the estimate
\begin{align*}
\sum_{\substack{p^{k} > Y \\ k \geq 2}}\frac{X(p)^{k}}{k p^{k\s} (\log{p^{k}})^{m}}
\ll_{\s, m} \frac{1}{Y^{\s - \frac{1}{2}}(\log{Y})^{m}}
\end{align*}
holds, we have
\begin{align}
&F_{\sigma,m}(s;\alpha)\\
&= \EE\Bigg[ \exp\Bigg(s\Re {e^{-i\a} P_{m, Y}(\s, X)}  
+ s \sum_{p > Y}\frac{\Re e^{-i\a}X(p)}{p^{\s} (\log{p})^{m}} \\
&\qquad\qquad\qquad\qquad\qquad\qquad\qquad\qquad\qquad
+ O_{\s, m}\l( \frac{|s|}{Y^{\s-\frac{1}{2}}(\log{Y})^{m}} \r)\Bigg) \Bigg]\\
\label{CREP1'}
&=\EXP{ \exp\l(s\Re{e^{-i\a} P_{m, Y}(\s, X)} 
+ s \sum_{p > Y}\frac{\Re e^{-i\a}X(p)}{p^{\s} (\log{p})^{m}}\r)}\\
&\qquad+ O_{\s, m} \Bigg(\EE\Bigg[{ \kappa\Re {e^{-i\a} P_{m, Y}(\s, X)}  
+ \kappa \sum_{p > Y}\frac{\Re e^{-i\a}X(p)}{p^{\s} (\log{p})^{m}}} \Bigg]
 \frac{|s|}{Y^{\s-\frac{1}{2}}(\log{Y})^{m}}  \Bigg).
\end{align}
Note that the independence of $X(p)$'s yields 
\begin{align}\label{eqMine54}
&\EXP{ \exp\l(s\Re {e^{-i\a} P_{m, Y}(\s, X)} 
+ s \sum_{p > Y}\frac{\Re e^{-i\a}X(p)}{p^{\s} (\log{p})^{m}}\r)}\nonumber\\\
&=\EXP{ \exp\l(s\Re {e^{-i\a} P_{m, Y}(\s, X)} \r)}
\times \prod_{p > Y}\EXP{\exp\l(s \frac{\Re e^{-i\a}X(p)}{p^{\s} (\log{p})^{m}}\r)}. 
\end{align}
Furthermore, if $p>Y$, we find that the inequality
\[\l|s\frac{\Re e^{-i\a}X(p)}{p^{\s} (\log{p})^{m}}\r|
\leq\frac{|s|}{p^\s(\log{p})^m} \leq c \]
holds for $|s|\leq{Y}^{\sigma-\frac{1}{2}}(\log{Y})^m$, where $c$ is suitably small if $Y$ is large.
From the Taylor expansion of $\exp(z)$ we deduce
\begin{align*}
\EXP{\exp\l(s \frac{\Re e^{-i\a}X(p)}{p^{\s} (\log{p})^{m}}\r)}
&=\EXP{1
+s \frac{\Re e^{-i\a}X(p)}{p^{\s} (\log{p})^{m}}
+O\l(\frac{|s|^2}{p^{2\s}(\log{p})^{2m}}\r)}\\
&=1+O\l(\frac{|s|^2}{p^{2\s}(\log{p})^{2m}}\r) 
\end{align*}
since the expected value $\EXP{\Re e^{-i\a}{X}(p)}$ vanishes. 
Therefore, we find that
\begin{align*}
  &\prod_{p > Y}\EXP{\exp\l(s \frac{\Re e^{-i\a}X(p)}{p^{\s} (\log{p})^{m}}\r)}
  = \prod_{p > Y}\l(1 + O\l( \frac{|s|^2}{p^{2\s}(\log{p})^{2m}} \r)\r)\\
  &= \exp\l( O\l(\sum_{p > Y}\frac{|s|^2}{p^{2\s}(\log{p})^{2m}}\r) \r)
  = \exp\l( O_{\s, m}\l(\frac{|s|^2}{Y^{2\s - 1} (\log{Y})^{2m}} \r) \r)\\
  &= 1 + O_{\s, m}\l( \frac{|s|}{Y^{\s - \frac{1}{2}} (\log{Y})^{2m}} \r).
\end{align*}
Hence, by equation \eqref{eqMine54}, the formula
\begin{align*}
&\EXP{ \exp\l(s\Re {e^{-i\a} P_{m, Y}(\s, X)}
+ s \sum_{p > Y}\frac{\Re e^{-i\a}X(p)}{p^{\s} (\log{p})^{m}}\r)}\\
&=\EXP{ \exp\l(s\Re {e^{-i\a} P_{m, Y}(\s, X)} \r)}\\
&\quad+O_{\s, m}\l(\EXP{ \exp\l(\kappa\Re {e^{-i\a} P_{m, Y}(\s, X)} \r)}
\frac{|s|}{Y^{\s-\frac{1}{2}}(\log{Y})^{m}}\r) 
\end{align*}
holds for $|s|\leq{Y}^{\sigma-\frac{1}{2}}(\log{Y})^m$. 
Inserting this to \eqref{CREP1'}, we finally obtain 
\begin{align*}
&F_{\sigma,m}(s;\alpha) \\
&=\EXP{ \exp\l(s\Re {e^{-i\a} P_{m, Y}(\s, X)} \r)}
\l(1+O_{\s, m}\l(\frac{|s|}{Y^{\s-\frac{1}{2}}(\log{Y})^{m}}\r)\r), 
\end{align*}
which yields the result. 
\end{proof}

\section{\textbf{An estimation of the difference between \\$\te_{m}(s)$ and Dirichlet polynomials}} \label{Sec_ED_ZvDP}

\begin{lemma} \label{MVEtevDP}
  Let $\s > \frac{1}{2}$.
  There exist positive constants $c$, $A = A(\s)$ such that for any $k \in \ZZ_{\geq 1}$, $3 \leq Y \leq T^{1/k}$
  \begin{align*}
    \frac{1}{T}\int_{T}^{2T}\l| \log{\zeta(\s + it)} - P_{0, Y}(\s + it) \r|^{2k}dt
    \leq A^{k} k^{4k} T^{c(1 - 2\s)} + A^{k} k^{k} Y^{k(1 -2\s)}.
  \end{align*}
\end{lemma}

\begin{proof}
  It follows from Proposition 3.3 of \cite{IL2021} that
  \begin{align*}
    &\frac{1}{T}\int_{T}^{2T}\l| \log{\zeta(\s + it)} - P_{0, Y}(\s + it) \r|^{2k}dt\\
    &\leq C^{k} k^{4k} T^{c(1 - 2\s)} + C^{k} k^{k} \l( \sum_{Y < p \leq T^{1/k}}\frac{1}{p^{2\s}} \r)^{k}
  \end{align*}
  for some absolute constants $c, C > 0$.
  The sum on the right hand side is $\ll_{\s} Y^{k(1 - 2\s)}$, so we obtain this lemma.
\end{proof}

\begin{lemma}	\label{ESEPE}
  Let $(\s, m) \in \mca{A}$. 
  Let $T>0$ be large, and let $3 \leq Y \leq T^{c_{1}}$ with $c_{1} = c_{1}(\s) > 0$ suitably small. 
  For $W > 0$, we denote by $B_{T} = B_{T}(Y, W; \sigma, m)$ the set
  \begin{align}	\label{def_B_Y}
    B_{T} = \Set{t \in [T, 2T]}{|\te_{m}(\s + it) - P_{m, Y}(\s+it)| \leq W Y^{\frac{1}{2} - \s}}.
  \end{align}
  When $m \not= 0$, there exists a small positive constant $c$ such that
  \begin{align}	\label{ESEPE1}
    \frac{1}{T}\meas([T, 2T] \setminus B_{T}) 
    \ll \exp\l( -c W^{2} (\log{Y})^{2m} \r)
  \end{align}
  for $0 < W \leq \l((\log{T})(\log{Y})^{-2(m+1)}\r)^{\frac{m}{2m+1}}$. 
  Moreover, for $m \in \ZZ_{\geq 0}$ we have 
  \begin{align}	\label{ESEPE2}
    \frac{1}{T}\meas([T, 2T] \setminus B_{T})
    \ll \exp\l( -c (W(\log{T})^{m})^{\frac{1}{m+1}} \r)
  \end{align}
  for $\l((\log{T})(\log{Y})^{-2(m+1)}\r)^{\frac{m}{2m+1}} \leq W \leq (\log{T})(\log{Y})^{-(m+1)}$. 
\end{lemma}
  
\begin{proof}
  When $m$ is positive, this lemma follows from Lemma 8 of \cite{II2020}. 
  We prove \eqref{ESEPE2} in the case $m = 0$.
  Let $1 \leq W \leq \frac{\log{T}}{\log{Y}}$.
  Using Lemma \ref{MVEtevDP}, we obtain
  \begin{align*}
    \frac{1}{T}\meas([T, 2T] \setminus B_{T})
    \ll \frac{Y^{k(2\s - 1)}}{T^{c_{2}(2\s - 1)}}\l( \frac{A k^{2}}{W} \r)^{2k} + \l( \frac{A k}{W^2} \r)^{k}.
  \end{align*}
  for any $1 \leq k \leq \frac{\log{T}}{\log{Y}}$ and for some constant $A > 0, c_{2} > 0$.
  We decide as $c_{1} = c_{2}/2$.
  Choosing $k = \l[c \frac{W}{\log(W + 2)} \r] + 1$ with $c = c(\s)$ a suitably small constant, we have
  \begin{align*}
    \frac{1}{T}\meas([T, 2T] \setminus B_{T})
    \ll \exp\l( -c W \r), 
  \end{align*}
  which completes the proof of this lemma in the case $m = 0$.
\end{proof}

\section{\textbf{Discrepancy estimate: Proof of Theorem \ref{thmDB1}}}  \label{Sec_PDB}

We prepare two mean value theorems for $\te_{m}(s)$ and $\te_{m}(\s, X)$.

\begin{lemma} \label{MVTeta}
  Let $(\s, m) \in \mca{A}$.
  Let $T$ be large.
  Let $k$ be a positive integer with $k \leq \sqrt{\log{T}}$.
	Then, we have
  \begin{align*}
    \frac{1}{T}\int_{T}^{2T}| \te_{m}(\s + it)|^{2k}dt
    \leq k! C^{k}
  \end{align*}
  for some constant $C = C(\s, m) > 0$.
\end{lemma}

\begin{proof}
  Let $m \in \ZZ_{\geq 1}$, $\s \geq \frac{1}{2}$ be fixed.
  Let $k$ be a positive integer.
	Let $T$ be large, $3 \leq Y \leq T^{1 / 135k}$.
  By Lemma 8 in \cite{II2020}, we have
  \begin{align}
    \label{pMVTBEDP1}
    &\frac{1}{T}\int_{T}^{2T}\bigg| \te_{m}(\s + it) - \sum_{2 \leq n \leq Y}\frac{\Lam(n)}{n^{\s + it} (\log{n})^{m+1}} \bigg|^{2k}dt\\
    &\ll C^{k} k! \frac{Y^{k(1 - 2\s)}}{(\log{Y})^{2km}} + C^{k} k^{2k(m+1)}\frac{T^{\frac{1 - 2\s}{135}}}{(\log{T})^{2km}}.
  \end{align}
  Therefore, we obtain
  \begin{align*}
    &\frac{1}{T}\int_{T}^{2T}| \te_{m}(\s + it)|^{2k}dt\\
    &\leq 2^{2k}\frac{1}{T}\int_{T}^{2T}\bigg| \te_{m}(\s + it) - \sum_{2 \leq n \leq Y}\frac{\Lam(n)}{n^{\s + it} (\log{n})^{m+1}} \bigg|^{2k}dt\\
    &\qquadf \qquadf+ 2^{2k} \frac{1}{T}\int_{T}^{2T} \bigg| \sum_{2 \leq n \leq Y}\frac{\Lam(n)}{n^{\s + it} (\log{n})^{m+1}} \bigg|^{2k}dt.
  \end{align*}
  By Lemma \ref{VSLS}, the second integral is 
  \begin{align*}
    \leq \int_{T}^{2T}\bigg|\sum_{p \leq Y}\frac{1}{p^{\s + it} (\log{p})^{m + 1}}\bigg|^{2k}dt + O(T C^{k})
    \leq T k! C^{k}
  \end{align*}
  for some constant $C = C(\s, m) > 0$.
  By inequality \eqref{pMVTBEDP1}, the first integral is $\leq T (k! C^{k} + k^{2k(m + 1)}/(\log{T})^{2km})$ for some constant $C = C(\s, m) > 0$.
  Hence, we obtain
  \begin{align*}
    \frac{1}{T}\int_{T}^{2T}| \te_{m}(\s + it)|^{2k}dt
    \leq k! C^{k} + C^{k} \frac{k^{2k(m + 1)}}{(\log{T})^{2km}}
    \leq k! C^{k},
  \end{align*}
  which completes the proof.
\end{proof}

\begin{lemma} \label{MVERE}
  Let $(\s, m) \in \mca{A}$. For any $k \in \ZZ_{\geq 1}$, we have
  \begin{align*}
    \EXP{|\te_{m}(\s, X)|^{2k}}
    \leq k! C^{k}
  \end{align*}
  for some constant $C = C(\s, m)$.
\end{lemma}

\begin{proof}
  Using Lemma \ref{VSLS}, we see that
  \begin{align*}
    \EXP{\l| \sum_{p \leq Y}\frac{\Li_{m + 1}(p^{-\s} X(p))}{(\log{p})^{m}} \r|^{2k}}
    \leq \EXP{2^{2k}\l| \sum_{p \leq Y}\frac{X(p)}{p^{\s} (\log{p})^{m}}\r|^{2k}}  + O( C^{k} )
    \leq k! C^{k}
  \end{align*}
  for some $C = C(\s, m)$. 
  Hence, we see, using Fatou's lemma, that
  \begin{align*}
    \EXP{|\te_{m}(\s, X)|^{2k}}
    \leq \liminf_{Y \rightarrow + \infty}\EXP{\l| \sum_{p \leq Y}\frac{\Li_{m + 1}(p^{-\s} X(p))}{(\log{p})^{m}} \r|^{2k}}
    \leq k! C^{k},
  \end{align*}
  which is the assertion of this lemma.
\end{proof}

Let $\mathcal{S}$ denote the set of all rectangles in the complex plane $\mathbb{C}$ with sides parallel to the coordinate axes. 
Put $\L = \log{\log{T}}$.
We define $\mathcal{S}_{\L} \subset \mathcal{S}$ as the collection of all $\mathcal{R}$ such that 
\begin{gather*}
\mathcal{R} 
\subset \mathcal{R}_{\L} 
:=[-\L, +\L] \times i[-\L, +\L]
\end{gather*}
for large $T$. 
Let $1/2<\sigma<1$ and $m \in \mathbb{Z}_{\geq0}$. 
If $\mathcal{R} \notin \mathcal{S}_{\L}$, then it holds that 
\begin{align}
&\PP_T(\te_{m}(\s+it)\in\mca{R}) \\
&=\PP_T(\te_{m}(\s+it)\in\mca{R}\cap \mathcal{R}_{\L})
+\PP_T(\te_{m}(\s+it)\in\mca{R} \setminus \mathcal{R}_{\L}). 
\end{align}
Using Lemma \ref{MVTeta} with $k = [\L]$, we have 
\begin{align*}
\PP_T(\te_{m}(\s+it)\in\mca{R} \setminus \mathcal{R}_{\L})
&\leq \sum_{\alpha \in \{0,\frac{\pi}{2},\pi,\frac{3\pi}{2}\}} \PP_T( \Re e^{-i \alpha} \te_{m}(\s+it)> \L) \\
&\ll (\log{T})^{-2}. 
\end{align*}
Thus, we obtain
\begin{gather*}
\PP_T(\te_{m}(\s+it)\in\mca{R})
=\PP_T(\te_{m}(\s+it)\in\mca{R}\cap \mathcal{R}_{\L})
+O\left((\log{T})^{-2}\right). 
\end{gather*}
Applying Lemma \ref{MVERE} with $k = [\L]$, 
we find that $\PP(\te_{m}(\s,X)\in\mca{R})$ is estimated as
\begin{align*}
\PP(\te_{m}(\s,X)\in\mca{R})
=\PP(\te_{m}(\s,X)\in\mca{R}\cap \mathcal{R}_{\L})
+O\left((\log{T})^{-2}\right). 
\end{align*}
From these, we deduce that
\begin{align*}
&\sup_{\mca{R} \in \mca{S}}
\l|\PP_T(\te_{m}(\s+it)\in\mca{R})-\PP(\te_{m}(\s, X)\in\mca{R})\r| \\
&= \sup_{\mca{R} \in \mca{S}}
\l|\PP_T(\te_{m}(\s+it)\in\mca{R} \cap \mca{R}_{\L})-\PP(\te_{m}(\s, X)\in\mca{R} \cap \mca{R}_{\L}) + O\l( (\log{T})^{-2} \r)\r|\\
&= \sup_{\mca{R} \in \mathcal{S}_{\L}}
\l|\PP_T(\te_{m}(\s+it)\in\mca{R})-\PP(\te_{m}(\s, X)\in\mca{R})\r|
+O\l((\log{T})^{-2}\r)
\end{align*}
by noting that $\mca{R}\cap \mathcal{R}_{\L}$ is a rectangle in $\mathcal{S}_{\L}$. 
Hence the quantity $D_{\sigma,m}(T)$ of \eqref{eq:DB} satisfies
\begin{align}\label{eq:1449}
D_{\sigma,m}(T)
&= \sup_{\mca{R} \in \mathcal{S}_{\L}}
\l|\PP_T(\te_{m}(\s+it)\in\mca{R})-\PP(\te_{m}(\s, X)\in\mca{R})\r| \nonumber \\
&\qquad+O\left((\log{T})^{-2}\right). 
\end{align}
Furthermore, we associate $D_{\sigma,m}(T)$ with the quantity
\begin{gather}\label{eq:1729}
D_{\sigma,m,Y}(T)
=\sup_{\mca{R} \in \mathcal{S}_{2\L}}
\l|\PP_T(P_{m,Y}(\s+it)\in\mca{R})-\PP(\te_{m}(\s, X)\in\mca{R})\r|
\end{gather}
as follows. 
Let $B_T=B_T(Y, W; \sigma, m)$ denote the set of \eqref{def_B_Y} and put $\epsilon=W Y^{\frac{1}{2} - \s}$. 
For any $\mca{R}=(c_1,d_1)\times i(c_2,d_2) \in \mathcal{S}_{\L}$, we have
\begin{align} \label{pthmDB1}
&\PP_T(\te_{m}(\s+it)\in\mca{R}) \\
&= \frac{1}{T} \meas \left\{ t \in B_T ~\middle|~ \te_{m}(\s+it)\in\mca{R} \right\}
+O\left(\frac{1}{T} \meas ([T, 2T] \setminus B_T)\right) \\
&\leq \PP_T(P_{m,Y}(\s+it)\in\mca{R}^{+\epsilon})
+O\left(\frac{1}{T} \meas ([T, 2T] \setminus B_T)\right)
\end{align}
by the definition of $B_T$, where $\mca{R}^{+\epsilon}=(c_1-\epsilon,d_1+\epsilon)\times i(c_2-\epsilon,d_2+\epsilon)$. 
We also obtain the opposite inequality
\begin{align} \label{pthmDB2}
&\PP_T(\te_{m}(\s+it)\in\mca{R}) \\
&\geq \PP_T(P_{m,Y}(\s+it)\in\mca{R}^{-\epsilon})
+O\left(\frac{1}{T} \meas ([T, 2T] \setminus B_T)\right)
\end{align}
if we use the rectangle $\mca{R}^{-\epsilon}=(c_1+\epsilon,d_1-\epsilon)\times i(c_2+\epsilon,d_2-\epsilon)$. 
If $\epsilon$ is small, then the rectangles $\mca{R}^{\pm \epsilon}$ belong to $\mathcal{S}_{2\L}$. 
Therefore
\begin{gather}
  \label{pthmDB3}
\PP_T(P_{m,Y}(\s+it)\in\mca{R}^{\pm \epsilon})
=\PP(\te_{m}(\s, X)\in\mca{R}^{\pm \epsilon})
+O\left(D_{\sigma,m,Y}(T)\right)
\end{gather}
follows. 
Furthermore, we use the boundedness of the density function $\DF_{\sigma, m}(z)$ of \eqref{eqMine12} 
coming from \eqref{neqMine1}, \eqref{eqMine14} to derive 
\begin{gather}
  \label{pthmDB4}
\PP(\te_{m}(\s, X)\in\mca{R}^{\pm \epsilon})
=\PP(\te_{m}(\s, X)\in\mca{R})
+O\left(\epsilon \L \right)
\end{gather}
by noting that $\meas(\mca{R}^{+\epsilon} \setminus \mca{R}), \meas(\mca{R} \setminus \mca{R}^{-\epsilon}) \ll \epsilon \L$ 
for $\mca{R} \in \mathcal{S}_{\L}$. 
From \eqref{pthmDB1}, \eqref{pthmDB2}, \eqref{pthmDB3}, and \eqref{pthmDB4}, we deduce the formula
\begin{align*}
&\PP_T(\te_{m}(\s+it)\in\mca{R}) \\
&=\PP(\te_{m}(\s, X)\in\mca{R})
+O\left(
\frac{1}{T} \meas ([T, 2T] \setminus B_T) 
+ D_{\sigma,m,Y}(T)
+\epsilon \L \right) 
\end{align*}
for any $\mca{R} \in \mathcal{S}_{\L}$. 
Together with \eqref{eq:1449}, we conclude that $D_{\sigma,m}(T)$ is bounded as
\begin{gather}\label{eq:1721}
D_{\sigma,m}(T)
\ll 
D_{\sigma,m,Y}(T)
+\frac{1}{T} \meas ([T, 2T] \setminus B_T) 
+\epsilon \L
+(\log{T})^{-2}
\end{gather}
by using the quantity $D_{\sigma,m,Y}(T)$. 
Then, we evaluate $D_{\sigma,m,Y}(T)$ by using the Beurling--Selberg function as in \cite[Section 6]{LLR2019}. 
Define 
\begin{gather}
  \label{def_GKf}
  \begin{aligned}
G(u) = \frac{2u}{\pi} + \frac{2(1 - u)u}{\tan(\pi u)}, 
\qquad
K(x)=\left(\frac{\sin \pi x}{\pi x}\right)^2, \\
\text{and}\quad
f_{c, d}(u) = \frac{1}{2}(e^{-2\pi i u c} - e^{-2\pi i u d})
  \end{aligned}
\end{gather}
with $c, d \in \RR$. For a set $A$, we denote the indicator function of $A$ by $\bm{1}_{A}$.

\begin{lemma}	\label{BSF2d}
Let $L > 0$. Let $c, d \in \RR$ with $c < d$.
For any $x \in \RR$, we have
\begin{align*}
\bm{1}_{(c, d)}(x)
= \Im \int_{0}^{L}G\l(\frac{u}{L}\r)e^{2\pi i u x}f_{c, d}(u)\frac{du}{u}
+O\l( K(L(x-c)) + K(L(x-d)) \r).
\end{align*}
\end{lemma}

\begin{proof}
This lemma is equation (6.1) in \cite{LLR2019}, which is proved essentially in \cite{KTDT}.
\end{proof}

Let $\mathcal{R}=(c_1,d_1)\times i(c_2,d_2) \in \mathcal{S}$. 
The indicator function of the rectangle $\mathcal{R}$ is given by 
$\bm{1}_{\mathcal{R}}(z)=\bm{1}_{(c_1, d_1)}(x)\bm{1}_{(c_2, d_2)}(y)$ for $z=x+iy$. 
Thus, we derive the formula
\begin{align}\label{eq:1653}
\bm{1}_{\mathcal{R}}(z)
=W_{L, \mathcal{R}}(z)
&+O \big( K(L(x-c_1)) + K(L(x-d_1)) \nonumber \\
&\qquad \qquad 
+ K(L(y-c_2)) + K(L(y-d_2)) \big) 
\end{align}
by Lemma \ref{BSF2d}, where $W_{L, \mathcal{R}}(z)$ is denoted by
\begin{multline*}
\frac{1}{2}\Re\int_{0}^{L}\int_{0}^{L}G\l(\frac{u}{L}\r)G\l( \frac{v}{L} \r)\\
\times \l( e^{2\pi i (ux - vy)}f_{c_{1}, d_{1}}(u)\ol{f_{c_{2}, d_{2}}(v)} 
- e^{2\pi i (ux + vy)} f_{c_{1}, d_{1}}(u) f_{c_{2}, d_{2}}(v) \r)\frac{du}{u}\frac{dv}{v}.
\end{multline*}
In this deformation, we used the identity $\Im z_{1} \Im z_{2} = \frac{1}{2}\Re(z_{1} \ol{z_{2}} - z_{1} z_{2})$.
Note that $\PP_T(P_{m,Y}(\s+it)\in\mca{R})$ and $\PP(\te_{m}(\s, X)\in\mca{R})$ are represented as 
\begin{align*}
\PP_T(P_{m,Y}(\s+it)\in\mca{R})
&=\frac{1}{T} \int_{T}^{2T} \bm{1}_{\mathcal{R}}(P_{m,Y}(\s+it)) \,dt, \\
\PP(\te_{m}(\s, X)\in\mca{R})
&=\EXP{\bm{1}_{\mathcal{R}}(\te_m(\sigma,X))}. 
\end{align*}
The contributions from terms on the function $W_{L, \mathcal{R}}$ are evaluated as follows. 

\begin{proposition}\label{prop:W}
Let $1/2<\sigma<1$ and $m \in \mathbb{Z}_{\geq0}$.  
For large $T>0$, we take the parameters
\begin{gather*}
L=c_1(\log{T})^\sigma (\log\log{T})^m
\quad\text{and}\quad
Y = (\log{T})^{5/(\s - \frac{1}{2})}, 
\end{gather*}
where $c_1=c_1(\sigma,m)$ is a small positive constant. 
Then we obtain the asymptotic formula
\begin{gather*}
\frac{1}{T} \int_{T}^{2T} W_{L, \mathcal{R}}(P_{m,Y}(\s+it)) \,dt
=\EXP{W_{L, \mathcal{R}}(\te_{m}(\s, X))}
+O\left((\log{T})^{-2}\right) 
\end{gather*}
for any $\mathcal{R}=(c_1,d_1)\times i(c_2,d_2) \in \mathcal{S}_{2\L}$. 
\end{proposition}

\begin{proof}
By the definition of $W_{L, \mathcal{R}}(z)$, we have 
\begin{align*}
&\frac{1}{T} \int_{T}^{2T} W_{L, \mathcal{R}}(P_{m,Y}(\s+it)) \,dt \\
&=\frac{1}{2}\Re\int_{0}^{L}\int_{0}^{L}G\l(\frac{u}{L}\r)G\l( \frac{v}{L} \r)
\Big( \Lambda_{T,\sigma,m}(2\pi \overline{w}) f_{c_{1}, d_{1}}(u)\ol{f_{c_{2}, d_{2}}(v)} \\
&\qquad\qquad\qquad\qquad\qquad\qquad
- \Lambda_{T,\sigma,m}(2\pi w) f_{c_{1}, d_{1}}(u) f_{c_{2}, d_{2}}(v) \Big)\frac{du}{u}\frac{dv}{v}
\end{align*}
for $w=u+iv \in \mathbb{C}$, where
\begin{align}\label{eq:1601}
\Lambda_{T,\sigma, m}(w; Y)
&= \frac{1}{T} \int_{T}^{2T}\exp\l( i \inp{P_{m, Y}(\s + it), w} \r)\\
&=\frac{1}{T} \int_{T}^{2T} \exp\l( \frac{i \overline{w}}{2} P_{m, Y}(\s+it) + \frac{iw}{2} \ol{P_{m, Y}(\s+it)} \r) \,dt. 
\end{align}
Then we denote by $A_T=A_T(V, Y; \sigma, m)$ the set of \eqref{def_A_Y} with $V>0$ given by
\begin{gather*}
V = c_2(\log{T})^{1-\s}(\log{\log{T}})^{-(m+1)}, 
\end{gather*}
where $0 < c_{2} = c_{2}(\sigma, m) \leq \frac{1}{2}$ is a fixed small constant such that inequality \eqref{eqY} holds. 
Since the integrand of \eqref{eq:1601} is bounded, we find, using Lemma \ref{ESE}, that
\begin{align*}
\Lambda_{T, \sigma, m}(2\pi w; Y)
&=\frac{1}{T}\int_{A_T}\exp\l( i \pi \overline{w} P_{m, Y}(\s+it) + i \pi w \ol{P_{m, Y}(\s+it)} \r) \,dt \\
&\qquad+O\left(\exp\l( -c_3 V^{\frac{1}{1-\s}}(\log{V})^{\frac{m+\s}{1-\s}} \r) \right),
\end{align*}
where $c_3=c_3(\sigma,m)$ is a positive constant. 
Note that the choices of the parameters yield that the condition 
\begin{gather}  \label{pprop:W1}
|i \pi \overline{w}|, |i \pi w|
\leq b_{1} V^{\frac{\s}{1-\s}}(\log{V})^{\frac{m+\s}{1-\s}}
\end{gather}
is satisfied for $|u|, |v|<L$, where $b_1$ is the positive constant of Proposition \ref{Key_Prop}. 
Here, we need to choose $c_{1}$ as small such that inequality \eqref{pprop:W1} holds.
Then we deduce from Proposition \ref{Key_Prop} and Lemma \ref{CREP} that $\Lambda_{T,\sigma,m}(2\pi w)$ is estimated as
\begin{align}\label{eq:1745}
\Lambda_{T, \sigma, m}(2\pi w; Y)
=\Lambda_{\sigma, m}(2\pi w)
+E_1
\end{align}
for $|u|<L$ and $|v|<L$, where $\Lambda_{\sigma,m}(w)$ is the characteristic function of  \eqref{eqMine3}, and the error term is evaluated as
\begin{align*}
E_1
&\ll\exp\l( -c_{3} V^{\frac{1}{1-\s}}(\log{V})^{\frac{m+\s}{1-\s}} \r)
+\frac{1}{T}\l( V^{\frac{\s}{1-\s}}(\log{V})^{\frac{m + \s}{1-\s}}Y \r)^{V^{\frac{1}{1-\s}}(\log{V})^{\frac{m+\s}{1-\s}}} \\
&\qquad+\exp\l( -b_{2} V^{\frac{1}{1-\s}}(\log{V})^{\frac{m+\s}{1-\s}} \r)
+\frac{L}{Y^{\s-\frac{1}{2}}(\log{Y})^{m}} \\
&\ll (\log{T})^{-4}
\end{align*}
with the positive constant $b_2=b_2(\sigma,m)$ of Proposition \ref{Key_Prop}. 
One can estimate the function $\Lambda_{T,\sigma,m}(2\pi \overline{w})$ along the same line. 
Therefore, we obtain 
\begin{align*}
&\frac{1}{T} \int_{T}^{2T} W_{L, \mathcal{R}}(P_{m,Y}(\s+it)) \,dt \\
&=\frac{1}{2}\Re\int_{0}^{L}\int_{0}^{L}G\l(\frac{u}{L}\r)G\l( \frac{v}{L} \r)
\Big( \Lambda_{\sigma,m}(2\pi \overline{w}) f_{c_{1}, d_{1}}(u)\ol{f_{c_{2}, d_{2}}(v)} \\
&\qquad\qquad\qquad\qquad\qquad\qquad
- \Lambda_{\sigma,m}(2\pi w) f_{c_{1}, d_{1}}(u) f_{c_{2}, d_{2}}(v) \Big)\frac{du}{u}\frac{dv}{v}
+E_2 \\
&=\EXP{W_{L, \mathcal{R}}(\te_{m}(\s, X))}+E_2, 
\end{align*}
where $E_2$ is evaluated as 
\begin{gather*}
E_2
\ll L^2 (\log{T})^{-4} (d_1-c_1) (d_2-c_2)
\end{gather*}
by the basic inequalities $G(u) \ll 1$ and $|f_{c, d}(u)| \leq \pi u |d-c|$. 
Now, we assume that the rectangle $(c_{1}, d_{1}) \times i(c_{2}, d_{2}) \in \mca{S}_{2\L}$, 
so $E_{2} \ll (\log{T})^{-2}$, which completes the proof of this proposition.
\end{proof}

\begin{proof}[Proof of Theorem \ref{thmDB1}]
Choosing $Y = (\log{T})^{5/(\s - \frac{1}{2})}$ and $W=(\log{T})^{1/2}$ in \eqref{eq:1721}, we derive the bound
\begin{gather*}
D_{\sigma,m}(T)
\ll D_{\sigma,m,Y}(T)+(\log{T})^{-2}
\end{gather*}
by Lemma \ref{ESEPE}. 
Thus, we show an upper bound of the quantity $D_{\sigma,m,Y}(T)$. 
Let $\mathcal{R}=(c_1,d_1) \times i(c_2,d_2) \in \mathcal{S}_{2\L}$. 
Then we apply \eqref{eq:1653} to approximate the value $\PP_T(P_{m,Y}(\s+it)\in\mca{R})$ as 
\begin{align}\label{eq:1759}
\PP_T(P_{m,Y}(\s+it)\in\mca{R})
&=\frac{1}{T} \int_{T}^{2T} W_{L, \mathcal{R}}(P_{m,Y}(\s+it)) \,dt \nonumber\\
&\qquad+O\left(I_T(c_1)+I_T(d_1)+J_T(c_2)+J_T(d_2)\right), 
\end{align}
where we put
\begin{align*}
I_T(\xi)
&=\frac{1}{T} \int_{T}^{2T} K(L(\Re P_{m,Y}(\s+it)-\xi)) \,dt, \\
J_T(\xi)
&=\frac{1}{T} \int_{T}^{2T} K(L(\Im P_{m,Y}(\s+it)-\xi)) \,dt
\end{align*}
for $\xi \in \mathbb{R}$. 
We evaluate these integrals by using the equality 
\begin{gather}\label{eq:1805}
K(Lx)
=\frac{2}{L^2} \Re \int_{0}^{L} (L-u) e^{2\pi i x u} \,du
\end{gather}
which follows from \cite[$(6.10)$]{LLR2019}. 
We have 
\begin{align*}
I_{T}(\xi)
&=\frac{2}{L^2} \Re \int_{0}^{L} (L-u) \frac{1}{T} \int_{T}^{2T} \exp\l(2\pi i (\Re P_{m,Y}(\s+it)-\xi) u\r) \,dt\,du \\
&=\frac{2}{L^2} \Re \int_{0}^{L} (L-u) e^{-2\pi i \xi u} \Lambda_{T, \sigma, m}(2\pi u; Y)  \,du \\
&=\frac{2}{L^2} \Re \int_{0}^{L} (L-u) e^{-2\pi i \xi u} \Lambda_{\sigma,m}(2\pi u)  \,du
+O\left((\log{T})^{-4}\right) 
\end{align*}
by using \eqref{eq:1745}. 
Similarly, the formula
\begin{gather*}
J_{T}(\xi)
=\frac{2}{L^2} \Re \int_{0}^{L} (L-u) e^{-2\pi i \xi v} \Lambda_{\sigma,m}(2\pi iv)  \,dv
+O\left((\log{T})^{-4}\right) 
\end{gather*}
is valid. 
By inequality \eqref{eqMine9}, we have 
\begin{gather*}
\Lambda_{\sigma,m}(2\pi u) \ll \exp(-\sqrt{u})
\quad\text{and}\quad
\Lambda_{\sigma,m}(2\pi iv) \ll \exp(-\sqrt{v})
\end{gather*}
for large $u,v>0$. 
These yield the bounds $I_{T}(\xi), J_{T}(\xi) \ll L^{-1}+(\log{T})^{-4}$ for any $\xi \in \RR$. 
Hence we deduce from \eqref{eq:1759} the asymptotic formula
\begin{align}\label{eq:1800}
\PP_T(P_{m,Y}(\s+it)\in\mca{R})
&=\frac{1}{T} \int_{T}^{2T} W_{L, \mathcal{R}}(P_{m,Y}(\s+it)) \,dt \nonumber\\
&\qquad+O\left(L^{-1}+(\log{T})^{-4}\right). 
\end{align}
Then, we also apply \eqref{eq:1653} to approximate the value $\PP(\te_m(\s,X)\in\mca{R})$ as 
\begin{align*}
\PP(\te_m(\s,X)\in\mca{R})
&=\EXP{W_{L, \mathcal{R}}(\te_m(\s,X))} \nonumber\\
&\qquad+O\left(I(c_1)+I(d_1)+J(c_2)+J(d_2)\right), 
\end{align*}
where we put
\begin{align*}
I(\xi)
&=\EXP{K(L(\Re \te_m(\s,X)-\xi))}, \\
J(\xi)
&=\EXP{K(L(\Im \te_m(\s,X)-\xi))}
\end{align*}
for $\xi \in \mathbb{R}$. 
Since we have 
\begin{align*}
I(\xi)
=\frac{2}{L^2} \Re \int_{0}^{L} (L-u) e^{-2\pi i \xi u} \Lambda_{\sigma,m}(2\pi u)  \,du
\end{align*}
by using \eqref{eq:1805}, the bound $I(\xi) \ll L^{-1}+(\log{T})^{-4}$ follows again from inequality \eqref{eqMine9}. 
We see that $J(\xi)$ is estimated similarly. 
Hence we obtain 
\begin{align}\label{eq:1801}
\PP(\te_m(\s,X)\in\mca{R})
=\EXP{W_{L, \mathcal{R}}(\te_m(\s,X))}
+O\left(L^{-1}+(\log{T})^{-4}\right). 
\end{align}
Combining \eqref{eq:1800} and \eqref{eq:1801}, we have 
\begin{align*}
D_{\sigma,m,Y}(T)
&\ll 
\left|\frac{1}{T} \int_{T}^{2T} W_{L, \mathcal{R}}(P_{m,Y}(\s+it)) \,dt
-\EXP{W_{L, \mathcal{R}}(\te_{m}(\s, X))}\right| \\
&\qquad+L^{-1}+(\log{T})^{-4}. 
\end{align*}
Thus, with the same assumption as in Proposition \ref{prop:W}, we obtain 
\begin{gather*}
D_{\sigma,m,Y}(T)
\ll (\log{T})^{-2}+L^{-1}
\ll (\log{T})^{-\sigma}(\log\log{T})^{-m},  
\end{gather*}
which yields the desired bound of $D_{\sigma,m}(T)$. 
\end{proof}

\section{\textbf{Cumulant-generating functions}} \label{Sec_PThmLD}

Let $(\s, m)\in\mca{A}$ and $\a\in\RR$. 
We recall that the moment-generating function is defined by \eqref{eqFY}, that is,  
\begin{gather*}
  F_{\s, m}(s; \alpha)
  =\EXP{\exp(s\Re e^{-i\a}\te_{m}(\s, X))}
\end{gather*}
for $s=\kappa+it\in\CC$. 
Note that $F_{\s, m}(\kappa; \alpha)>0$ if $\kappa\in\RR$ by the definition. 
Then, we define the cumulant-generating function
\begin{align*}
  f_{\s, m}(\kappa; \alpha)
  =\log{F}_{\s, m}(\kappa; \alpha)
\end{align*}
for $\kappa\in\RR$, which is a real analytic function. 
The purpose in this section is to give the asymptotic formula for $f_{\s, m}^{(n)}(\kappa; \alpha)$ for $1/2 < \s < 1$ and $m\in\ZZ_{\geq 0}$. 
The formula is given by the following proposition.

\begin{proposition}\label{propCumu1}
  Let $1/2<\s<1$, $m\in\ZZ_{\geq0}$ and $\a \in \mathbb{R}$. 
  There exists a large constant $c = c(\s, m)$ such that 
  for any $\kappa \geq c$ and $n \in \ZZ_{\geq 0}$ 
  \begin{align*}
    f_{\s, m}^{(n)}(\kappa; \alpha)
    =\s^{m/\s}g_n(\s)\frac{\kappa^{(1/\s)-n}}{(\log\kappa)^{(m/\s)+1}}
    \l(1 + O_{\s, m} \l( \frac{2^n n!}{| g_n(\sigma) |} \frac{1+ m \log\log\kappa}{\log\kappa}\r)\r),
  \end{align*}
  where
  \begin{equation}\label{eq_gn}
    g_{n}(\s) = 
    \int_{0}^{\infty}\frac{g^{(n)}(u)}{u^{(1/\s)+1-n}}\, du 
  \end{equation}
  with $g(z)=\log{I}_{0}(z)$ and $I_{0}$ the modified Bessel function of order $0$.
\end{proposition}

To prove this proposition, we need the analyticity of $f_{\s, m}(s; \a)$ for a certain complex domain.
In other words, we need the fact that $F_{\s, m}(s; \a)$ does not have zeros in the complex domain.
In the following, we prove some results to reveal the nonexistence of zeros.

The function $F_{\s, m}(s; \a)$ satisfies
\begin{equation}\label{eqMine3T}
  F_{\s, m}(s; \a) = \prod_{p} F_{\s, m, p} (s; \a)
\end{equation}
by the independence of $X(p)$'s, where $F_{\s, m, p} (s; \a)$ is the function defined by
\begin{align}
  F_{\s, m, p}(s; \a)
  &= \EXP{\exp\l( s \Re e^{-i\a} \te_{m, p}(\s, X(p)) \r)}\\
  \label{eqFp}
  &= \frac{1}{2\pi} \int_{0}^{2\pi}\exp(s \Re e^{-i\a} \te_{m, p}(\s, e^{i\theta}))d\theta.
\end{align}
Our plan for the proof of the non-vanishing of $F_{\s, m, p}$ is the following:
In Section \ref{SSec_RPL}, we prove some properties of polylogarithmic functions.
The properties are used in the proof of the non-vanishing of $F_{\s, m, p}$ when $p$ is small.
When $p$ is large, we use some properties of the Bessel function $I_{0}$.
We describe those properties in Section \ref{Sec_RBF}.

\subsection{Results on polylogarithmic functions} \label{SSec_RPL}

In this section, we prove the following proposition.

\begin{proposition}\label{thmF}
  Let $(\s, m)\in\mca{A}$ and $\a\in\RR$. 
  Let $s = \kappa + it$ satisfy $\kappa>c$ and $|t|\leq\kappa$ with $c$ sufficiently large. 
  For any $p^{\s} (\log{p})^{m} \leq \kappa (\log{\kappa})^{-3}$, we have
  \begin{align*}
  &F_{\s, m, p}(s; \a) =\\
  &\exp\l(\frac{s}{(\log{p})^m}\lambda_{p^{-\s}}(\theta_1; m, \a)\r)\sqrt{\frac{(\log{p})^m}{2\pi{s}|\lambda''_{p^{-\s}}(\theta_1; m, \a)|}}
  \Bigg\{ 1+  O\left( \sqrt{\frac{p^\s (\log p)^m}{\kappa }} \right) \Bigg\}. 
  \end{align*}
  Here, the above square root takes positive values on the positive real axis.
\end{proposition}

To prove this, we use the saddle method. 
The following lemmas are used in the method.

Let $m \in \ZZ$ and $\a \in \RR$. 
We define
\begin{align*}
  \lambda_{r}(\theta; m, \a)
  =\Re e^{-i\a} \Li_{m+1}(re^{i\theta})
  =\sum_{k=1}^\infty\frac{r^k}{k^{m+1}}\cos(k\theta-\a) 
\end{align*}
for $\theta\in\RR$, where $0 < r \leq 1 / \sqrt{2}$ is a real number. 
By the definition, the function $\lambda_{r}(\theta; m, \a)$ satisfies the differential relation
\begin{gather}\label{eqNote1}
\lambda'_{r}(\theta; m, \a)
=\lambda_{r}(\theta; m-1, \a-\pi/2). 
\end{gather}
We begin by the following lemma on zeros of $\lambda_{r}(\theta; m, \a)$. 

\begin{lemma}\label{lemZeros}
Let $m \geq -2$ and $\a\in\RR$. 
For any fixed real number $0<r\leq1/\sqrt{2}$, the function $\lambda_{r}(\theta; m, \a)$ 
has exactly two distinct zeros in the interval $[0, 2\pi)$. 
\end{lemma}

\begin{proof}
We prove this lemma by induction. 
Recall that $\Li_{-1}(z) = z / (1 - z)^2$.
Hence, we can easily confirm that the equation $\lam_{r}(\theta; -2, \a) = 0$ is equivalent to 
that $\cos(\theta - \b) = \frac{2r \cos{\a}}{\sqrt{A^2 + B^2}}$,
where $A = (1 + r^2) \cos{\a}$, $B = (1 - r^2)\sin{\a}$, and $\b$ is defined by the equations $\cos{\b} = A / \sqrt{A^2 + B^2}$ 
and $\sin{\b} = B / \sqrt{A^2 + B^2}$.
Note that we can also easily confirm that $\l|\frac{2r \cos{\a}}{\sqrt{A^2 + B^2}} \r| < 1$ for $0 < r \leq 1 / \sqrt{2}$, 
so $\lam_{r}(\theta; -2, \a)$ has exactly two zeros.
Let $m\in\ZZ_{\geq -1}$. 
We have $\lambda_{r}(\theta; m, \a)=\lambda'_{r}(\theta; m+1, \a+\pi/2)$ by relation \eqref{eqNote1}. 
Note that the function $\lambda_{r}(\theta; m+1, \a+\pi/2)$ is smooth and periodic with period $2\pi$. 
Thus $\lambda'_{r}(\theta; m+1, \a+\pi/2)$ vanishes at least twice in the period. 
Hence, there exist at least two zeros of $\lambda_{r}(\theta; m, \a)$ in $[0, 2\pi)$.  
If there were three or more zeros of $\lambda_{r}(\theta; m, \a)$ in $[0, 2\pi)$, 
then we saw that $\lambda'_{r}(\theta; m, \a)$ has at least three zeros in $[0, 2\pi)$ by Rolle's theorem. 
However, it implies that the function $\lambda_{r}(\theta; m-1, \a-\pi/2)$ has at least three zeros in $[0, 2\pi)$ 
by relation \eqref{eqNote1}, which contradicts the assumption of induction. 
Thus, the function $\lambda_{r}(\theta; m, \a)$ 
has exactly two distinct zeros in the interval $[0, 2\pi)$.
\end{proof}

Let $m\in\ZZ_{\geq 0}$ and $\a\in\RR$. 
Denote by $\theta_1$ and $\theta_2$ the zeros of $\lambda'_{r}(\theta; m, \a)$ with  $0\leq\theta_1<\theta_2<2\pi$. 
Then we have $\lambda_{r}(\theta_1; m, \a)\neq\lambda_{r}(\theta_2; m, \a)$; 
otherwise we have the third zero of $\lambda'_{r}(\theta; m, \a) = \lambda_{r}(\theta; m - 1, \a - \pi / 2)$ between $\theta_1$ and $\theta_2$ 
which contradicts Lemma \ref{lemZeros}.
Furthermore, we obtain the following result as a consequence of Lemma \ref{lemZeros}.

\begin{lemma}\label{propMono}
Let $m \geq 0$ and $\a \in \RR$. 
For $0<r\leq1/\sqrt{2}$, there exist real numbers $\theta_1= \theta_1(\alpha, r) = \theta_1(\alpha, r; m)$ 
and $\theta_2 = \theta_{2}(\a, r) = \theta_2 (\alpha, r; m)$ satisfying  
$\theta_{1} \in [0, 2 \pi)$, $\theta_{1} < \theta_{2} < \theta_{1} + 2\pi$ 
such that the function $\lambda_{r}(\theta; m, \a)$ is strictly decreasing for $\theta_1\leq\theta\leq\theta_2$ 
and is strictly increasing for $\theta_2\leq\theta\leq\theta_1+2\pi$. 
In particular, $\theta_{1}$, $\theta_{2}$ are uniquely determined with respect to $\a$, $r$, $m$. 
\end{lemma}

\begin{proof}
Let $0\leq\tilde{\theta}_1<\tilde{\theta}_2<2\pi$ be the zeros of $\lambda'_{r}(\theta; m, \a)$. 
If $\lambda_{r}(\tilde{\theta}_1; m, \a)>\lambda_{r}(\tilde{\theta}_2; m, \a)$, then we have 
\[\lambda'_{r}(\theta; m, \a)
\begin{cases}
<0 & \text{for $\tilde{\theta}_1<\theta<\tilde{\theta}_2$, }\\
>0 & \text{for $\tilde{\theta}_2<\theta<\tilde{\theta}_1+2\pi$}
\end{cases}\]
since there exist no zeros except for $\tilde{\theta}_1$ and $\tilde{\theta}_2$ by Lemma \ref{lemZeros}. 
Then the result follows by taking $\theta_j=\tilde{\theta}_j$. 
In the case $\lambda_{r}(\tilde{\theta}_1; m, \a)<\lambda_{r}(\tilde{\theta}_2; m, \a)$, we take $\theta_1=\tilde{\theta}_2$ and $\theta_2=\tilde{\theta}_1+2\pi$. 
Then we obtain the desired result similarly. 
\end{proof}

In the rest of this section, $\theta_{1}$ and $\theta_{2}$ are the numbers in Lemma \ref{propMono}.

\begin{lemma} \label{SZPL}
  Let $m \geq 0$, $\a \in \RR$, and let $r \in (0, 1 / \sqrt{2}]$.
  The zeros $\theta_{1}$ and $\theta_{2}$ of $\lam_{r}'(\theta; m, \a)$ are simple, 
  that is, $\lam_{r}''(\theta_{j}; m, \a) \not= 0$ for $j = 1, 2$.
\end{lemma}

\begin{proof}
  Let $m \geq 0$.
  If $\theta_{1}$ is a multiple zero, then $\lam_{r}''(\theta_{1}; m, \a) = 0$.
  By Rolle's theorem, there are zeros $\vartheta_{1}$, $\vartheta_{2}$ of $\lam_{r}''(\theta; m, \a)$ 
  such that $\vartheta_{1} \in (\theta_{1}, \theta_{2})$ and $\vartheta_{2} \in (\theta_{2}, \theta_{1} + 2\pi)$.
  Therefore, the function $\lam_{r}''(\theta; m, \a)$ 
  has at least the distinct three zeros $\theta_{1}$, $\vartheta_{1}$, $\vartheta_{2}$ in $[\theta_{1}, \theta_{1} + 2\pi)$.
  Since the equation $\lam_{r}''(\theta; m, \a) = \lam_{r}(\theta; m - 2, \a - \pi)$ holds by relation \eqref{eqNote1}, 
  the function $\lam_{r}(\theta; m - 2, \a - \pi)$ also has at least three zeros in $[0, 2\pi)$.
  This is a contradiction with Lemma \ref{lemZeros}. 
  Hence, $\theta_{1}$ is simple.
  Similarly, $\theta_{2}$ is also simple, and hence we complete the proof of this lemma.
\end{proof}

\begin{lemma} \label{lemLB}
  Let $m \geq 0$, $\a \in \RR$, and let $r \in (0, 1 / \sqrt{2}]$.
  Then we have $|\lambda''_{r}(\theta_1; m, \a)| \gg r$. 
  Moreover, there exists an absolute constant $d > 0$ such that $d \leq \theta_2 - \theta_1  \leq 2\pi - d$.
\end{lemma}

\begin{proof}
  Let $\mathcal{B}$ denote 
  \[
  \mathcal{B}
  = \Set{ \l(m, (0, 1/\sqrt{2}] \r)}{ m \in \mathbb{Z}_{\geq 2}}
  \cup \Set{ \l(m, (0, 0.27] \r)}{m = 0, 1}.
  \]  
  We first consider the case $(m, I) \in \mathcal{B}$ and $r \in I$.
  Since $\lambda'_{r}(\theta_1; m, \alpha)=0$, we find that
\begin{align*}
\l | \sin(\theta_1-\a) \r|
= \l |-\sum_{k=2}^\infty\frac{r^{k -1}}{k^{m}}\sin(k\theta_1-\a) \r |
\leq \sum_{k=2}^\infty\frac{r^{k -1}}{k^{m}}
\leq f_m(r), 
\end{align*}
\[
f_m(r) 
=\begin{dcases}
r^{-1} \Li_{2}(r) - 1 & \text{if $m \in \mathbb{Z}_{\geq 2}$}, \\
\frac{r}{1 - r}       & \text{if $m = 0, 1$.}
\end{dcases}
\]
holds for $r \in I$.
Note that $f_{m}$ is increasing.
Since $f_{m}(1 / \sqrt{2}) = 0.275\dots$ for $m \in \ZZ_{\geq 2}$, and $f_m(0.27)=0.369\dots$ for $m = 0, 1$, 
we have
\begin{equation}
\l| \cos(\theta_1-\a) \r|
= \sqrt{ 1 - \sin(\theta_1-\a)^2 }
\geq \sqrt{ 1 - f_m(r)^2 }
\geq \sqrt{ 1 - (0.37)^2 }
=: c_1
\end{equation}
for $r \in I$. 
Furthermore, we obtain
\begin{align*}
\l|\lambda''_{r}(\theta_1; m, \a)+r\cos(\theta_1-\a)\r|
\leq\sum_{k=2}^\infty\frac{r^k}{k^{m-1}}
\leq r g_m(r), 
\end{align*}
\[
g_m(r) 
=\begin{dcases}
r^{-1}\Li_{1}(r) - 1 & \text{if $m \in \mathbb{Z}_{\geq 2}$}, \\
\frac{ 2r - r^2 }{( 1 - r )^2} & \text{if $m = 0, 1$}
\end{dcases}
\]
for $r \in I$.
If we suppose $\cos(\theta_1-\a) \leq - \sqrt{ 1 - f_{m} (r)^2 }$, 
then we have
\[
\lambda''_{r}(\theta_1; m, \a)
\geq -r \cos(\theta_{1} - \a) - rg_{m}(r)
\geq \l(\sqrt{ 1 - f_m (r)^2 } - g_m (r) \r)r =: h_m(r) r,
\]
say. Here, we note that $h_{m}$ is decreasing since $f_{m}$ and $g_{m}$ are increasing.
Therefore, we have
\[
h_m(r) 
\begin{cases}
\geq h_{m}(1 / \sqrt{2}) \geq 0.224\dots  & \text{if $m \in \mathbb{Z}_{\geq 2}$}, \\
\geq h_m(0.27) = 0.0525\dots              & \text{if $m = 0, 1$}
\end{cases}
\]
for $r \in I$, which contradicts that $\lambda_{r}''(\theta_{1}; m, \a) \leq 0$ since $\lam_{r}(\theta; m, \a)$ takes the maximum value at $\theta=\theta_1$. 
Thus we have $\cos(\theta_1-\a)>\sqrt{ 1 - f_m (r)^2 } \geq c_1 > 0$, and therefore 
\[
\lambda''_{r}(\theta_1; m, \a)
< - \left( \sqrt{ 1 - f_m (r)^2 } - g_m(r) \right) r
=- h_m(r)r 
\leq - 0.0525\dots \times r.
\] 
Since $\l|\lambda''_{r}(\theta_1; m, \a)\r|=-\lambda''_{r}(\theta_1; m, \a)$, 
we obtain $|\lam_{r}''(\theta_{1}; m, \a)| \gg r$ when $(m, I) \in \mathcal{B}$ and $r \in I$. 
On the other hand, we have $\cos(\theta_2-\a)< - \sqrt{ 1 - f_m (r)^2 } \leq - c_1$ by a similar calculation. 
Putting 
\begin{gather*}
d_{1} = \inf \set{ | \omega_1 - \omega_2 |}{\cos{\omega_1} \in [c_1, 1], \cos{\omega_2} \in [-1, -c_1]} > 0, \\
d_{2} = 2\pi - \sup\set{ | \omega_1 - \omega_2 |}{\cos{\omega_1} \in [c_1, 1], \cos{\omega_2} \in [-1, -c_1]} > 0,
\end{gather*}
we also have $d' \leq \theta_{2} - \theta_{1} \leq 2\pi - d'$ with $d' = \min\{ d_{1}, d_{2} \}$ when $(m, I) \in \mathcal{B}$ and $r \in I$.

Next, we consider the case when $0.27 < r \leq 1/ \sqrt{2}$, $m = 0, 1$.
Let $m\in \{0,1\}$ and let $\e$ be a small positive number. 
For any $(\a, r) \in \RR \times (0, 1/\sqrt{2} + \e)$, 
we denote by $\Theta_1(\a, r;m)$ (resp.\ $\Theta_{2}(\a, r; m)$) the subset of $\RR$ consisting of all $\theta$ such that $\lambda_r(\theta;m,\a)$ 
takes the maximum (resp.\ minimum) value.
Using the periodicity of $\lam_{r}(\theta; r, m)$ for $\theta$ and Lemma \ref{propMono}, 
we find that $\Theta_j(\a, r;m) \in \RR/2\pi \ZZ$ for $j = 1,2$, 
and that $\Theta_1(\a, r;m) \neq \Theta_2(\a, r;m)$. 
We show the mapping $\RR \times (0, 1/\sqrt{2} + \e) \ni (\a, r) \mapsto \Theta_j(\a, r;m)\in \RR/2\pi \ZZ$ is continuous for any $j = 1, 2$.

We fix $(\a_0, r_0) \in \RR \times (0, 1/\sqrt{2} + \e)$.
Define $f(\theta, \a, r) = f(\theta, \a, r; m) = \lam_{r}'(\theta; m, \a)$.
Then we can apply the implicit function theorem to $f(\theta, \a_{0}, r_{0})$ at $\theta = \theta_{j}(\a_0, r_0)$ 
since $\frac{\partial}{\partial \theta}f(\theta_{j}(\a_{0}, r_{0}); \a_{0}, r_{0}) \neq 0$ for $j =1,2$ by Lemma \ref{SZPL}. 
Hence, there exist some open neighborhood 
of $(\a_0, r_0)$ 
and the continuous functions $\widetilde{\theta_1}(\a, r)$ and $\widetilde{\theta_2}(\a, r)$ such that 
$\widetilde{\theta_{j}}(\a_{0}, r_{0}) = \theta_{j}(\a_{0}, r_{0})$, $j = 1, 2$, and
\[
\lambda'_r(\widetilde{\theta_1}(\a, r);m,\a) = \lambda'_r(\widetilde{\theta_2}(\a, r);m,\a) = 0
\]
for any $(\a, r)$ belonging to some neighborhood of $(\a_{0}, r_{0})$.
Moreover, using the inequality $\lambda_{r_0}(\theta_1(\a_0, r_0);m,\a_0)) > \lambda_{r_0}(\theta_2(\a_0, r_0);m,\a_0))$ 
and the continuities of $\lam_{r}(\theta; m, \a)$ and $\widetilde{\theta_{j}}$, we have
\[
\lambda_r(\widetilde{\theta_1}(\a, r);m,\a) > \lambda_r(\widetilde{\theta_2}(\a, r);m,\a)
\]
when $(\a, r)$ belongs to a suitable neighborhood of $(\a_{0}, r_{0})$.
These imply $\widetilde{\theta_j}(\a, r) \in \Theta_j(\a, r; m)$ for $j = 1, 2$ for those $(\a, r)$. 
Letting the quotient mapping $\Pi : \RR \ni \theta \mapsto \theta + 2 \pi \ZZ \in \RR/2\pi \ZZ$, 
we find that $\Theta_j = \Pi \circ \widetilde{\theta_j}$ for $j = 1, 2$. 
This gives the continuities of $\Theta_1$ and $\Theta_2$. 
Therefore $\Theta_1([0,2\pi] \times [0,27, 1/\sqrt{2}];m)$ is a compact subset of $\RR/2\pi \ZZ$ by the continuity of $\Theta_1$. 
Since $\lambda''_r(\theta; m, \a)$ is continuous for $(\theta, \a, r)$ and periodic for $\theta$ with period $2 \pi$, 
and $| \lambda''_r(\vartheta_1;m,\a) | > 0$ for $\vartheta_1 + 2 \pi \ZZ \in \Theta_1([0, 2\pi] \times [0,27, 1/\sqrt{2}];m)$, 
we obtain
\[
c := \min_{m \in \{0, 1\}}\min_{\vartheta_1 + 2\pi \ZZ \in \Theta_1([0, 2\pi] \times [0,27, 1/\sqrt{2}];m)} \l| \lambda''_r(\vartheta_1;m,\a) \r| 
> 0.
\]
Thus we conclude $|\lambda''_r(\theta_1;m,\a) | \gg r$ even for $m = 0, 1$ and $0.27 < r \leq 1/\sqrt{2}$.

Now we denote by $\| \cdot \|_{\sim}$ the quotient norm on $\RR/2\pi \ZZ$, 
which is given by $\| x + 2 \pi \ZZ \|_{\sim}:= \min\l\{| x + 2 \pi n | \mid n \in \ZZ\r\}$. 
Then we have
\begin{align*}
&\|(\theta_{2}(\a, r; m) - \theta_{1}(\a, r; m)) + 2\pi \ZZ \|_{\sim}\\
&\geq \min_{m \in \{ 0, 1\}}\min_{(\a, r) \in [0, 2 \pi] \times [0.27, 1/\sqrt{2}]} \l\| \Theta_2(\a, r; m) - \Theta_1(\a, r; m) \r\|_{\sim}
=: d_{3}
\end{align*}
for any $(\a, r) \in [0, 2\pi] \times [0.27, 1/\sqrt{2}], m \in \{0, 1\}$
by the continuities of $\Theta_1$ and $\Theta_2$. 
For any fixed $(\a, r) \in [0, 2\pi] \times [0.27, 1/\sqrt{2}], m \in \{0, 1\}$, 
it holds that $\|\Theta_2(\a, r; m) - \Theta_1(\a, r; m)\|_{\sim} > 0$, so $d_{3}$ is an absolute positive constant.
Hence we obtain $d_{3} \leq \theta_2(\a, r; m) - \theta_1(\a, r; m) \leq 2 \pi - d_3$ for $m = 0, 1$ and $0.27 < r \leq 1/\sqrt{2}$.
Taking $d = \min\{d', d_3\}$, we have the last assertion of this lemma. 
\end{proof}

\begin{lemma}\label{lemUB}
  Let $m \geq -1$ and $\a\in\RR$. 
  Then we have uniformly
  \[\lambda_{r}^{(n)}(\theta; m, \a)\ll{n}!\, {r}\]
  for $0<r\leq1/\sqrt{2}$, $n\geq0$, and $\theta\in\RR$. 
  \end{lemma}
  
  \begin{proof}
  By \eqref{eqNote1} and the definition of $\lambda_{r}(\theta; m, \a)$, we have 
  \[\lambda_{r}^{(n)}(\theta; m, \a)
  =\lambda_{r}(\theta; m-n, \a-n/2)
  \ll\sum_{k=1}^\infty{k}^nr^k=:S_n(r). \]
  Then we prove the upper bound $S_n(r)\ll{n}!\, r$ by induction on $n$.
  The bound is elementary for $n=0$. 
  If $n\geq1$, we have 
  \begin{align*}
    (1-r)S_n(r)
    =r+\sum_{k=1}^\infty\l\{(k+1)^n-k^n\r\}r^{k+1}
    &=r\l(1+\sum_{j=0}^{n-1}\binom{n}{j}S_j(r)\r)\\
    &\ll r + n! r^{2}
    \ll n! r
  \end{align*}
  by the assumption of induction.
\end{proof}

\begin{proof}[Proof of Proposition \ref{thmF}]
Let $(\s, m)\in\mca{A}$, $\a\in\RR$. 
We use the saddle point method. 
For convenience, we write
\[
\lambda(\theta)=\lambda_{p^{-\s}}(\theta; m, \a). 
\] 
Using the periodicity of $\lambda(\theta) = (\log{p})^{m} \Re e^{-i\a} \te_{m, p}(p^{-\s} e^{i\theta})$ and the integral representation of \eqref{eqFp}, 
we have
\begin{align*}
  &F_{\s, m, p}(s; \a)
  = \frac{1}{2\pi}\int_{\theta_{1} - \e}^{\theta_{1} + 2\pi - \e}\exp\l(\frac{s}{(\log{p})^m}\lambda(\theta)\r)\, d\theta\\
  &= \frac{1}{2\pi}\int_{\theta_{1} - \e}^{\theta_{1} + \e}\exp\l(\frac{s}{(\log{p})^m}\lambda(\theta)\r)\, d\theta
  + \frac{1}{2\pi}\int_{\theta_{1} + \e}^{\theta_{1} + 2\pi - \e}\exp\l(\frac{s}{(\log{p})^m}\lambda(\theta)\r)\, d\theta\\
  &=: I_{1} + I_{2},
\end{align*}
where 
\[
\epsilon
= \epsilon(\kappa, p; \s, m) =  \log{\kappa} \sqrt{\frac{2 ( \log p )^m}{ \kappa | \lambda^{''}(\theta_1) |}}.
\]

We begin with the estimate of $I_1$.
Let $d$ be the same absolute positive constant as in Lemma \ref{lemLB}.
Then we have $\epsilon < d$ if $\kappa$ is sufficiently large.
Since $\lambda'(\theta_1) = 0$, 
we have
\[
\frac{s}{(\log p)^m}\left(\lambda(\theta) - \lambda(\theta_1)\right)
= \frac{ s \lambda''(\theta_1)}{ 2(\log p)^{m} } (\theta - \theta_1)^{2} + E_{0} 
\]
for $|\theta - \theta_{1}| \leq \epsilon$, where $E_{0} \ll \frac{\kappa}{p^\s ( \log p )^m} ( \theta - \theta_1 )^{3}$
by Lemma \ref{lemUB} when $|t| \leq \kappa$.
Hence we can write
\begin{align*}
  &\int_{\theta_{1} - \e}^{\theta_{1} + \e}\exp\l(\frac{s}{(\log{p})^m}\lambda(\theta)\r)\, d\theta\\
  &= \exp\l( \frac{s \lam(\theta_{1})}{(\log{p})^{m}} \r)\int_{\theta_{1} - \e}^{\theta_{1} + \e}
  \exp\l( \frac{s \lam''(\theta_{1})}{2(\log{p})^{m}}(\theta - \theta_{1})^2 + E_{0} \r)\, d\theta.
\end{align*}
We note that our choice of $\e$ yields $\kappa p^{-\s} ( \log p )^{-m} ( \theta - \theta_1 )^{3} \ll 1$ for $|\theta - \theta_{1}| \leq \e$. 
The change of variables $\theta - \theta_{1} = x\sqrt{2(\log{p})^{m} / \kappa |\lam''(\theta_{1})|}$ and the Taylor expansion of $\exp$ yield
\begin{align*}
&\int_{\theta_1- \e}^{\theta_1+\epsilon}\exp\l(\frac{s}{(\log{p})^m}\lambda(\theta)\r)\, d\theta\\
&= \exp\l( \frac{s}{( \log p )^m} \lambda(\theta_1) \r) \sqrt{ \frac{2 ( \log p )^m }{\kappa |\lambda''(\theta_1)|}} 
\left\{ \int_{- \log{\kappa}}^{\log{\kappa}} \exp \l( - \frac{s}{\kappa} x^{2} \r) \, dx + E_1 \right\}.
\end{align*}
Here, we have
\begin{align*}
E_1 
&\ll \frac{\kappa}{p^\s (\log p)^m} \left( \frac{2(\log p)^m}{\kappa | \lambda'' (\theta_1) |} \right)^{\frac{3}{2}} \int_{-\infty}^{\infty} x^{3}  
\exp\l( - x^{2} \r)\, dx
\ll \left( \frac{p^\s (\log p)^m}{\kappa} \right)^{\frac{1}{2}}
\end{align*}
since $|\lambda''(\theta_1)| \gg p^{- \s}$ holds.
The calculations 
\[
\int_{-\infty}^{\infty} \exp \l( - \frac{s}{\kappa} x^{2} \r) \, dx
= \l(\frac{\kappa}{s}\r)^{\frac{1}{2}} \int_{-\infty}^{\infty} \exp( - x^{2}) dx
= \l(\frac{\kappa}{s}\r)^{\frac{1}{2}} \sqrt{\pi}
\]
and
\begin{align*}
\int_{\log \kappa}^{\infty} \exp ( - x^{2} ) dx 
&\leq \frac{1}{ \log{\kappa} } \int_{\log \kappa}^{\infty} x \exp ( - x^{2} ) dx
= \frac{1}{2 \log{\kappa} } \exp \left( - ( \log \kappa )^{2} \right)
\end{align*}
imply
\begin{align*}
  I_{1}
  = \exp\l( \frac{s}{( \log p )^m} \lambda(\theta_1) \r) \sqrt{ \frac{ ( \log p )^m}{ 2 \pi s | \lambda''(\theta_1) |}}
  \left( 1 + O\left( \sqrt{ \frac{p^\sigma (\log p)^m}{\kappa}} \right) \right). 
\end{align*}

Next, we estimate $I_{2}$.
We can write
\begin{align*}
I_{2}
=  \exp\l(\frac{s}{(\log{p})^m}\lambda(\theta_1)\r) \int_{\theta_{1} + \e}^{\theta_{1} + 2\pi - \e} \exp\l(\frac{s}{(\log{p})^m} \l( \lambda(\theta) -  \lambda(\theta_1) \r)\r)\, d\theta. 
\end{align*}
Since $\lambda(\theta)$ is decreasing for $\theta_1 + \epsilon \leq \theta \leq \theta_2$ 
and increasing for $\theta_{2} \leq \theta \leq \theta_{1} + 2\pi - \e$ by Lemma \ref{propMono}, we find that
\begin{align*}
\lambda(\theta)-\lambda(\theta_1)
&\leq\lambda(\theta_1 \pm \epsilon) - \lambda(\theta_1) \\
&= \frac{\lambda'' (\theta_1) }{2} \epsilon^{2} \left( 1 + O(\e) \right)
\leq - \frac{ (\log \kappa)^{2} (\log p)^{m}}{2 \kappa}
\end{align*}
by the definition of $\epsilon$. 
Therefore we have
\[
I_{2} = \int_{\theta_1+\epsilon}^{\theta_1 + 2\pi - \e} \exp\l(\frac{s}{(\log{p})^m} \l( \lambda(\theta) -  \lambda(\theta_1) \r)\r)\, d\theta
\ll \exp\left( - \frac{(\log \kappa)^{2}}{2} \right).
\]
Thus we have 
\begin{align*}
  &F_{\s, m, p}(s; \a)\\
  &= \exp\l( \frac{s}{( \log p )^m} \lambda(\theta_1) \r) \sqrt{ \frac{( \log p )^m}{ 2\pi s | \lambda''(\theta_1) |}}
  \left( 1 + O\left( \sqrt{ \frac{p^\sigma (\log p)^m}{\kappa}} \right) \right),
\end{align*}
which gives the asymptotic formula $F_{\s, m, p} (s; \a)$.
\end{proof}

Define the function $f_{\s, m, p}$ by 
\begin{align} \label{def_f_smp}
  f_{\s, m, p}(\kappa; \a) = \log{F_{\s, m, p}(\kappa; \a)}
\end{align} 
  for $\kappa > 0$.
The following corollary is an immediate consequence of Proposition \ref{thmF}.

\begin{corollary}\label{corf}
Assume the same situation as Proposition \ref{thmF}.
Then the function $f_{\s, m, p}$ is analytically continued for $s = \kappa + it$ 
with $\kappa \geq C p^{\s} (\log{p})^{m + 3}$, $|t| \leq \kappa$, 
where $C = C(m)$ is a positive constant depending only on $m$.
Moreover, we have for the region 
\begin{align*}
  f_{\s, m, p}(s; \a)
  \ll \frac{\kappa}{p^\s(\log{p})^m}. 
\end{align*}
\end{corollary}

\subsection{Results on the Bessel function $I_{0}$} \label{Sec_RBF}

We further prepare some lemmas on the Bessel function $I_{0}(z)  = \frac{1}{2\pi}\int_{-\pi}^{\pi} \exp(z \cos \theta)d\theta$. 
Put
\begin{gather}\label{eqdelta}
\Delta=\set{z=x+iy}{x\geq0, |y|\leq{x}}. 
\end{gather}

\begin{lemma}\label{lemBessel}
We have $|I_0(z)|\asymp{I}_0(x)$ for all $z\in\Delta$. 
\end{lemma}

\begin{proof}
The inequality $|I_0(z)|\leq{I}_0(x)$ is deduced from the definition $I_{0}$. 
We prove that $|I_0(x)/I_0(z)|$ is bounded if $x\geq0$ and $|y|\leq{x}$. 
Recall that the asymptotic formula (see \cite[page 122]{LSF})
\begin{align}	\label{BFB0}
I_0(z)=\frac{e^z}{\sqrt{2\pi{z}}}\l(1 + O\l(|z|^{-1}\r)\r)
\end{align}
holds if $\Re{z}>0$. 
Hence we see that there exists an absolute constant $R>0$ such that 
\[\l|\frac{I_0(x)}{I_0(z)}\r|\leq2\sqrt{\frac{|z|}{x}}\leq2\sqrt{2}\]
if $|z|>{R}$ and $z\in\Delta$. 
Recall that the modified Bessel function ${I}_0(z)$ is non-zero and holomorphic for $\Re{z}>0$. 
Hence $|I_0(x)/I_0(z)|$ is bounded if $|z|\leq{R}$ and $z\in\Delta$, and so we complete the proof. 
\end{proof}

Define 
\begin{equation}\label{eqn:LogBe}
g(z)=\log{I}_0(z)\end{equation}
as a holomorphic function on $\Re{z}>0$, whose values are real on the real axis. 

\begin{lemma}\label{lemBF}
We have the followings; 
\begin{itemize}
\item[{\rm (i)}] The estimate $g(z) = z^2/4 + O\l( | z |^4 \r)$ holds for $|z| \leq 1$. 
\item[{\rm (ii)}] The estimate $g(z) \ll |z|$ holds for $z \in \Delta$.
\item[{\textrm{(iii)}}] The estimate $g'(z) \ll 1$ for $z \in \Delta$.
\end{itemize}
\end{lemma}

\begin{proof}
By the Taylor expansion of $\exp(z)$, we have
\[
I_0(z)
= 1 + E(z) z^2, \quad 
E(z) = \frac{1}{4} + \sum_{n =2}^\infty \frac{1}{4(n!)^2} \l(\frac{z}{2}\r)^{2(n -1)}.
\]
Since the estimate $\l | E(z) \r | < 1$ holds for $|z| \leq 1$, 
we obtain
\[
\log I_0 (z)
= \sum_{k =1}^\infty \frac{(-1)^{k-1}}{k}E(z)^{k} z^{2km}
= \frac{z^2}{4} + O\left( |z|^4 \right)
\]
for $|z| \leq 1$, 
which gives the first assertion.
The second assertion immediately follows from the asymptotic formula \eqref{BFB0}.
The third assertion can be easily obtained from the second assertion and Cauchy's integral formula.
This completes the proof.
\end{proof}

\begin{lemma}\label{lemCumu}
Let $(\s, m)\in\mca{A}$ and $\a\in\RR$. 
Suppose that $\kappa{p}^{-2\s}(\log{p})^{-m} \leq \delta$ is satisfied with a positive small absolute constant $\delta$. 
Then the function $f_{\s, m, p}(s; \a)$ given by \eqref{def_f_smp} can be analytically continued to the region $\Delta$.
Moreover we have
\[
f_{\s, m, p}(s; \a)
=g \l(\frac{s}{p^\s(\log{p})^m}\r)
+O\l(\frac{\kappa}{p^{2\s}(\log{p})^m}\r)\]
for $s = \kappa + it \in \Delta$.
Here, the region $\Delta$ is given by \eqref{eqdelta}.
\end{lemma}

\begin{proof}
We can write
\[
\te_{m, p}(\s, X(p))
=\frac{X(p)}{p^\s(\log{p})^m} + E_{\s, m} (p), \quad E_{\s, m} (p) \ll\frac{1}{p^{2\s}(\log{p})^m}.
\] 
Recalling that $|s|\leq2\kappa$ holds for every $s=\kappa+it\in\Delta$, we have 
\[\exp\l( s \Re e^{-i\a} E_{\s, m}(p)  \r)=1+O\l(\frac{\kappa}{p^{2\s}(\log{p})^m}\r)\]
for $s\in\Delta$ if $p$ satisfies $\kappa{p}^{-2\s}(\log{p})^{-m}\leq\delta$. 
Hence $F_{\s, m, p}(s; \a)$ is calculated as 
\begin{align*}
F_{\s, m, p}(s; \a)
&=\EXP{\exp\l(s\frac{\Re e^{-i\a} X(p)}{p^\s(\log{p})^m}\r)}\\
&\qquad+\EXP{\exp\l(s\frac{\Re e^{-i\a} X(p)}{p^\s(\log{p})^m}\r) O\l(\frac{\kappa}{p^{2\s}(\log{p})^m}\r)}\\
&=I_0\l(\frac{s}{p^\s(\log{p})^m}\r)+O\l({I}_0\l(\frac{\kappa}{p^\s(\log{p})^m}\r)\frac{\kappa}{p^{2\s}(\log{p})^m}\r).
\end{align*}
By Lemma \ref{lemBessel}, it implies
\begin{gather}\label{eqMine23T}
F_{\s, m, p}(s; \a)
=I_0\l(\frac{s}{p^\s(\log{p})^m}\r)\l(1+O\l(\frac{\kappa}{p^{2\s}(\log{p})^m}\r)\r).
\end{gather}
Therefore, if $\delta$ is sufficiently small, $F_{\s, m, p}(s; \a)\neq0$ for $s\in\Delta$. 
Hence we define $f_{\s, m, p}(s; \a) = \log F_{\s, m, p}(s; \a)$ as before.
We have by \eqref{eqMine23T} the formula
\[f_{\s, m, p}(s; \a)=g\l(\frac{s}{p^\s(\log{p})^m}\r)+O\l(\frac{\kappa}{p^{2\s}(\log{p})^m}\r). \]
This completes the proof.
\end{proof}

\subsection{Proof of Proposition \ref{propCumu1}}

\begin{proof}[Proof of Proposition \ref{propCumu1}]
Let $\a \in \RR$.
First, we show the asymptotic formula
\begin{align}
&f_{\s, m}(s; \alpha)
 \label{eqEndo1T}
=\s^{\frac{m}{\s}}g_0(\s)\frac{s^{\frac{1}{\s}}}{(\log\kappa)^{\frac{m}{\s}+1}}
\l(1+O_{\s, m}\l(\frac{1 + m\log\log\kappa}{\log\kappa} 
\r)\r)  \end{align}
for $s=\kappa+it$ satisfying $\kappa > c$ and $|t| \leq \kappa$ with $c > 0$ a large constant, 
where we choose the branch of $s^{\frac{1}{\s}}$ which takes positive values on the positive real axis.
Let $y_1$ and $y_2$ be the parameters determined by
\[\frac{\kappa}{{y}_1^{2\s}}=\delta
\quad\text{and}\quad
\frac{\kappa}{{y}_2^{\s}}=\l(\frac{1}{\log\kappa}\r)^{\frac{\s}{2\s-1}}, \]
where $\delta>0$ is a constant in Lemma \ref{lemCumu}.
Using formula \eqref{eqMine3T} along with Corollary \ref{corf}, Lemmas \ref{lemBF} and \ref{lemCumu}, we have 
\begin{gather}\label{eqMine24T}
f_{\s, m}(s; \alpha)
=\sum_{y_1<p\leq y_2} g\l(\frac{s}{p^\s(\log{p})^m}\r) 
+ E_1,
\end{gather}
where 
\begin{align*}
E_1
&\ll \sum_{p\leq{y}_1}\frac{\kappa}{p^\s(\log{p})^m}
+ \sum_{p > y_1 } \frac{\kappa}{p^{2\s}(\log{p})^m}
+ \sum_{p>y_2} \frac{\kappa^2}{p^{2\s}(\log{p})^{2m}} \\
&\ll_{\s, m} \frac{\kappa {y_1}^{1 - \s}}{ (\log {y_1})^{m +1} }
+\frac{\kappa^2 {y_2}^{1-2\s}}{ (\log {y_2})^{2m +1} }.
\end{align*}
The first term comes from Corollary \ref{corf}, the second term from Lemma \ref{lemCumu}, 
and the third term from Lemmas \ref{lemBF}, \ref{lemCumu}.
Since $y_1\asymp_{\s}\kappa^{\frac{1}{2\sigma}}$ and $y_2=\kappa^{\frac{1}{\s}}(\log\kappa)^{\frac{1}{2\s-1}}$, we obtain 
\begin{gather}\label{eqMine25T}
E_1
\ll_{\s, m} \frac{\kappa^{\frac{1}{2\s}+\frac{1}{2}}}{(\log\kappa)^{m+1}}
+ \frac{\kappa^{\frac{1}{\s}}}{(\log\kappa)^{2m+2}}
\ll_{m} \frac{\kappa^{\frac{1}{\s}}}{(\log\kappa)^{\frac{m}{\s}+2}}
\end{gather}
if $\kappa$ is large enough.

The main term comes from the terms for $y_1<p\leq y_2$. 
Recall that the asymptotic formula
\[\pi(y) := \sum_{p \leq y} 1=\int_2^y\frac{dt}{\log{t}}+O\l(ye^{-8\sqrt{\log{y}}}\r)\]
holds. 
Then, by partial summation, we have 
\[
\sum_{y_1<p\leq{y}_2}g \l(\frac{s}{p^\s(\log{p})^m}\r)
=\int_{y_1}^{y_2}g \l(\frac{s}{y^\s(\log{y})^m}\r)\frac{dy}{\log{y}}+E_2, \]
where 
\begin{align*}
E_2
\ll_m& \l | g\l(\frac{s}{{y_1}^\s(\log{y_1})^m}\r) \r | y_1e^{-8\sqrt{\log{y}_1}}
+ \l | g\l(\frac{s}{{y_2}^\s(\log{y_2})^m}\r) \r | y_2 e^{-8\sqrt{\log{y}_2}} \\
& + \kappa \int_{y_1}^{y_2}  \l | g' \l(\frac{s}{{y}^\s(\log{y})^m}\r) \r | \frac{e^{-8\sqrt{\log{y}}}}{y^{\s} ( \log y )^{m}} \, dy.
\end{align*}
Since $s \in \Delta$ and $\l | s {y_2}^{-\s}(\log{y_2})^{- m} \r | \leq 1$ hold,  
Lemma \ref{lemBF} yields
\[
\l | g\l(\frac{s}{{y_1}^\s(\log{y_1})^m}\r) \r | y_1e^{-8\sqrt{\log{y}_1}}
\ll \frac{\kappa e^{-8\sqrt{\log{y}_1}}}{{y_1}^{\s -1}(\log{y_1})^m}
\ll_{\s, m}\frac{\kappa^{\frac{1}{2\s}+\frac{1}{2}}}{(\log\kappa)^m}e^{-4\sqrt{\log\kappa}},
\]
and
\[
\l | g\l(\frac{s}{{y_2}^\s(\log{y_2})^m}\r) \r | y_2 e^{-8\sqrt{\log{y}_2}}
\ll \frac{\kappa^2 e^{-8\sqrt{\log{y}_2}}}{{y_2}^{2 \s -1}(\log{y_2})^{2m}}
\ll_{\s, m} \frac{\kappa^{\frac{1}{\s}}}{(\log\kappa)^{2m+1}}e^{-8\sqrt{\log\kappa}}.
\]
The third term is estimated as 
\begin{align*}
&\kappa \int_{y_1}^{y_2}  \l | g' \l(\frac{s}{{y}^\s(\log{y})^m}\r) \r | \frac{e^{-8\sqrt{\log{y}}}}{y^{\s} ( \log y )^{m}} \, dy \\
\ll \, & \kappa {e}^{-8\sqrt{\log{y}_1}} \int_{1}^{y_2} \frac{d y}{y^{\s} ( \log y )^{m}}
\ll_{\s, m} \kappa^{\frac{1}{\s}}e^{-4\sqrt{\log\kappa}}
\end{align*}
by Lemma \ref{lemBessel}.
As a result, the error term $E_2$ is estimated as 
\[E_2\ll_{\s, m}\frac{\kappa^{\frac{1}{\s}}}{(\log\kappa)^{\frac{m}{\s}+2}}.\]
Next, making the change of variables $u = \frac{\kappa}{y^{\s}(\log{y})^{m}}$, we obtain 
\begin{align*}
&\int_{y_1}^{y_2}g \l(\frac{s}{y^\s(\log{y})^m}\r)\frac{dy}{\log{y}}\\
&=\s^{\frac{m}{\s}}\kappa^{\frac{1}{\s}} \int_{u_2}^{u_1} \frac{g\l(\frac{su}{\kappa}\r)}{u^{\frac{1}{\s} +1 }(\log(\frac{\kappa}{u}))^{\frac{m}{\s}+1}} \l(1+O_{\s, m}\l(\frac{m \log\log\kappa}{\log\kappa}\r)\r) du, 
\end{align*}
where we put 
\[
u_1 = \frac{\kappa}{y_1^\s ( \log y_1 )^m}
\quad \text{and} \quad 
u_2 = \frac{\kappa}{y_2^\s ( \log y_2 )^m}.
\]
Since we have 
\[\frac{1}{\l(\log\l(\frac{\kappa}{u}\r)\r)^{\frac{m}{\s}+1}}
=\frac{1}{(\log\kappa)^{\frac{m}{\s}+1}}\l(1+O_{\s, m}\l(\frac{|\log{u}|}{\log\kappa}\r)\r)\]
for $u_1\leq{u}\leq{u}_2$ and the estimates
\[
\int_{0}^{\infty} \frac{\l |g\l(\frac{su}{\kappa}\r) \r |  }{u^{ \frac{1}{\s} +1}}du \ll_\s 1, \quad \int_{0}^{\infty} \frac{\l |g\l(\frac{su}{\kappa}\r) \log u \r |  }{u^{ \frac{1}{\s} +1}}du \ll_\s 1
\]
holds by Lemma \ref{lemBF}, the integral is calculated as 
\begin{align*}
&\int_{y_1}^{y_2} g\l( \frac{s}{p^\s(\log{p})^m} \r)\frac{dy}{\log{y}}\\
&=\s^{\frac{m}{\s}}\frac{\kappa^{\frac{1}{\s}}}{(\log\kappa)^{\frac{m}{\s}+1}}
\left\{ \int_{u_2}^{u_1}\frac{g\l(\frac{su}{\kappa}\r)}{u^{\frac{1}{\s}+1}}\, du
+O_{\s, m}\l(\frac{1 + m \log\log\kappa}{\log\kappa}\r) \right\}.
\end{align*}
Finally, we see that the estimates
\begin{align*}
\int_0^{u_2} \frac{g\l(\frac{su}{\kappa}\r)}{u^{\frac{1}{\s}+1}}\, du
&\ll_\sigma {u}_2^{2-\frac{1}{\s}}\ll_{\s, m}\frac{1}{\log\kappa} \\
\int_{u_1}^\infty \frac{g\l(\frac{su}{\kappa}\r)}{u^{\frac{1}{\s}+1}}\, du
&\ll_\sigma {u}_1^{1-\frac{1}{\s}}\ll_{\s, m}\kappa^{\frac{1}{2}-\frac{1}{2\s}} (\log \kappa)^m
\end{align*}
hold by Lemma \ref{lemBF}, and therefore, the asymptotic formula
\begin{align*}
\int_{u_2}^{u_1}\frac{g\l(\frac{su}{\kappa}\r)}{u^{\frac{1}{\s}+1}}\, du
=\int_0^{\infty}\frac{g\l(\frac{su}{\kappa}\r)}{u^{\frac{1}{\s}+1}}\, du 
+O_{\s, m}\l(\frac{1}{\log\kappa}
\r)
\end{align*}
follows. 
Using the equation
\[
\int_0^{\infty}\frac{g\l(\frac{su}{\kappa}\r)}{u^{\frac{1}{\s}+1}}\, du 
= \frac{s^{\frac{1}{\s}}}{\kappa^{\frac{1}{\s}}} \int_0^{\infty}\frac{g\l(u\r)}{u^{\frac{1}{\s}+1}}\, du
= \frac{s^{\frac{1}{\s}}}{\kappa^{\frac{1}{\s}}} g_{0}(\s), 
\]
we conclude
\begin{align}
&\sum_{y_1<p\leq{y}_2}g\l(\frac{s}{p^\s(\log{p})^m}\r)\\
&=\s^{\frac{m}{\s}}g_0(\s)\frac{s^{\frac{1}{\s}}}{(\log\kappa)^{\frac{m}{\s}+1}}
\l(1+ O_{\s, m}\l(\frac{1 + m\log \log \kappa}{\log\kappa}
\r) \r).\label{eqMine27T}
\end{align}
Combining \eqref{eqMine24T}, \eqref{eqMine25T}, and \eqref{eqMine27T}, we obtain the asymptotic formula \eqref{eqEndo1T}. 

Let $\kappa$ be large enough depending on $\s$ and $m$.
Then we have 
\[
f_{\s, m}(z; \alpha)
= \s^{\frac{m}{\s}}g_0(\s)\frac{z^{\frac{1}{\s}}}{(\log\kappa)^{\frac{m}{\s}+1}} + h_{\s, m}(z; \a), 
\]
where
\[
h_{\s, m}(z; \a) 
\ll_{\s, m} \frac{\kappa^{\frac{1}{\s}}}{(\log\kappa)^{\frac{m}{\s}+1}}
\frac{1 + m \log\log\kappa}{\log\kappa} 
\]
for $| z - \kappa | \leq \kappa /2$ by the asymptotic formula \eqref{eqEndo1T}.
By Cauchy's integral formula, we have
\begin{align*}
h^{(n)}_{\s, m}(\kappa; \a)
&= f^{(n)}_{\s, m}(\kappa; \a) - \s^{\frac{m}{\s}}G_n(\s) g_0 (\s) \frac{\kappa^{\frac{1}{\s} - n}}{(\log\kappa)^{\frac{m}{\s}+1}} \\
&= \frac{n!}{2\pi{i}}\int_{|z-\kappa|=\kappa/2}\frac{h_{\s, m}(z; \a)}{(z-\kappa)^{n+1}}\, dz \\
&\ll_{\s, m} \frac{ 2^n n! \kappa^{\frac{1}{\s} - n}}{(\log\kappa)^{\frac{m}{\s}+1}}
\frac{1 + m\log\log\kappa}{\log\kappa}, 
\end{align*}
where $G_n(\s) = \prod_{j = 0}^{n -1}(\frac{1}{\s} - j)$, $G_{0}(\s) = 1$.
Using the equation $g_n(\s)=G_n(\s) g_0(\s)$ obtained by integration by parts, 
we have the conclusion.
\end{proof}

\section{\textbf{A transformation of the density function}} \label{Sec_FPDF}

Let $(\s, m)\in\mca{A}$ and $\a\in\RR$. 
We define a non-negative continuous function $\DF_{\s, m}(x; \a)$ as 
\[
\DF_{\s, m}(x; \a)
=\int_{\RR} \DF_{\s, m}(e^{i\a}(x + iy))\frac{dy}{\sqrt{2\pi}}, 
\]
where $\DF_{\s, m}(z)$ is the probability density function determined by \eqref{eqMine14}. 
Then the function $\DF_{\s, m}(x; \a)$ satisfies
\begin{align}
\EXP{\Phi\l(\Re e^{-i\a} \te_{m}(\s, X)\r)}
&=\int_{\CC} \Phi\l(\Re e^{-i\a}z\r)\DF_{\s, m}(z)\, |dz|\\
\label{def_odM}
&=\int_{\RR} \Phi(u) \DF_{\s, m}(u; \a)\, |du|
\end{align}
for all Lebesgue measurable functions $\Phi$ on $\RR$, where $|du|=(2\pi)^{-1/2}du$. 
Here, the second equation is obtained by the change of variables $u = x \cos{\a} + y \sin{\a}$ and $v = -x\sin{\a} + y \cos{\a}$.
Hence $\DF_{\s, m}(x; \a)$ is again a probability density function. 
Additionally, its moment-generating function is given by 
\begin{gather}\label{eqMine29}
F_{\s, m}(s; \a)=\int_\RR{e}^{sx} \DF_{\s, m}(x; \a)\, |dx|
\end{gather}
which agrees with \eqref{eqFY}. 
In this section, we make a transformation to the density function $\DF_{\s, m}(x; \a)$ such that
\begin{gather}\label{eqMine29a}
\nDF_{\s, m}^\tau(x; \a)
=\frac{e^{\kappa\tau}}{F_{\s, m}(\kappa; \a)}e^{\kappa{x}} \DF_{\s, m}(x+\tau; \a)
\end{gather}
with $\tau>0$, where $\kappa=\kappa(\tau; \s, m, \a)$ is a positive real number determined as follows.

\begin{lemma}\label{lemSP}
Let $(\s, m)\in\mca{A}$ and $\a\in\RR$. 
For each $\tau>0$, there exists a unique real number $\kappa=\kappa(\tau; \s, m, \a)>0$ such that 
\begin{gather}\label{eqMine55}
f'_{\s, m}(\kappa; \a)=\tau.
\end{gather}
Furthermore, we have $\kappa\to\infty$ as $\tau\to\infty$. 
\end{lemma}

\begin{proof}
Since $f_{\s, m}(\kappa; \a)=\log{F}_{\s, m}(\kappa; \a)$, we have 
\[
f'_{\s, m}(\kappa; \a)=\frac{F'_{\s, m}(\kappa; \a)}{F_{\s, m}(\kappa; \a)}.
\] 
In particular, we obtain
\[f'_{\s, m}(0; \a)=\frac{F'_{\s, m}(0; \a)}{F_{\s, m}(0; \a)}=\EXP{Y}=0, \]
where we define $Y = \Re e^{-i\a}\te_{m, p}(\s, X(p))$. 
Therefore it is sufficient to show that $f''_{\s, m}(\kappa; \a)>0$ for $\kappa>0$ for the proof of the result. 
Note that we have 
\begin{align*}
f''_{\s, m}(\kappa; \a)
&=\frac{F''_{\s, m}(\kappa; \a)F_{\s, m}(\kappa; \a)-F'_{\s, m}(\kappa; \a)^2}{F_{\s, m}(\kappa; \a)^2}\\
&=\frac{1}{F_{\s, m}(\kappa; \a)}\EXP{\l(Y-f'_{\s, m}(\kappa; \a)\r)^2\exp(\kappa{Y})}>0.
\end{align*}
Hence the result follows. 
\end{proof}

Then, the goal of this section is to show the following asymptotic formula of the function $\nDF_{\s, m}^{\tau}(x; \a)$. 

\begin{proposition}\label{propE}
Let $1/2<\s<1$, $m\in\ZZ_{\geq0}$, and $\a\in\RR$. 
For $\tau>0$, we take $\kappa=\kappa(\tau; \s, m, \a)>0$ satisfying \eqref{eqMine55}. 
Then we have 
\[
\nDF_{\s, m}^\tau(x; \a)
=\frac{1}{\sqrt{{f}''_{\s, m}(\kappa; \a)}}
\l\{\exp\l(-\frac{x^2}{2{f}''_{\s, m}(\kappa; \a)}\r)
+O_{\s, m}\l(\kappa^{-\frac{1}{2\s}}(\log\kappa)^{\frac{1}{2}(\frac{m}{\s}+1)}\r)\r\}\]
for all $x\in\RR$ if $\tau>0$ is large enough. 
\end{proposition}

\subsection{Several inequalities}

\begin{lemma} \label{FCvR}
  Let $s = \kappa + it$.
  There exist positive constants $c_{1} = c_{1}(\s, m)$, $C_{1} = C_{1}(\s, m)$ such that 
  for any $\kappa \geq C_{1}$, $\kappa^{\frac{3}{4\s}} \leq p \leq c_{1}\kappa^{\frac{1}{\s}}(\log{\kappa})^{-\frac{m}{\s}}$, 
  and any $|t| \leq c_{1}\kappa$
  \begin{align*}
    \l| \frac{F_{\s, m, p}(s; \a)}{F_{\s, m, p}(\kappa; \a)} \r|
    \leq \exp\l( -\frac{t^2}{5\kappa^{2}} \r).
  \end{align*}
\end{lemma}

\begin{proof}
  Let $s = \kappa + it$.
  Suppose $\kappa \geq C_{1}$ and $|t| \leq c_{1}\kappa$ with $c_{1}$ small and $C_{1}$ large.
  Let $\kappa^{\frac{3}{4\s}} \leq p \leq c_{1}\kappa^{\frac{1}{\s}}(\log{\kappa})^{-\frac{m}{\s}}$.
  By formula \eqref{BFB0}, it then holds that
  \begin{align*}
    I_{0}(s) = \frac{e^{s}}{\sqrt{2\pi s}}\l( 1 + h_{1}(s) \r).
  \end{align*}
  Here, $h_{1}(s)$ satisfies $h_{1}(s) \ll |s|^{-1}$.
  Moreover, it follows from equation \eqref{eqMine23T} that
  \begin{align*}
    F_{\s, m, p}(s; \a)
    = I_{0}\l( \frac{s}{p^{\s}(\log{p})^{m}} \r)\l( 1 + h_{2}(s; p) \r),
  \end{align*} 
  where $h_{2}(s; p) \ll |s| / p^{2\s} (\log{p})^{m}$.
  Therefore, we have
  \begin{align*}
    F_{\s, m, p}(s; \a)
    = e^{s / p^{\s}(\log{p})^{m}}\sqrt{\frac{p^{\s}(\log{p})^{m}}{2\pi s}}
    \l( 1 + h_{3}(s; p) \r),
  \end{align*}
  and $h_{3}(s; p) \ll \frac{p^{\s}(\log{p})^{m}}{|s|} + \frac{|s|}{p^{2\s} (\log{p})^{m}}$.
  We also see that $h_{3}(s; p) \ll \frac{p^{\s}(\log{p})^{m}}{|s|}$ when $p \geq \kappa^{\frac{3}{4\s}}$, and $\kappa$ is large.
  We note that $h_{3}(s; p)$ is regular for $\Re s > 0$ since $h_{1}, h_{2}$ are regular, and $h_{3}(\kappa; p)$ is real when $\kappa > 0$.
  
  We can write as
  \begin{align*}
    \l| \frac{F_{\s, m, p}(s; \a)}{F_{\s, m, p}(\kappa; \a)} \r| =
    \l| \l( 1 + i\frac{t}{\kappa} \r)^{-\frac{1}{2}} \l( 1 
    + \frac{h_{3}(s; p) - h_{3}(\kappa; p)}{1 + h_{3}(\kappa; p)} \r) \r|.
  \end{align*}
  Hence, when $C_{1}$ is large, $c_{1}$ small, it holds that
  \begin{multline*}
    \l| \frac{F_{\s, m, p}(s; \a)}{F_{\s, m, p}(\kappa; \a)} \r|
    = \exp\bigg( -\frac{1}{4}\frac{t^2}{\kappa^2} 
    + \frac{\Re (h_{3}(s; p) - h_{3}(\kappa; p))}{1 + h_{3}(\kappa; p)}\\
    + O\l( \frac{|t|^3}{\kappa^{3}} + |h_{3}(s; p) - h_{3}(\kappa; p)|^2 \r) \bigg).
  \end{multline*}
  Since $h_{3}(s; p)$ is regular for $\Re s = \kappa > 0$, we have
  \begin{align*}
    h_{3}(s; p) - h_{3}(\kappa; p) = \sum_{n = 1}^{\infty}\frac{h_{3}^{(n)}(\kappa; p)}{n!}(it)^{n}.
  \end{align*}
  Using Cauchy's integral formula and the estimate $h_{3}(s; p) \ll \frac{p^{\s}(\log{p})^{m}}{|s|}$, 
  we see that $h_{3}^{(n)}(\kappa; p) \ll 2^{n}n! \kappa^{-n} \frac{p^{\s}(\log{p})^{m}}{|s|}$.
  Therefore, we have
  \begin{gather*}
    \frac{\Re (h_{3}(s; p) - h_{3}(\kappa; p))}{1 + h_{3}(\kappa; p)}
    \ll \frac{p^{\s} (\log{p})^{m}}{\kappa} \frac{t^2}{\kappa^2}, \\
    |h_{3}(s; p) - h_{3}(\kappa; p)|^2
    \ll \frac{p^{2\s}(\log{p})^{2m}}{\kappa^2} \frac{t^2}{\kappa^2}.
  \end{gather*}
  Hence, we obtain
  \begin{align*}
    \l| \frac{F_{\s, m, p}(s; \a)}{F_{\s, m, p}(\kappa; \a)} \r|
    \leq \exp\l( -\frac{1}{5}\frac{t^2}{\kappa^2} \r)
  \end{align*}
  when $c_{1}$ is sufficiently small.
  Thus, we complete the proof of this lemma.
\end{proof}

\begin{lemma}\label{lemBound1td}
Let $(\s, m) \in \mathcal{A}$, and let $\a \in \RR$. 
Suppose $s = \kappa+it$ with $\kappa$ sufficiently large and $t \in \RR$. 
There exists a positive constant $c_{2} = c_{2}(\s, m)$ such that 
for any $|t| \leq c_{2} \kappa$
\begin{align}
  \label{lemBound1td1}
  \frac{|F_{\s, m}(s; \a)|}{F_{\s, m}(\kappa; \a)} 
  \leq \exp\l(- c_{2} t^2\kappa^{\frac{1}{\s} - 2} (\log{\kappa})^{-\frac{m}{\s} - 1}\r),
\end{align}
and for any $|t| \geq c_{2} \kappa$
\begin{align}
  \label{lemBound1td2}
  \frac{|F_{\s, m}(s; \a)|}{F_{\s, m}(\kappa; \a)} 
  \leq \exp\l(-|t|^{1/(2\s)}\r).
\end{align}
\end{lemma}

\begin{proof}
We first prove \eqref{lemBound1td1}.
Suppose that $\kappa$ is large and $|t| \leq c_{1} \kappa$, 
where $c_{1}$ is the same constants as in Lemma \ref{FCvR}.
Put $L_{1} = c_{1}\kappa^{1/\s} (\log{\kappa})^{-m/\s}$.
Using Lemma \ref{FCvR} and the prime number theorem, we obtain
\begin{align*}
  \frac{|F_{\s, m}(s; \a)|}{F_{\s, m}(\kappa; \a)}
  &\leq \prod_{\frac{L_{1}}{2} \leq p \leq L_{1}}\exp\l( -\frac{1}{5}\frac{t^2}{\kappa^2} \r)
  \leq \exp\l(- c_{2} t^2\kappa^{\frac{1}{\s} - 2} (\log{\kappa})^{-\frac{m}{\s} - 1}\r),
\end{align*}
which is the assertion of \eqref{lemBound1td1}.

Next, we prove \eqref{lemBound1td2}.
We see from \eqref{eqMine3T} that
\begin{align}
\label{plemBound1td1}
\frac{|F_{\s, m}(s; \a)|}{F_{\s, m}(\kappa; \a)} 
\leq \prod_{p \geq M_{1} |t|^{1/\s}}
\frac{
\l|\EXP{\exp(s \Re e^{-i\a} \te_{m, p}(\s, X(p)) )}\r|}{
\EXP{\exp(\kappa \Re e^{-i\a} \te_{m, p}(\s, X(p)))}}
\end{align}
with $M_{1}$ a suitably large constant may depend on $\s$ and $m$.
We suppose that $s=\kappa+it$ satisfies $|t| \geq c_{2}\kappa$, then $|s| \leq (1 + 1/c_{2})|t|$. 
By the Taylor expansion of $\exp(z)$, we obtain 
\begin{align*}
\exp(s \Re e^{-i\a} \te_{m, p}(\s, X(p)) )
&= 1 + s\Re e^{-i\a} \te_{m, p}(\s, X(p))\\
&+ \frac{1}{2}\l(s\Re e^{-i\a} \te_{m, p}(\s, X(p)) \r)^{2}
+O\l(\frac{|t|^3}{p^{3\s}(\log{p})^{3m}}\r)
\end{align*}
for $p> M_{1} |t|^{1/\s}$ when $M_{1}$ is suitably large. 
We can easily see that
\begin{align*}
\EXP{\Re e^{-i\a} \te_{m, p}(\s, X(p))}
=0, 
\end{align*}
and
\begin{align*}
\EXP{\l(\Re e^{-i\a} \te_{m, p}(\s, X(p)) \r)^2}
&= \frac{1}{2}\frac{\Li_{2m+2}(p^{-2\s})}{(\log{p})^{2m}}.
\end{align*}
By these formulas, it follows that
\begin{align*}
&\EXP{\exp(s \Re e^{-i\a} \te_{m, p}(\s, X(p)) )}\\
&= 1 + \frac{s^{2}}{4}\frac{\Li_{2m+2}(p^{-2\s})}{(\log{p})^{2m}}
+O\l(\frac{|t|^3}{p^{3\s}(\log{p})^{3m}}\r).
\end{align*}
Therefore, we have
\begin{align*}
&\frac{
\l|\EXP{\exp(s \Re e^{-i\a} \te_{m, p}(\s, X(p)) )}\r|}{
\EXP{\exp(\kappa \Re e^{-i\a} \te_{m, p}(\s, X(p)) )}}\\
&= \bigg|1 + \frac{2i\kappa t - t^2}{4}\frac{\Li_{2m+2}(p^{-2\s})}{(\log{p})^{2m}}
+O\l(\frac{|t|^3}{p^{3\s}(\log{p})^{3m}}\r)\bigg|
\end{align*}
for $p \geq M_{1}|t|^{1/\s}$.
In particular, when $M_{1}$ is sufficiently large, it also holds that
\begin{align*}
\l| \frac{2i\kappa t - t^2}{4}\frac{\Li_{2m+2}(p^{-2\s})}{(\log{p})^{2m}}
+O\l(\frac{|t|^3}{p^{3\s}(\log{p})^{3m}}\r) \r|
\leq \frac{1}{2}.
\end{align*}
From these results and inequality \eqref{plemBound1td1}, we obtain
\begin{align*}
&\frac{|F_{\s, m}(s; \a)|}{F_{\s, m}(\kappa; \a)} 
\\
&\leq \exp\Bigg( \sum_{p \geq M_{1}|t|^{1/\s}}
\Re \log\Bigg( 1 + \frac{2i\kappa t - t^2}{4}\frac{\Li_{2m+2}(p^{-2\s})}{(\log{p})^{2m}}\\
&\qquad\qquad\qquad\qquad\qquad\qquad\qquad\qquad\qquad\qquad
+O\l(\frac{|t|^3}{p^{3\s}(\log{p})^{3m}}\r) \Bigg) \Bigg)\\
&= \exp\Bigg( \sum_{p \geq M_{1}|t|^{1/\s}}
\Re \Bigg( \frac{2i\kappa t - t^2}{4}\frac{\Li_{2m+2}(p^{-2\s})}{(\log{p})^{2m}} \\
&\qquad\qquad\qquad\qquad\qquad\qquad\qquad\qquad\qquad\qquad
+O\l(\frac{|t|^3}{p^{3\s}(\log{p})^{3m}}\r) \Bigg)  \Bigg)\\
&\leq \exp\l( \sum_{p \geq M_{1}|t|^{1/\s}}\l(\frac{- t^2}{4p^{2\s}(\log{p})^{2m}}
+O\l(\frac{|t|^3}{p^{3\s}(\log{p})^{3m}}\r) \r) \r)\\
&\leq \exp\l( -t^2\sum_{p \geq M_{1}|t|^{1/\s}}\frac{1}{8p^{2\s}(\log{p})^{2m}}\r)
\leq \exp\l( -|t|^{1/2\s} \r)
\end{align*}
when $M_{1}$ is sufficiently large.
Thus, we also obtain inequality \eqref{lemBound1td2}.
\end{proof}

\subsection{Fourier analysis for the function $\nDF_{\s, m}^{\tau}(x; \a)$}\label{secTrans}

\begin{lemma}\label{lemE}
Let $(\s, m)\in\mca{A}$ and $\a\in\RR$. 
For $\tau>0$, the function $\nDF_{\s, m}^{\tau}(x; \a)$ is a probability density function, whose Fourier transform is given by
\[
\widetilde{\nDF}_{\s, m}^\tau(t; \a)
:=\int_{-\infty}^{\infty}{e}^{itx}\nDF_{\s, m}^\tau(x; \a)\, |dx|
={e}^{-it\tau}\frac{F_{\s, m}(\kappa+it; \a)}{F_{\s, m}(\kappa; \a)}.\]
\end{lemma}

\begin{proof}
By the definition, we have 
\[\int_{\infty}^{\infty} \nDF_{\s, m}^\tau(x; \a)\, |dx|
=\frac{1}{F_{\s, m}(\kappa; \a)}\int_{-\infty}^\infty{e}^{\kappa(x+\tau)}\DF_{\s, m}(x+\tau; \a)\, |dx|=1\]
due to \eqref{eqMine29}. 
Similarly, the Fourier transform of $\nDF_{\s, m}^{\tau}(x; \a)$ is calculated as 
\begin{align*}
\widetilde{\nDF}_{\s, m}^{\tau}(t; \a)
&=\frac{1}{F_{\s, m}(\kappa; \a)}\int_{-\infty}^\infty{e}^{\kappa(x+\tau)+itx}\DF_{\s, m}(x+\tau; \a)\, |dx|\\
&=\frac{1}{F_{\s, m}(\kappa; \a)}\int_{-\infty}^\infty{e}^{\kappa{x}+it(x-\tau)}\DF_{\s, m}(x; \a)\, |dx|\\
&={e}^{-it\tau}\frac{F_{\s, m}(\kappa+it; \a)}{F_{\s, m}(\kappa; \a)}, 
\end{align*}
which completes the proof.
\end{proof}

By Lemma \ref{lemBound1td}, we find that $\widetilde{\nDF}_{\s, m}^{\tau}(t; \a)$ is absolutely integrable over $\RR$. 
Hence the function $\nDF_{\s, m}^\tau(x; \a)$ can be recovered by the inversion formula
\begin{gather}\label{eqMine30}
\nDF_{\s, m}^{\tau}(x; \a)
=\int_\RR\widetilde{\nDF}_{\s, m}^\tau(t; \a)e^{-itx}\, |dt|.
\end{gather}
Finally, we apply \eqref{eqMine30} to show Proposition \ref{propE}. 

\begin{proof}[Proof of Proposition \ref{propE}]
Put $\kappa_1=\kappa^{1-\frac{1}{3\sigma}}$ and $\kappa_2=c_2 \kappa$ with the positive constant $c_2=c_2(\sigma,m)$ of Lemma \ref{lemBound1td}. 
First, we divide the integral of \eqref{eqMine30} as 
\begin{align}\label{eqMine31}
\nDF_{\s, m}^{\tau}(x; \a)
&=\left( \int_{|t| \leq \kappa_1} + \int_{\kappa_1<|t|<\kappa_2} + \int_{|t|\geq \kappa_2} \right)
\widetilde{\nDF}_{\s, m}^{\tau}(t; \a)e^{-itx}\, |dt| \nonumber \\
&=\int_{-\kappa_1}^{\kappa_1} \widetilde{\nDF}_{\s, m}^{\tau}(t; \a)e^{-itx}\, |dt|+E_1+E_2, 
\end{align}
say. 
We evaluate the integral $E_1$ as 
\begin{align*}
E_1
&\ll \int_{\kappa_1}^{\kappa_2} \exp\l(- c_{2} t^2\kappa^{\frac{1}{\s} - 2} (\log{\kappa})^{-\frac{m}{\s} - 1}\r) \,dt \\
&\ll \exp\l(- \frac{c_{2}}{2} \kappa_1^2\kappa^{\frac{1}{\s} - 2} (\log{\kappa})^{-\frac{m}{\s} - 1}\r)
\int_{0}^{\infty} \exp\l(-\frac{c_{2}}{2} t^2\kappa^{\frac{1}{\s} - 2} (\log{\kappa})^{-\frac{m}{\s} - 1}\r) \,dt \\
&\ll \kappa^{1-\frac{1}{2\sigma}} (\log{\kappa})^{\frac{1}{2}(\frac{m}{\sigma}+1)} 
\exp\l(- \frac{c_{2}}{2} \kappa^{\frac{1}{3\s}} (\log{\kappa})^{-\frac{m}{\s} - 1}\r) 
\end{align*}
by using Lemmas \ref{lemBound1td} and \ref{lemE}. 
Similarly, we derive
\begin{align*}
E_2
\ll \int_{\kappa_2}^{\infty} \exp\l(-t^{1/(2\s)}\r) \,dt 
\ll \sqrt{\kappa} \exp\left(-\sqrt{\kappa_2}\right). 
\end{align*}
Recall that Proposition \ref{propCumu1} gives the estimate
\begin{gather}\label{eq:Mine30a}
\frac{1}{\sqrt{f''_{\sigma,m}(\kappa;\alpha)}}
\asymp_{\sigma,m} \kappa^{1-\frac{1}{2\sigma}} (\log{\kappa})^{\frac{1}{2}(\frac{m}{\sigma}+1)}. 
\end{gather}
Therefore, \eqref{eqMine31} yields the formula
\begin{gather}\label{eqMine31a}
\nDF_{\s, m}^{\tau}(x; \a)
=\int_{-\kappa_1}^{\kappa_1} \widetilde{\nDF}_{\s, m}^{\tau}(t; \a)e^{-itx}\, |dt|
+O_{\sigma,m}\left(\frac{\kappa^{-\frac{1}{2\s}}(\log\kappa)^{\frac{1}{2}(\frac{m}{\s}+1)}}
{\sqrt{f''_{\sigma,m}(\kappa;\alpha)}} \right). 
\end{gather}
To estimate the remaining integral, we define a holomorphic function $G(z)$ as  
\begin{align}\label{eqG}
G(z)
&=\exp\l(-\tau{z}-\frac{{f}''_{\s, m}(\kappa; \a)}{2}z^2\r)\frac{F_{\s, m}(z+\kappa; \a)}{F_{\s, m}(\kappa; \a)}. 
\end{align}
The function $G(z)$ is also expressed as 
\begin{align*}
G(z)
&=\exp\l(f_{\s, m}(z+\kappa; \a)-f_{\s, m}(\kappa; \a)-f'_{\s, m}(\kappa; \a)z-\frac{f''_{\s, m}(\kappa; \a)}{2}z^2\r) \\
&=\exp\l(\sum_{n=3}^\infty\frac{f^{(n)}_{\s, m}(\kappa; \a)}{n!}z^n\r)
\end{align*}
near the origin. 
Hence, we can write
\begin{align} \label{TEG}
  G(z)
  =1+\sum_{n=3}^\infty\frac{a_n}{n!}z^n,
\end{align}
where the coefficients $a_n$ are calculated as 
\[a_n
=\sum_{k=1}^{\lfloor{n}/3\rfloor}
\frac{1}{k!}\sum_{\substack{n_1+\cdots+{n}_k=n\\ \forall{j}, ~ n_j\geq3}}
\binom{n}{n_1, \ldots, n_k}f_{\s, m}^{(n_1)}(\kappa; \a)\cdots{f}_{\s, m}^{(n_k)}(\kappa; \a).\]
Then, by Lemma \ref{lemE}, we have 
\[\widetilde{\nDF}_{\s, m}^{\tau}(t; \a)
=\exp\l(-\frac{f''_{\s, m}(\kappa; \a)}{2}t^2\r)G(it). \]
Hence we obtain 
\begin{gather}\label{eqMine32}
\int_{-\kappa_1}^{\kappa_1}\widetilde{\nDF}_{\s, m}^{\tau}(t; \a)e^{-itx}\, |dt|
=\int_{-\kappa_1}^{\kappa_1}\exp\l(-\frac{f''_{\s, m}(\kappa; \a)}{2}t^2\r)e^{-itx}\, |dt|+E_3,
\end{gather}
where 
\begin{align*}
E_3
&=\int_{-\kappa_1}^{\kappa_1}\exp\l(-\frac{f''_{\s, m}(\kappa; \a)}{2}t^2\r)\l(G(it)-1\r)e^{-itx}\, |dt|\\
&\ll\int_0^{\kappa_1}\exp\l(-\frac{f''_{\s, m}(\kappa; \a)}{2}t^2\r)\l(\sum_{n=3}^\infty\frac{|a_n|}{n!}t^n\r)\, dt. 
\end{align*}
We further evaluate the error term $E_3$ as follows. 
Notice that we have 
\[
\sum_{n=3}^\infty\frac{|{a}_n|}{n!}t^n
\leq \exp\l(\sum_{n=3}^\infty\frac{\l|f_{\s, m}^{(n)}(\kappa; \a)\r|}{n!}t^n\r)-1. 
\]
for $0\leq{t}\leq \kappa_1$. 
Furthermore, it is deduced from Proposition \ref{propCumu1} and $g_{n}(\s) \ll_{\s} n!$ that 
\[\sum_{n=3}^\infty\frac{\l|f_{\s, m}^{(n)}(\kappa; \a)\r|}{n!}t^n
\ll_{\s, m}\frac{\kappa^{\frac{1}{\s}}}{(\log\kappa)^{\frac{m}{\s}+1}}
\sum_{n=3}^\infty\l(\frac{2t}{\kappa}\r)^n
\ll\frac{\kappa^{\frac{1}{\s}-3}}{(\log\kappa)^{\frac{m}{\s}+1}}t^3. \]
Here, we remark that $\kappa^{\frac{1}{\s}-3} (\log\kappa)^{-(\frac{m}{\s}+1)} t^3$ is bounded for $0\leq{t}\leq \kappa_1$. 
Hence there exists a constant $C_{\s, m}>0$ such that 
\begin{align}\label{eq_an}
\sum_{n=3}^\infty\frac{|{a}_n|}{n!}t^n
&\leq\sum_{n=1}^\infty\frac{1}{n!}
\l(C_{\s, m}\frac{\kappa^{\frac{1}{\s}-3}}{(\log\kappa)^{\frac{m}{\s}+1}}t^3\r)^n \nonumber\\
&\ll_{\s, m} \frac{\kappa^{\frac{1}{\s}-3}}{(\log\kappa)^{\frac{m}{\s}+1}}t^3, 
\end{align}
which deduces
\begin{align*}
E_3
&\ll_{\s, m} \frac{\kappa^{\frac{1}{\s}-3}}{(\log\kappa)^{\frac{m}{\s}+1}}
\int_0^{\infty} \exp\l(-\frac{f''_{\s, m}(\kappa; \a)}{2}t^2\r)t^{3}\, dt\\
&\ll\frac{1}{\sqrt{f''_{\s, m}(\kappa; \a)}}
\frac{\kappa^{\frac{1}{\s}-3}}{f''_{\s, m}(\kappa; \a)^{3/2}(\log\kappa)^{\frac{m}{\s}+1}}.
\end{align*}
We use \eqref{eq:Mine30a} to see that 
\[\frac{\kappa^{\frac{1}{\s}-3}}{f''_{\s, m}(\kappa; \a)^{3/2}(\log\kappa)^{\frac{m}{\s}+1}}
\ll_{\s, m}\kappa^{-\frac{1}{2\s}}(\log\kappa)^{\frac{1}{2}(\frac{m}{\s}+1)}. \]
Therefore we arrive at the estimate 
\begin{gather}\label{eqMine31b}
E_3
\ll_{\s, m}\frac{\kappa^{-\frac{1}{2\s}}(\log\kappa)^{\frac{1}{2}(\frac{m}{\s}+1)}}
{\sqrt{f''_{\sigma,m}(\kappa;\alpha)}}
\end{gather}
if $\kappa>0$ is large enough. 
Finally, we calculate the integral in the right-hand side of \eqref{eqMine32} as 
\begin{align}\label{eqMine33}
&\int_{-\kappa_1}^{\kappa_1}\exp\l(-\frac{f''_{\s, m}(\kappa; \a)}{2}t^2\r)e^{-itx}\, |dt|\\
&=\int_{-\infty}^\infty\exp\l(-\frac{f''_{\s, m}(\kappa; \a)}{2}t^2\r)e^{-itx}\, |dt|+E_4\nonumber\\
&=\frac{1}{\sqrt{{f}''_{\s, m}(\kappa; \a)}}\exp\l(-\frac{x^2}{2{f}''_{\s, m}(\kappa; \a)}\r)+E_4, 
\end{align}
where
\begin{align*}
E_4
&\ll\int_{\kappa_1}^\infty\exp\l(-\frac{f''_{\s, m}(\kappa; \a)}{2}t^2\r)\, dt \\
&\ll \frac{1}{\sqrt{{f}''_{\s, m}(\kappa; \a)}}
\exp\l(-\frac{f''_{\s, m}(\kappa; \a)}{4}\kappa_1^2\r) \\
&\ll_{\sigma,m} \frac{\kappa^{-\frac{1}{2\s}}(\log\kappa)^{\frac{1}{2}(\frac{m}{\s}+1)}}
{\sqrt{f''_{\sigma,m}(\kappa;\alpha)}}. 
\end{align*}
Combining \eqref{eqMine31a}, \eqref{eqMine32}, \eqref{eqMine31b} and \eqref{eqMine33}, 
we derive the desired asymptotic formula of $\nDF_{\s, m}^{\tau}(x; \a)$. 
\end{proof}

\section{\textbf{Large deviations: Proof of Theorem \ref{thmLD}}} \label{Sec_PLD}

\subsection{Preliminaries}

In this subsection, we prove the following proposition.

\begin{proposition}\label{propFT}
Let $1/2<\s<1$, $m\in\ZZ_{\geq0}$, and $\a\in\RR$. 
For $\tau>0$, we take $\kappa=\kappa(\tau; \s, m, \a)>0$ satisfying \eqref{eqMine55}. 
Then we have 
\begin{align}\label{eqMine56}
&\PP\l(\Re e^{-i\a} \te_{m}(\s, X) > \tau\r) \\
&=\frac{F_{\s, m}(\kappa; \a)e^{-\tau\kappa}}{\kappa\sqrt{2\pi{f}''_{\s, m}(\kappa; \a)}} 
\l\{1+O_{\s, m}\l(\kappa^{-\frac{1}{2\s}}(\log\kappa)^{\frac{1}{2}(\frac{m}{\sigma}+1)}\r)\r\} 
\end{align}
if $\tau>0$ is large enough. 
\end{proposition}

We prepare some lemmas toward the proof of Proposition \ref{propFT}.

\begin{lemma}[Granville--Soundararajan \cite{GS2003}]\label{lemGS}
Let $\lambda>0$ be a real number.
For $y>0$ and $c>0$, we have 
\begin{align*}
0&\leq\frac{1}{2\pi{i}}\int_{c-i\infty}^{c+i\infty}y^s\l(\frac{e^{\lambda{s}}-1}{\lambda{s}}\r)\frac{ds}{s}-\chi(y)\\
&\leq\frac{1}{2\pi{i}}\int_{c-i\infty}^{c+i\infty}y^s\l(\frac{e^{\lambda{s}}-1}{\lambda{s}}\r)\l(\frac{1-e^{-\lambda{s}}}{s}\r)\, ds, 
\end{align*}
where $\chi(y)=1$ if $y>1$ and $\chi(y)=0$ otherwise. 
\end{lemma}

\begin{lemma}\label{lemFT1}
With the assumptions of Proposition \ref{propFT}, we have 
\begin{align*}
&\frac{1}{2\pi{i}}\int_{\kappa-i\kappa_1}^{\kappa+i\kappa_1}
F_{\s, m}(s; \a)e^{-\tau{s}}
\l(\frac{e^{\lambda{s}}-1}{\lambda{s^2}}\r)\, ds\\
&=\frac{F_{\s, m}(\kappa; \a)e^{-\tau\kappa}}{\kappa\sqrt{2\pi{f}''_{\s, m}(\kappa; \a)}}
\l\{1+O_{\s, m}\l(\kappa^{-\frac{1}{2\s}}(\log\kappa)^{\frac{1}{2}(\frac{m}{\sigma}+1)}\r)\r\}, 
\end{align*}
where $\lambda=\kappa^{-3}$ and $\kappa_1=\kappa^{1-\frac{1}{3\sigma}}$. 
\end{lemma}

\begin{proof}
Let $G(z)$ be the function defined as \eqref{eqG}. 
Then we see that 
\[F_{\s, m}(\kappa+it; \a)e^{-\tau(\kappa+it)}
=F_{\s, m}(\kappa; \a)e^{-\tau\kappa}\exp\l(-\frac{f''_{\s, m}(\kappa; \a)}{2}t^2\r)G(it) \]
by the definition.
Furthermore we obtain
\[\frac{e^{\lambda(\kappa+it)}-1}{\lambda(\kappa+it)^2}
=\frac{1}{\kappa}\l(1-i\frac{t}{\kappa}+O\l(\lambda\kappa+\frac{t^2}{\kappa^2}\r)\r) \]
for $|t|\leq\kappa_1$. 
Hence the integral is calculated as
\begin{align*}
&\frac{1}{2\pi i}\int_{\kappa-i\kappa_1}^{\kappa+i\kappa_1}
F_{\s, m}(s; \a)e^{-\tau{s}}
\l(\frac{e^{\lambda{s}}-1}{\lambda{s^2}}\r)\, ds\\
&=\frac{F_{\s, m}(\kappa; \a)e^{-\tau\kappa}}{\kappa}
\frac{1}{2\pi }\int_{-\kappa_1}^{\kappa_1}
\exp\l(-\frac{f''_{\s, m}(\kappa; \a)}{2}t^2\r)\\
&\qquad\qquad\qquad\qquad\qquad\qquad\qquad
 \times G(it)\l(1-i\frac{t}{\kappa}+O\l(\lambda\kappa+\frac{t^2}{\kappa^2}\r)\r)\, dt.
\end{align*}
Using \eqref{TEG}, we find that this is equal to
\begin{align*}
&\frac{F_{\s, m}(\kappa; \a)e^{-\tau\kappa}}{\kappa}
\Bigg\{\frac{1}{2\pi}\int_{-\kappa_1}^{\kappa_1}\exp\l(-\frac{f''_{\s, m}(\kappa; \a)}{2}t^2\r)\, dt\\
&\qquad+\frac{1}{2\pi}\int_{-\kappa_1}^{\kappa_1}\exp\l(-\frac{f''_{\s, m}(\kappa; \a)}{2}t^2\r)\l(-i\frac{t}{\kappa}+O\l(\lambda\kappa+\frac{t^2}{\kappa^2}\r)\r)\, dt\\
&\qquad+\frac{1}{2\pi}\int_{-\kappa_1}^{\kappa_1}\exp\l(-\frac{f''_{\s, m}(\kappa; \a)}{2}t^2\r)
\l(\sum_{n=3}^\infty\frac{a_n}{n!}(it)^n\r)\\
&\qquad\qquad\qquad\qquad\qquad
\times\l(1-i\frac{t}{\kappa}+O\l(\lambda\kappa+\frac{t^2}{\kappa^2}\r)\r)\, dt\Bigg\}\\
&=\frac{F_{\s, m}(\kappa; \a)e^{-\tau\kappa}}{\kappa}(I_1+I_2+I_3), 
\end{align*}
say. 
We have 
\[I_1=\frac{1}{\sqrt{2\pi{f}''_{\s, m}(\kappa; \a)}}
\l\{1+O\l(\kappa_{1}^{-1}\exp\l(-\frac{{f}''_{\s, m}(\kappa; \a)}{2}\kappa_1^2\r)\r)\r\},\]
and furthermore,
\begin{align*}
I_2
&=\frac{1}{2\pi}\int_{-\kappa_1}^{\kappa_1}\exp\l(-\frac{f''_{\s, m}(\kappa; \a)}{2}t^2\r)
O\l(\lambda\kappa+\frac{t^2}{\kappa^2}\r) \, dt \\
&\ll\frac{1}{\sqrt{{f}''_{\s, m}(\kappa; \a)}}\l(\lambda\kappa+\frac{1}{\kappa^2 {f}''_{\s, m}(\kappa; \a)}\r). 
\end{align*}
Since $\lambda=\kappa^{-3}$, the contributions of $I_1$ and $I_2$ are evaluated as 
\[\frac{F_{\s, m}(\kappa; \a)e^{-\tau\kappa}}{\kappa}(I_1+I_2)
=\frac{F_{\s, m}(\kappa; \a)e^{-\kappa\tau}}{\kappa\sqrt{2\pi{f}''_{\s, m}(\kappa; \a)}}
\l\{1+O_{\s, m}\l(\kappa^{-\frac{1}{\s}}(\log\kappa)^{\frac{m}{\s}+1}\r)\r\}\]
by using \eqref{eq:Mine30a}. 
As for $I_3$, we use upper bounds \eqref{eq:Mine30a} and \eqref{eq_an} to derive  
\begin{align*}
I_3
&\ll_{\s, m} \frac{\kappa^{\frac{1}{\s}-3}}{(\log\kappa)^{\frac{m}{\s}+1}}
\int_0^{\infty} \exp\l(-\frac{f''_{\s, m}(\kappa; \a)}{2}t^2\r)t^{3}\, dt\\
&\ll\frac{1}{\sqrt{f''_{\s, m}(\kappa; \a)}}
\frac{\kappa^{\frac{1}{\s}-3}}{f''_{\s, m}(\kappa; \a)^{3/2}(\log\kappa)^{\frac{m}{\s}+1}} \\
&\ll_{\s, m} \frac{\kappa^{-\frac{1}{2\s}}(\log\kappa)^{\frac{1}{2}(\frac{m}{\s}+1)}}{\sqrt{f''_{\s, m}(\kappa; \a)}}. 
\end{align*}
Hence the desired result follows. 
\end{proof}

\begin{proof}[Proof of Proposition \ref{propFT}]
Recall that Lemma \ref{lemGS} with $y = \exp(\Re e^{-i\a} \te_{m}(\s, X) - \tau)$ deduces the inequality
\begin{align*}
0&\leq\frac{1}{2\pi{i}}\int_{c-i\infty}^{c+i\infty}F_{\s, m}(s; \a)e^{-\tau{s}}
\l(\frac{e^{\lambda{s}}-1}{\lambda{s^2}}\r)\, ds
-\PP\l(\Re e^{-i\a} \te_{m}(\s, X) > \tau \r)\\
&\leq\frac{1}{2\pi{i}}\int_{c-i\infty}^{c+i\infty}F_{\s, m}(s; \a)e^{-\tau{s}}
\l(\frac{e^{\lambda{s}}-1}{\lambda{s}}\r)\l(\frac{1-e^{-\lambda{s}}}{s}\r)\, ds. 
\end{align*}
We put $\kappa_1=\kappa^{1-\frac{1}{3\sigma}}$ and $\kappa_2=c_2 \kappa$ with the positive constant $c_2=c_2(\sigma,m)$ of Lemma \ref{lemBound1td}. 
Then, by Lemma \ref{lemFT1}, the desired result follows if the integrals
\begin{gather}\label{eqMine57}
E_1
=\frac{1}{2\pi{i}}\int_{\kappa\pm{i}\kappa_1}^{\kappa\pm{i}\infty}
F_{\s, m}(s; \a)e^{-\tau{s}}
\l(\frac{e^{\lambda{s}}-1}{\lambda{s^2}}\r)\, ds, \\
E_2
=\frac{1}{2\pi{i}}\int_{\kappa-i\infty}^{\kappa+i\infty}
F_{\s, m}(s; \a)e^{-\tau{s}}
\l(\frac{e^{\lambda{s}}-1}{\lambda{s}}\r)\l(\frac{1-e^{-\lambda{s}}}{s}\r)\, ds
\end{gather}
are evaluated as
\begin{gather}\label{eq:E_j}
E_j
\ll_{\s, m} \frac{F_{\s, m}(\kappa; \a)e^{-\kappa\tau}}{\kappa\sqrt{{f}''_{\s, m}(\kappa; \a)}}
\kappa^{-\frac{1}{2\sigma}}(\log\kappa)^{\frac{1}{2}(\frac{m}{\s}+1)}. 
\end{gather}
First, we divide the integral of $E_1$ as 
\begin{align*}
E_1
&=\frac{1}{2\pi{i}}
\left(\int_{\kappa\pm{i}\kappa_1}^{\kappa\pm{i}\kappa_2}
+\int_{\kappa\pm{i}\kappa_2}^{\kappa\pm{i}\infty}\right)
F_{\s, m}(s; \a)e^{-\tau{s}}
\l(\frac{e^{\lambda{s}}-1}{\lambda{s^2}}\r)\, ds \\
&=E_{1,1}+E_{1,2}, 
\end{align*}
say. 
Since we take $\lambda=\kappa^{-3}$, we obtain $|e^{\lambda{s}}-1| \leq 4$. 
It yields the upper bound
\begin{gather*}
\frac{e^{\lambda{s}}-1}{\lambda{s^2}}
\ll\frac{\kappa^3}{t^2}. 
\end{gather*}
Then we deduce from Lemma \ref{lemBound1td} the estimates
\begin{align*}
E_{1,1}
&\ll{F}_{\s, m}(\kappa; \a)e^{-\tau\kappa}
\int_{\kappa_1}^{\infty}
\exp\l(- c_{2} t^2\kappa^{\frac{1}{\s} - 2} (\log{\kappa})^{-\frac{m}{\s} - 1}\r)
\frac{\kappa^3}{t^2}\, dt \\
&\ll{F}_{\s, m}(\kappa; \a)e^{-\tau\kappa}
\exp\l(- c_{2} \kappa^{\frac{1}{3\s}} (\log{\kappa})^{-\frac{m}{\s} - 1}\r)
\kappa^{2+\frac{1}{3\sigma}} \\
\end{align*}
and 
\begin{align*}
E_{1,2}
&\ll{F}_{\s, m}(\kappa; \a)e^{-\tau\kappa}
\int_{\kappa_2}^{\infty}\exp(-t^{1/(2\sigma)})\frac{\kappa^3}{t^2}\, dt \\
&\ll{F}_{\s, m}(\kappa; \a)e^{-\tau\kappa}\exp\l(-c_2\kappa^{1/(2\sigma)}\r) \kappa^2.  
\end{align*}
Using \eqref{eq:Mine30a}, we conclude that \eqref{eq:E_j} is satisfied for $j=1$. 
Next, we divide the integral of $E_2$ as 
\begin{align*}
E_2 
&=\frac{1}{2\pi{i}}\l(\int_{\kappa-{i}\kappa_2}^{\kappa+{i}\kappa_2}
+\int_{\kappa\pm{i}\kappa_2}^{\kappa\pm{i}\infty}\r) F_{\s, m}(s; \a)e^{-\tau{s}}
\l(\frac{e^{\lambda{s}}-1}{\lambda{s}}\r)\l(\frac{1-e^{-\lambda{s}}}{s}\r)\, ds\\
&=E_{2,1}+E_{2,2}, 
\end{align*}
say. 
Then, we use the inequalities $|e^{\lambda{s}}-1| \leq 4$ and $|1-e^{-\lambda s}| \leq 2$ to derive
\[
\l(\frac{e^{\lambda{s}}-1}{\lambda{s}}\r)\l(\frac{1-e^{-\lambda{s}}}{s}\r)
\ll
\begin{cases}
\kappa^{-3} & \text{if $|t|\leq\kappa_2$}, \\
\kappa^3t^{-2} & \text{if $|t|>\kappa_2$}.
\end{cases}
\]
Therefore, we evaluate $E_{2,1}$ as
\begin{gather*}
E_{2,1}
\ll{F}_{\s, m}(\kappa; \a)e^{-\tau\kappa}
\int_{-\kappa_2}^{\kappa_2}\kappa^{-3}\, dt
\ll{F}_{\s, m}(\kappa; \a)e^{-\tau\kappa}\kappa^{-2}. 
\end{gather*}
The integral $E_{2,2}$ can be estimated similarly to $E_{1,2}$. 
We have 
\begin{gather*}
E_{2,2}
\ll{F}_{\s, m}(\kappa; \a)e^{-\tau\kappa}\exp\l(-c_2\kappa^{1/(2\sigma)}\r) \kappa^2.  
\end{gather*}
As a result, we deduce
\begin{align*}
E_2
&\ll{F}_{\s, m}(\kappa; \a)e^{-\tau\kappa}\kappa^{-2} \\
&\ll_{\s, m} \frac{F_{\s, m}(\kappa; \a)e^{-\kappa\tau}}{\kappa\sqrt{{f}''_{\s, m}(\kappa; \a)}}
\kappa^{-\frac{1}{2\sigma}}(\log\kappa)^{\frac{1}{2}(\frac{m}{\s}+1)}
\end{align*}
by estimate \eqref{eq:Mine30a}. 
Thus \eqref{eq:E_j} is satisfied for $j=2$, and the result follows. 
\end{proof}

\begin{corollary}\label{corLD}
Let $1/2<\s<1$, $m\in\ZZ_{\geq0}$, and $\a\in\RR$. 
For large $\tau>0$, we have 
\begin{align*}
&\PP\l(\Re {e^{-i\a} \te_m(\s, X)} > \tau\r)\\
&=\exp\l(-A_m(\s)\tau^{\frac{1}{1-\s}}(\log\tau)^{\frac{m+\s}{1-\s}}\l(1+O_{\s, m}\l(\frac{\log\log\tau}{\log\tau}\r)\r)\r), 
\end{align*}
where $A_m(\s)$ can be explicitly described as 
\[
A_m(\s)
=\l(\frac{\s}{(1-\s)^{\frac{m-1}{\s}+2}g_1(\s)}\r)^{\frac{\s}{1-\s}} 
\]
with the constant $g_1(\s)$ of \eqref{eq_gn}. 
\end{corollary}

\begin{proof}
By Proposition \ref{propCumu1}, one can estimate $\kappa=\kappa(\tau; \s, m, \a)$ of Lemma \ref{lemSP} as 
\begin{align} \label{def_kappa}
  \kappa
  = C_m(\s)\tau^{\frac{\s}{1-\s}}(\log\tau)^{\frac{m+\s}{1-\s}}
  \l(1 + O_{\s, m}\l(\frac{\log\log\tau}{\log\tau}\r)\r)
\end{align}
if $\tau>0$ is large enough, where 
\begin{align*}
  C_m(\s)
  = \l(\frac{\s}{(1-\s)^{\frac{m}{\s}+1}g_1(\s)}\r)^{\frac{\s}{1-\s}}.
\end{align*}
Inserting it to \eqref{eqMine56} and using Proposition \ref{propCumu1}, we obtain the corollary. 
\end{proof}

\subsection{Proof of Theorem \ref{thmLD}}\label{secLD}

Let $\mca{X}\subset[T, 2T]$ be a Lebesgue measurable set. 
Then we define 
\begin{align*}
\PP_T^{\mca{X}}(f(t)\in{A})
=\frac{1}{T}\meas\set{t\in\mca{X}}{f(t)\in{A}}
\end{align*}
for any $A$ Lebesgue measurable set on $\RR$, where $f:\RR \to \RR$ is a Lebesgue measurable function. 
Denote by $\mu$ and $\nu$ the measures on $\RR$ such that
\begin{gather*}
\mu(A) = \PP_T^{A_T}\l(\Re {e^{-i\a} P_{m, Y}(\s+it)} \in A \r), \\
\nu(A) = \PP\l( \Re {e^{-i\a} \te_{m}(\s, X)} \in A \r), 
\end{gather*}
respectively, where $A_T=A_T(V, Y; \s, m)$ is given by \eqref{def_A_Y}, and $Y, V$ are functions of $T$ determined later. 
Towards the proof of Theorem \ref{thmLD}, we further define the measures $P$ and $Q$ as
\begin{gather*}
P(A)
=\int_{A} e^{\kappa u}\, d\mu(u), \quad
Q(A)=\int_{A} e^{\kappa u}\, d\nu(u)
\end{gather*}
for any Lebesgue measurable set $A$ on $\RR$, where $\kappa$ is a real number chosen later.
Then, for any $\tau > 0$, we obtain
\begin{align}
\label{eqMine46}
\mu((\tau, \infty))
= \int_{(\tau, \infty)} e^{-\kappa u}\, dP(u)
= \int_{\kappa \tau}^{\infty}e^{-x}P((\tau, x / \kappa))dx, 
\end{align}
and
\begin{align}
\label{eqMine47}
\nu((\tau, \infty))
= \int_{(\tau, \infty)}e^{-\kappa u}\, dQ(u)
= \int_{\kappa \tau}^{\infty}e^{-x}Q((\tau, x / \kappa))dx.
\end{align}
We begin with estimating the difference between $P$ and $Q$.

\begin{lemma}	\label{SPMPQ}
Let $1/2<\s<1$, $m\in\ZZ_{\geq0}$, and $\a\in\RR$.  
Suppose that $Y$ satisfies \eqref{eqY} and the inequality
\[
V^{\frac{\s}{1-\s}}(\log{V})^{\frac{m+\s}{1-\s}}
\leq \e{Y}^{\sigma-\frac{1}{2}}(\log{Y})^m
\]
with a small constant $\e>0$. 
There exists a positive constant $b_{3} = b_{3}(\s, m)$ such that 
for any $|\kappa| \leq b_{3} V^{\frac{\s}{1-\s}}(\log{V})^{\frac{m + \s}{1 - \s}}$ we have 
\begin{align*}
P((c, d))
= Q((c, d)) + E 
\end{align*}
for any $c, d \in \RR$ with $c < d$, where
\begin{align}
\label{ERSPMPQ}
E 
&\ll_{\s, m} 
\frac{\kappa^{1 - \frac{1}{2\s}} (\log{\kappa})^{\frac{m + \s}{2\s}}}{V^{\frac{\s}{1-\s}}(\log{V})^{\frac{m+\s}{1-\s}}} F_{\s, m}(\kappa; \a)\\
& \qquad 
+ (1 + (d - c)V^{\frac{\s}{1-\s}}(\log{V})^{\frac{m+\s}{1-\s}}) 
\Bigg\{\frac{V^{\frac{\s}{1-\s}}(\log{V})^{\frac{m+\s}{1-\s}}}{Y^{\s-\frac{1}{2}}(\log{Y})^{m}} F_{\s, m}(\kappa; \a)\\
& \qquad+\frac{1}{T}\l( V^{\frac{\s}{1-\s}}(\log{V})^{\frac{m + \s}{1-\s}}Y \r)^{V^{\frac{1}{1-\s}}(\log{V})^{\frac{m+\s}{1-\s}}}
+ \exp\l( -b_{2} V^{\frac{1}{1-\s}}(\log{V})^{\frac{m+\s}{1-\s}} \r)\Bigg\},
\end{align}
and $b_{2}$ is the same constant as in Proposition \ref{Key_Prop}.
\end{lemma}

\begin{proof}
Put $L = b_{3} V^{\frac{\s}{1-\s}}(\log{V})^{\frac{m+\s}{1-\s}}$.
By Lemma \ref{BSF2d} and the definition of $P$, we can write
\begin{align}
\label{pSPMPQ2}
P((c, d))
= \int_{-\infty}^{\infty}e^{\kappa x}\Im\int_{0}^{L}G\l( \frac{u}{L} \r)e^{2\pi i u x}f_{c, d}(u)\frac{du}{u} d\mu(x) 
+ E_{1}, 
\end{align}
where 
\begin{align*}
E_{1} \ll
\int_{\RR}\l( K(L(u - c)) + K(L(u - d)) \r)e^{\kappa u}d\mu(u),
\end{align*}
and $G$, $K$, and $f_{c, d}$ are defined by \eqref{def_GKf}.
First, we estimate $E_{1}$. For $z \in \CC$, we define
\begin{align*}
M(z)
= \int_{\RR}e^{z x}d\mu(x).
\end{align*}
Then it holds that
\begin{align*}
  M(z)
  &= \frac{1}{T}\int_{A_{T}}\exp\l(z \Re {e^{-i\a} P_{m, Y}(\s+it)} \r)dt\\
  &= \frac{1}{T}\int_{A_{T}}\exp\l( \frac{z}{2}e^{-i\a}P_{m, Y}(\s+it)
  + \frac{z}{2}e^{i\a}\ol{P_{m, Y}(\s+it)} \r)dt.
\end{align*}
For any $\ell, u \in \RR$, we can write
\begin{align*}
  K(L(u - \ell)) 
  = \l( \frac{\sin(\pi L (u - \ell))}{\pi L (u - \ell)} \r)^2
  &= \frac{2}{L^2}\int_{0}^{L}(L - \xi)\cos(2\pi(u - \ell)\xi)d\xi\\
  &= \frac{2}{L^2}\Re \int_{0}^{L}(L - \xi)e^{2\pi i (u - \ell)\xi}d\xi, 
\end{align*}
and therefore we have
\begin{align}
  \int_{\RR}K(L(u - \ell))e^{\kappa u}d\mu(u)
  &= \frac{2}{L^2} \Re \int_{0}^{L}(L - \xi)\int_{\RR}e^{\kappa u + 2\pi i (u - \ell)\xi}d\mu(u)d\xi\\
  \label{pSPMPQ1}
  &= \frac{2}{L^2} \Re \int_{0}^{L}(L - \xi)e^{-2\pi i \ell \xi}M(\kappa + 2\pi i \xi)d\xi.
\end{align}
Here, we decide $b_{3}$ as $b_{1} / 4$, where $b_{1}$ is the same constant as in Proposition \ref{Key_Prop}.
Then, we can apply Proposition \ref{Key_Prop}, and obtain
\begin{align}
\label{AKProp}
M(\kappa + 2\pi i \xi)
= \EXP{ \exp\l( (\kappa + 2\pi i\xi)\Re {e^{-i\a} P_{m, Y}(\s, X)} \r) }
+ O(E_{2})
\end{align}
for $|\xi| \leq L$, where 
\begin{align*}
E_{2}
= \frac{1}{T}
\l( V^{\frac{\s}{1-\s}}(\log{V})^{\frac{m + \s}{1-\s}}Y \r)^{V^{\frac{1}{1-\s}}(\log{V})^{\frac{m+\s}{1-\s}}}
+ \exp\l( -b_{2} V^{\frac{1}{1-\s}}(\log{V})^{\frac{m+\s}{1-\s}} \r),
\end{align*}
where $b_{2}$ is the same constant as in Proposition \ref{Key_Prop}.
Applying further Lemma \ref{CREP2}, we derive
\begin{align}
\label{AKProp'}
M(\kappa + 2\pi i \xi)
= F_{\s, m}(\kappa + 2\pi i \xi; \a)+ O(E_{2}+E_3), 
\end{align}
where 
\begin{align*}
  E_3
  =F_{\s, m}(\kappa; \a)
  \frac{|\kappa+2\pi i\xi|}{Y^{\s-\frac{1}{2}}(\log{Y})^{m}}
  \ll_{\s, m} F_{\s, m}(\kappa; \a)
  \frac{V^{\frac{\s}{1-\s}}(\log{V})^{\frac{m+\s}{1-\s}}}{Y^{\s-\frac{1}{2}}(\log{Y})^{m}}.
\end{align*}
for $0 \leq \xi \leq L$.
By this formula and Lemma \ref{lemBound1td}, we find that
\begin{align*}
  &\frac{2}{L^2}\int_{0}^{L}(L - \xi)e^{-2\pi i \ell \xi}M(\kappa + 2\pi i \xi)d\xi\\
  &\ll \frac{F_{\s, m}(\kappa; \a)}{L^2}\bigg\{\int_{0}^{c_{2}\kappa}
  (L - \xi) \exp\l(- c_{2} \xi^2 \kappa^{\frac{1}{\s} - 2} (\log{\kappa})^{-\frac{m + \s}{\s}}\r)d\xi\\
  &\qquadf \qquadf +\int_{c_{2}\kappa}^{L}(L - \xi)\exp(-\xi^{1/(2\s)})d\xi\bigg\} + E_{2}+E_3. 
\end{align*}
By the change of variables $\xi = u \kappa^{1 - \frac{1}{2\s}} (\log{\kappa})^{\frac{m+\s}{2\s}}$, 
we see that the first integral is $\ll L\kappa^{1 - \frac{1}{2\s}} (\log{\kappa})^{\frac{m + \s}{2\s}}$.
The second integral is $\ll L$.
Hence, we have
\begin{align*}
  \frac{2}{L^2}\int_{0}^{L}(L - \xi)e^{-2\pi i \ell \xi}M(\kappa + 2\pi i \xi)d\xi
  \ll \frac{\kappa^{1 - \frac{1}{2\s}} (\log{\kappa})^{\frac{m + \s}{2\s}}}{L}F_{\s, m}(\kappa; \a) 
  + E_{2} + E_3
\end{align*}
uniformly for $\ell \in \RR$.
From this estimate and equation \eqref{pSPMPQ1}, we obtain
\begin{align*}
\int_{\RR}K(L(u - c))e^{\kappa u}d\mu(u)
\ll \frac{\kappa^{1 - \frac{1}{2\s}} (\log{\kappa})^{\frac{m + \s}{2\s}}}{L}F_{\s, m}(\kappa; \a) 
+ E_{2}+E_3, 
\end{align*}
and
\begin{align*}
\int_{\RR}K(L(u - d))e^{\kappa u}d\mu(u)
\ll \frac{\kappa^{1 - \frac{1}{2\s}} (\log{\kappa})^{\frac{m + \s}{2\s}}}{L}F_{\s, m}(\kappa; \a) 
+ E_{2}+E_3.
\end{align*}
Thus, we can estimate the error term $E_{1}$ of equation \eqref{pSPMPQ2} by
\begin{align}
\label{pSPMPQ5}
E_{1}
\ll \frac{\kappa^{1 - \frac{1}{2\s}} (\log{\kappa})^{\frac{m + \s}{2\s}}}{L}F_{\s, m}(\kappa; \a) 
+ E_{2}+E_3.
\end{align}

Next, we calculate the main term of \eqref{pSPMPQ2}.
Using Fubini's theorem, we find that the main term is equal to
\begin{align*}
&\Im \int_{0}^{L}G\l( \frac{u}{L} \r)\frac{f_{c, d}(u)}{u} 
M(\kappa + 2\pi i u)du.
\end{align*}
Additionally, by equation \eqref{AKProp'} and the estimates $G(x) \ll 1$ for $0 \leq x \leq 1$ 
and $|f_{c, d}(u) / u| \leq \pi |d - c|$, 
this is equal to
\begin{align}
\label{pSPMPQ3}
\Im \int_{0}^{L}G\l( \frac{u}{L} \r) \frac{f_{c, d}(u)}{u}F_{\s, m}(\kappa + 2\pi i u; \a)du
+ O(L (d - c)(E_{2}+E_3)).
\end{align}
Since we can write
\begin{align*}
F_{\s, m}(\kappa + 2\pi i u; \a)
= \int_{\RR}e^{(\kappa + 2\pi i u)\xi}d\nu(\xi), 
\end{align*}
we find, using Fubini's theorem, that \eqref{pSPMPQ3} equals to
\begin{align*}
\Im \int_{\RR}\int_{0}^{L}
G\l( \frac{u}{L} \r)e^{2\pi i u\xi}f_{c, d}(u)\frac{du}{u}
e^{\kappa \xi}d\nu(\xi)
+ O(L (d - c)(E_{2}+E_3)).
\end{align*}
Applying Lemma \ref{BSF2d} again, this is equal to
\begin{align*}
Q((c, d))
+ O\l( L (d - c)(E_{2}+E_3) + E_{4} \r), 
\end{align*}
where 
\begin{align*}
E_{4}
\ll \int_{\RR}\l(K(L(u - c)) + K(L(u - d))\r)e^{\kappa u}d\nu(u).
\end{align*}
Similarly to the proof of \eqref{pSPMPQ5}, we can obtain
\begin{align*}
E_{4}
\ll \frac{\kappa^{1 - \frac{1}{2\s}} (\log{\kappa})^{\frac{m + \s}{2\s}}}{L} F_{\s, m}(\kappa; \a).
\end{align*}
Thus, we obtain this lemma.
\end{proof}

\begin{proposition}\label{propLDpre}
Let $1/2<\s<1$, $m\in\ZZ_{\geq0}$, and $\a\in\RR$. 
Let $T>0$ be large enough. 
Define the functions $Y$ and $V$ as
\begin{align*}
  Y &=(\log{T})^{B}, \\
  V &=c_{1}(\log{T})^{1-\s}(\log\log{T})^{-m-1}
\end{align*}
with $B=\frac{6}{\s-1/2}$ and $c_{1} = c_{1}(\s, m)$ a small positive constant.
Then there exists a constant $b_m(\sigma)>0$ such that
\begin{align*}
&\PP_T^{A_T}\l(\Re e^{-i\a}P_{m, Y}(\s+it) > \tau \r)\\
&=\PP\l(\Re e^{-i\a}\te_{m}(\s, X) > \tau \r)
\l(1+O_{\s, m}\l(\frac{\tau^{\frac{\s}{1-\s}}(\log\tau)^{\frac{m+\s}{1-\s}}}{V^{\frac{\s}{1-\s}}(\log{V})^{\frac{m+\s}{1-\s}}}\r)\r)
\end{align*}
for any large $\tau$ with $\tau \leq {b}_m(\s)V$. 
\end{proposition}

\begin{proof}
Under our definitions of $Y$ and $V$, the parameter $Y$ satisfies the condition of Lemma \ref{SPMPQ} 
when $T$ is sufficiently large, and $c_{1}$ is sufficiently small.
Let $\tau$ be large number with $\tau \leq b_{m}(\s) V$, where $b_{m}(\s)$ is a suitably small constant chosen later.
Let $\kappa$ satisfy \eqref{eqMine55}. 
Recall that this $\kappa$ satisfies \eqref{def_kappa}, particularly
$\kappa \asymp_{\s, m} \tau^{\frac{\s}{1-\s}}(\log\tau)^{\frac{m+\s}{1-\s}}$.
We can then use Lemma \ref{SPMPQ} for this $\tau$ when $b_{m}(\s)$ is sufficiently small.
By equation \eqref{eqMine46}, we have 
\begin{align*}
\PP_T^{A_T}\l( \Re e^{-i\a}P_{m, Y}(\s+it) > \tau \r)
= \int_{\kappa \tau}^{\infty}e^{-x} P((\tau, x / \kappa))dx.
\end{align*}
Hence, by Lemma \ref{SPMPQ}, we can easily show that this is equal to
\begin{align*}
\int_{\kappa \tau}^{\infty}e^{-x} Q((\tau, x / \kappa))dx
+ O_{\s, m}\l( e^{-\kappa \tau} (E_{1} + E_{2} + E_{3}) \r), 
\end{align*}
where
\begin{gather*}
  E_{1} 
  = \l( \frac{\kappa^{1 - \frac{1}{2\s}} (\log{\kappa})^{\frac{m + \s}{2\s}}}{V^{\frac{\s}{1 - \s}} (\log{V})^{\frac{m + \s}{1 - \s}}} 
  + \frac{\l(V^{\frac{\s}{1 - \s}} (\log{V})^{\frac{m + \s}{1 - \s}}\r)^2}{\kappa Y^{\s - \frac{1}{2}} (\log{Y})^{m}} \r)F_{\s, m}(\kappa; \a),\\
  E_{2}
  = \frac{1}{T}\frac{V^{\frac{\s}{1 - \s}} (\log{V})^{\frac{m+\s}{1-\s}}}{\kappa}
  \l( V^{\frac{\s}{1-\s}}(\log{V})^{\frac{m + \s}{1-\s}}Y \r)^{V^{\frac{1}{1-\s}}(\log{V})^{\frac{m+\s}{1-\s}}}, \\
  E_{3} 
  = \frac{V^{\frac{\s}{1 - \s}} (\log{V})^{\frac{m+\s}{1-\s}}}{\kappa}
  \exp\l( -b_{2} V^{\frac{1}{1-\s}}(\log{V})^{\frac{m+\s}{1-\s}} \r).
\end{gather*}
Also, by equation \eqref{eqMine47}, the main term is equal to 
$\PP\l( \Re e^{-i\a}\te_{m}(\s, X) > \tau \r)$.

Using Proposition \ref{propFT} and noting the definitions of $V$ and $Y$ and the range of $\tau$, we have
\begin{align*}
  e^{-\kappa \tau} E_{1}
  \ll \PP\l(\Re e^{-i\a} \te_{m}(\s, X) > \tau\r) 
  \frac{\kappa^{2 - \frac{1}{2\s}} (\log{\kappa})^{\frac{m + \s}{2\s}} \sqrt{{f}''_{\s, m}(\kappa; \a)}}{
    V^{\frac{\s}{1 - \s}} (\log{V})^{\frac{m + \s}{1 - \s}}}.
\end{align*}
It follows from this estimate, \eqref{eq:Mine30a}, 
and the estimate $\kappa \asymp_{\s, m} \tau^{\frac{\s}{1-\s}}(\log{\tau})^{\frac{m + \s}{1 - \s}}$ that
\begin{align*}
  e^{-\kappa \tau} E_{1}
  \ll_{\s, m} \PP\l(\Re e^{-i\a} \te_{m}(\s, X) > \tau\r) 
  \frac{\tau^{\frac{\s}{1 - \s}} (\log{\tau})^{\frac{m + \s}{1 - \s}}}{V^{\frac{\s}{1 - \s}} (\log{V})^{\frac{m + \s}{1 - \s}}}.
\end{align*} 
By the definitions of $Y$ and $V$, it holds that $E_{2} \ll T^{-3/4}$ when $c_{1}$ is small.
Hence, using Corollary \ref{corLD}, we have
\begin{align*}
  e^{-\kappa \tau}E_{2}
  \ll \PP\l(\Re e^{-i\a} \te_{m}(\s, X) > \tau\r) T^{- 1/2}.
\end{align*}
Also, when $\tau \leq b_{m}(\s) V$ with $b_{m}(\s)$ small, we find, using Corollary \ref{corLD} again, that
\begin{align*}
  e^{-\kappa \tau}E_{3}
  &\leq \PP\l(\Re e^{-i\a} \te_{m}(\s, X) > \tau\r) \exp\l( - \frac{b_{2}}{2} V^{\frac{1}{1-\s}}(\log{V})^{\frac{m+\s}{1-\s}} \r)\\
  &\ll_{\s, m} \frac{1}{V^{\frac{1}{1-\s}}(\log{V})^{\frac{m+\s}{1-\s}}}
  \PP\l(\Re e^{-i\a} \te_{m}(\s, X) > \tau\r).
\end{align*}
Thus, we complete the proof of this proposition.
\end{proof}

\begin{proof}[Proof of Theorem \ref{thmLD}]
Let $\tau$ be large with $\tau \leq b_{m}'(\s) V$ with $b_{m}'(\s)$ a small positive constant to be chosen later.
We choose the functions $Y, V$ as in Proposition \ref{propLDpre} and put 
\begin{gather*}
\e
=WY^{\frac{1}{2}-\s}
=W(\log{T})^{-6}, 
\end{gather*}
where $W$ is a function of $T$ determined later. 
Here, we assume that $\epsilon \ll 1$ is satisfied. 
Denote by $B_T=B_{T}(Y, W; \s, m)$ the subset of $[0,T]$ given by \eqref{def_B_Y}. 
Then the condition $t \in B_T$ deduces that 
\begin{gather*}
\Re e^{-i\a}\te_m(\s+it) > \tau
\quad\Longrightarrow\quad
\Re e^{-i\a}P_{m, Y}(\s+it) > \tau-\e
\end{gather*}
by the definition. 
Thus we have 
\begin{align}\label{eq:Mine03}
&\PP_T\l(\Re e^{-i\a}\te_m(\s+it) > \tau\r)\\
&=\PP_T^{B_{T}}\l(\Re e^{-i\a}\te_m(\s+it) > \tau\r)+O\l(\frac{1}{T}\meas([T, 2T]\setminus{B}_{T})\r)\\
&\leq\PP_T^{B_{T}}\l(\Re e^{-i\a}P_{m, Y}(\s+it) > \tau-\e\r)+O\l(\frac{1}{T}\meas([T, 2T]\setminus{B}_{T})\r). 
\end{align}
In a similar way, we also obtain the inequality
\begin{align}\label{eq:Mine04}
&\PP_T\l(\Re e^{-i\a}\te_m(\s+it) > \tau\r) \\
&\geq\PP_T^{B_{T}}\l(\Re e^{-i\a}P_{m, Y}(\s+it) > \tau+\e\r)+O\l(\frac{1}{T}\meas([T, 2T]\setminus{B}_{T})\r). 
\end{align}
Note that $\PP_T^{B_{T}}(\cdots)$ is modified as 
\begin{align}\label{eq:Mine05}
&\PP_T^{B_{T}}\l(\Re e^{-i\a}P_{m, Y}(\s+it) > \tau\pm\e\r)\\
&=\PP_T^{A_{T}}\l(\Re e^{-i\a}P_{m, Y}(\s+it) > \tau\pm\e\r)\\
&\qquad+O\l(\frac{1}{T}\meas([0, T]\setminus{A}_{T})+\frac{1}{T}\meas([T, 2T]\setminus{B}_{T})\r)
\end{align}
by using the set $A_T=A_T(V, Y; \s, m)$ given by \eqref{def_A_Y}. 
Since $\epsilon \ll1$, there exists a constant ${b}'_m(\s)>0$ such that $\tau\leq{b}'_m(\s)V$ implies $\tau\pm \epsilon \leq{b}_m(\s)V$, 
where ${b}_m(\s)$ is the positive constant of Proposition \ref{propLDpre}. 
Then we deduce from Proposition \ref{propLDpre} the asymptotic formulas
\begin{align}\label{eq:Mine01}
&\PP_T^{A_T}\l(\Re e^{-i\a}P_{m, Y}(\s+it) > \tau\pm\e\r)\\
&=\PP\l(\Re e^{-i\a}\te_{m}(\s, X) > \tau\pm\e\r)
\l(1+O_{\s, m}\l(\frac{\tau^{\frac{\s}{1-\s}}(\log\tau)^{\frac{m+\s}{1-\s}}}{(\log{T})^\s(\log\log{T})^m}\r)\r)
\end{align}
in the range $\tau\leq{b}'_m(\s)V$. 
Next, we evaluate the differences
\begin{gather}\label{eq:diff}
\PP\l(\Re e^{-i\a}\te_{m}(\s, X) > \tau\pm\e\r)
-\PP\l(\Re e^{-i\a}\te_{m}(\s, X) > \tau\r). 
\end{gather}
Note that we obtain
\begin{align*}
\PP\l(\Re e^{-i\a}\te_{m}(\s, X) > \tau+\gamma\r)
&=\int_{\gamma}^{\infty} \DF_{\s, m}(x+\tau; \a)\,|dx| \\
&=F_{\s, m}(\kappa; \a)e^{-\kappa\tau}\int_{\gamma}^\infty{e}^{-\kappa{x}}\nDF_{\s, m}^{\tau}(x; \a)\, |dx|
\end{align*}
for any $\gamma\in\RR$ by using the functions $\DF_{\s, m}(x; \a)$, $\nDF_{\s, m}^{\tau}(x; \a)$ 
defined by \eqref{def_odM}, \eqref{eqMine29a} respectively. 
It yields
\begin{align*}
&\PP\l(\Re e^{-i\a}\te_{m}(\s, X) > \tau-\e\r) - \PP\l(\Re e^{-i\a}\te_{m}(\s, X) > \tau\r)\\
&={F}_{\s, m}(\kappa; \a)e^{-\kappa\tau}\int_{-\e}^0{e}^{-\kappa{x}}\nDF_{\s, m}^{\tau}(x; \a)\, |dx|\\
&\ll \frac{{F}_{\s, m}(\kappa; \a)e^{-\kappa\tau}}{\sqrt{{f}''_{\s, m}(\kappa; \a)}} \frac{{e}^{\kappa\e}-1}{\kappa}
\end{align*}
by using Proposition \ref{propE}, and similarly, 
\begin{align*}
&\PP\l(\Re e^{-i\a}\te_{m}(\s, X) > \tau\r) - \PP\l(\Re e^{-i\a}\te_{m}(\s, X) > \tau+\e\r)\\
&={F}_{\s, m}(\kappa; \a)e^{-\kappa\tau}\int_{0}^{\epsilon} {e}^{-\kappa{x}}\nDF_{\s, m}^{\tau}(x; \a)\, |dx|\\
&\ll \frac{{F}_{\s, m}(\kappa; \a)e^{-\kappa\tau}}{\sqrt{{f}''_{\s, m}(\kappa; \a)}} \frac{1-e^{-\kappa \e}}{\kappa} 
\end{align*}
follows. 
Combining them, we evaluate \eqref{eq:diff} as 
\begin{align*}
&\PP\l(\Re e^{-i\a}\te_{m}(\s, X) > \tau\pm\e\r)
-\PP\l(\Re e^{-i\a}\te_{m}(\s, X) > \tau\r) \\ 
&\ll \PP\l(\Re e^{-i\a}\te_{m}(\s, X) > \tau\r) \delta
\end{align*}
by Proposition \ref{propFT}, where we put $\delta=\max\{{e}^{\kappa\e}-1, 1-e^{-\kappa\e}\}$.  
Hence formula \eqref{eq:Mine01} derives
\begin{align}\label{eq:Mine06}
&\PP_T^{A_T}\l(\Re e^{-i\a}P_{m, Y}(\s+it) > \tau\pm\e\r)\\
&=\PP\l(\Re e^{-i\a}\te_{m}(\s, X) > \tau\r)
\l(1+O_{\s, m}\l(\frac{\tau^{\frac{\s}{1-\s}}(\log\tau)^{\frac{m+\s}{1-\s}}}{(\log{T})^\s(\log\log{T})^m}
+\delta\r)\r). 
\end{align}
By \eqref{eq:Mine03}, \eqref{eq:Mine05}, and \eqref{eq:Mine06}, we obtain the inequality
\begin{align*}
&\PP_T\l(\Re e^{-i\a}\te_{m}(\s+it) > \tau\r)\\
&\leq\PP\l(\Re e^{-i\a}\te_{m}(\s, X) > \tau\r)
\l(1+O_{\s, m}\l(\frac{\tau^{\frac{\s}{1-\s}}(\log\tau)^{\frac{m+\s}{1-\s}}}{(\log{T})^\s(\log\log{T})^m}
+\delta \r)\r)\\
&\qquad+O\l(\frac{1}{T}\meas([T, 2T]\setminus{A}_{T})+\frac{1}{T}\meas([T, 2T]\setminus{B}_{T})\r).  
\end{align*}
If we use \eqref{eq:Mine04} in place of \eqref{eq:Mine03}, we also obtain
\begin{align*}
&\PP_T\l(\Re e^{-i\a}\te_{m}(\s+it) > \tau\r)\\
&\geq\PP\l(\Re e^{-i\a}\te_{m}(\s, X) > \tau\r)
\l(1+O_{\s, m}\l(\frac{\tau^{\frac{\s}{1-\s}}(\log\tau)^{\frac{m+\s}{1-\s}}}{(\log{T})^\s(\log\log{T})^m}
+\delta \r)\r)\\
&\qquad+O\l(\frac{1}{T}\meas([T, 2T]\setminus{A}_{T})+\frac{1}{T}\meas([T, 2T]\setminus{B}_{T})\r).
\end{align*}
Therefore, the asymptotic formula
\begin{align}\label{eq:Mine07}
&\PP_T\l(\Re e^{-i\a}\te_{m}(\s+it) > \tau\r)\\
&=\PP\l(\Re e^{-i\a}\te_{m}(\s, X) > \tau\r)
\l(1+O_{\s, m}\l(\frac{\tau^{\frac{\s}{1-\s}}(\log\tau)^{\frac{m+\s}{1-\s}}}{(\log{T})^\s(\log\log{T})^m}
+\delta\r)\r)\\
&\qquad+O\l(\frac{1}{T}\meas([T, 2T]\setminus{A}_{T})+\frac{1}{T}\meas([T, 2T]\setminus{B}_{T})\r)  
\end{align}
follows. 
Then, we determine the function $W$ as 
\[W
= a(\log{T})(\log\log{T})^{-m-1},  
\]
where the constant $a=a(\sigma,m)>0$ is taken so that 
\begin{gather*}
\l((\log{T})(\log{Y})^{-2(m+1)}\r)^{\frac{m}{2m+1}} \leq W \leq (\log{T})(\log{Y})^{-(m+1)}
\end{gather*}
is satisfied. 
By \eqref{def_kappa}, we have $\kappa \ll (\log{T})^{\sigma}(\log\log{T})^{m}$ in the range $\tau\leq{b}'_m(\s)V$. 
Hence $\kappa \epsilon \ll (\log{T})^{-4}$ is valid due to $\epsilon= WY^{\frac{1}{2}-\s}$. 
It gives that the condition $\epsilon\ll1$ is satisfied. 
Furthermore, we evaluate $\delta=\max\{{e}^{\kappa\e}-1, 1-e^{-\kappa\e}\}$ as
\begin{gather}\label{eq:Mine08}
\delta
\ll (\log{T})^{-4}
\ll \frac{\tau^{\frac{\s}{1-\s}}(\log\tau)^{\frac{m+\s}{1-\s}}}{(\log{T})^\s(\log\log{T})^m}.
\end{gather}
Finally, we apply Lemmas \ref{ESE} and \ref{ESEPE} to deduce
\begin{align*}
&\frac{1}{T}\meas([T, 2T]\setminus{A}_{T})
= \PP_T\l( |P_{m, Y}(\s+it)| > V \r)\\
&\ll\exp\l( -c V^{\frac{1}{1-\s}}(\log{V})^{\frac{m+\s}{1-\s}} \r)
\ll\exp\l(-c' \frac{\log{T}}{\log\log{T}}\r),
\end{align*}
and
\begin{align*}
\frac{1}{T}\meas([T, 2T]\setminus{B}_{T})
&\ll\exp\l( -c (W(\log{T})^{m})^{\frac{1}{m+1}} \r)
\ll\exp\l(-c' \frac{\log{T}}{\log\log{T}}\r) 
\end{align*}
with positive constants $c=c(\s, m)$ and $c'=c'(\s, m)$. 
Note that Corollary \ref{corLD} yields the upper bound
\[\PP\l(\Re e^{-i\a}\te_{m}(\s, X) > \tau\r)^{-1}
\ll\exp\l(\frac{c'}{2}\frac{\log{T}}{\log\log{T}}\r)\]
in the range $\tau\leq{b}'_m(\s)V$ if $b'_m(\s)>0$ is sufficiently small. 
Hence we arrive at
\begin{align}\label{eq:Mine09}
&\frac{1}{T}\meas([T, 2T]\setminus{A}_{T})+\frac{1}{T}\meas([T, 2T]\setminus{B}_{T}) \\
&\ll\PP\l(\Re e^{-i\a}\te_{m}(\s, X) > \tau\r)\exp\l(-\frac{c'}{2}\frac{\log{T}}{\log\log{T}}\r) \\
&\ll\PP\l(\Re e^{-i\a}\te_{m}(\s, X) > \tau\r)\frac{\tau^{\frac{\s}{1-\s}}(\log\tau)^{\frac{m+\s}{1-\s}}}{(\log{T})^\s(\log\log{T})^m}. 
\end{align}
The result follows from \eqref{eq:Mine07}, \eqref{eq:Mine08}, and \eqref{eq:Mine09}. 
\end{proof}

\section{\textbf{Appendix}}

\subsection{Proof of Proposition \ref{GCGCS}}

We give a lemma to prove of Proposition \ref{GCGCS}.
The following lemma is just a rewrite of the statement of Proposition 1 in \cite{II2020}.

\begin{lemma} \label{EQLtecl}
  The following statements are equivalent.
  \begin{itemize}
    \item[(A)'] The Lindel\"of Hypothesis is true.
    \item[(B)'] For any positive integer $m$, the estimate $\Re \te_{m}(\tfrac{1}{2} + it) = o(\log{t})$ holds as $t \rightarrow + \infty$.
    \item[(C)'] There exists a positive integer $m$ such that
    the estimate $\Re \te_{m}(\tfrac{1}{2} + it) = o(\log{t})$ holds as $t \rightarrow + \infty$. 
  \end{itemize}
\end{lemma}

\begin{proof}[Proof of Proposition \ref{GCGCS}]
  The implication (B) $\Rightarrow$ (C) is clear.
  We first consider (A) $\Rightarrow$ (B).
  Let $m \in \ZZ_{\geq 1}$ be fixed. 
  We assume the Lindel\"of Hypothesis.
  By the same method as the proof of Theorem 13.6 (B) in \cite{TT}, 
  we can show that the Lindel\"of Hypothesis implies $\Re \te_{m}(\s + it) = o(\log{t})$ for any $\s \geq \frac{1}{2}$.

  Next, we show that the Lindel\"of Hypothesis is true when
  $\Re \te_{m_{0}}(\s + it) = o(\log{t})$ for any fixed $\s > \frac{1}{2}$ with some $m_{0} \in \ZZ_{\geq 1}$.
  By Lemma \ref{EQLtecl}, it suffices to derive, from the assumption,
   $|\Re \te_{m_{0}}(\tfrac{1}{2} + it)| \leq \e \log{t}$
  for any small $\e > 0$, $t \geq t_{0}(\e)$.
  Put $\s(\e) = \frac{1}{2} + \e^2$.
  We can write
  \begin{align} \label{pGCGCS1}
    \Re \te_{m_{0}}(\tfrac{1}{2} + it)
    = \int_{\s(\e)}^{\frac{1}{2}}\Re \te_{m_{0} - 1}(\s + it)d\s + \Re \te_{m_{0}}(\s(\e) + it).
  \end{align}
  From our assumption, the second term on the right hand side is $o(\log{t})$.
  When $m_{0} \geq 2$, we can show that $\Re \te_{m_{0} - 1}(\s + it) \ll \log{t}$ similarly to the proof of Theorem 9.9 (A) in \cite{TT}.
  Hence, when $m_{0} \geq 2$, the absolute value of the integral on the right hand side of \eqref{pGCGCS1} is $\leq \e\log{t}$, 
  where $C$ is a positive constant.
  When $m_{0} = 1$, using the well known formula (see Theorem 9.6 (B) in \cite{TT})
  \begin{align*}
    \log{\zeta(\s + it)} = \sum_{|t - \gamma| \leq 1}\log(s - \rho) + O(\log{t}) \quad \text{for $-1 \leq \s \leq 2$}
  \end{align*}
  and the fact that the number of zeros with $|t - \gamma| \leq 1$ is $= O(\log{t})$ (see Theorem 9.2 in \cite{TT}), we see that
  \begin{align*}
    \l|\int_{\s(\e)}^{\frac{1}{2}}\Re \log{\zeta(\s + it)}d\s\r|
    \leq \e \log{t}.
  \end{align*}
  Hence, we obtain $|\Re \te_{m_{0}}(\tfrac{1}{2} + it)| \leq \e \log{t}$ for any $t \geq t_{0}(\e)$.
\end{proof}

\subsection{Positivity of the probability density function}

\begin{proposition}\label{prop:app}
Let $\DF_{\s, m}(z)$ be the probability density function of Proposition \ref{propPDF}. 
\begin{itemize}
\item[$(\mathrm{i})$] 
Let $m\geq1$. 
If $1/2\leq\s<1$, then $\DF_{\s, m}(z)>0$ for all $z\in\CC$. 
If $\s\geq1$, then $\DF_{\s, m}$ is compactly supported. 
\item[$(\mathrm{ii})$] 
Let $m=0$. 
If $1/2<\s\leq1$, then $\DF_{\s, m}(z)>0$ for all $z\in\CC$. 
If $\s>1$, then $\DF_{\s, m}$ is compactly supported.
\end{itemize}
\end{proposition}

Using Theorem \ref{thmDB1} and Proposition \ref{prop:app},
we strengthen the previous result due to the first and second authors \cite{EI2020} for the denseness of $\te_{m}(\s + it)$.
Actually, we obtain the following corollary.

\begin{corollary}
  Let $m$ be a positive integer, and $\frac{1}{2} < \s < 1$.
  For any rectangle $\mca{R} \subset \CC$, we have
  \begin{align*}
    \lim_{T \rightarrow + \infty}\PP_{T}(\te_{m}(\s + it) \in \mca{R})
    = \int_{\mca{R}} \DF_{\s, m}(z)|dz| > 0.
  \end{align*}
  In particular, the set $\set{\te_{m}(\s + it)}{t \in [0, \infty)}$ is dense in the complex plane.
\end{corollary}

We introduce some notation to prove Proposition \ref{prop:app}.
The distribution of the random variable $\te_m(\s, X)$ is the probabilistic measure on $\mathbb{C}$ defined as 
\begin{gather}\label{eqMine10}
\mu_{\s, m}(A)=\PP(\te_m(\s, X)\in{A}). 
\end{gather}
Let $p$ be a prime number. 
Then we also define
\[\mu_{\s, m, p}(A)=\PP(\te_{m, p}(\s, X(p))\in{A}), \] 
where $\te_{m, p}(\s, X(p))$ is defined from \eqref{eqlocal}.

\begin{lemma}\label{lemConv}
Let $(\s, m)\in\mca{A}$. 
Then the convolution measure 
\begin{gather}\label{eqMine16}
\nu_{\sigma, m, N}=\mu_{\sigma, m, p_1}\ast\cdots\ast\mu_{\sigma, m, p_N}
\end{gather}
converges weakly to $\mu_{\sigma, m}$ as $N\to\infty$, 
where $p_n$ indicates the $n$-th prime number. 
Furthermore, the convergence is absolute in the sense that it converges to $\mu_{\sigma, m}$ in any order of terms of the convolution. 
\end{lemma}

\begin{proof}
Recall that $\te_{m, p}(\s, X(p))$ and $\te_{m, q}(\s, X(q))$ are independent if $p$ and $q$ are distinct prime numbers. 
Hence the distribution of $\te_{m, p}(\s, X(p))+\te_{m, q}(\s, X(q))$ equals to $\mu_{\sigma, m, p}\ast\mu_{\sigma, m, q}$. 
More generally, we see that 
\[\nu_{\sigma, m, N}(A)
=\PP\l(\sum_{n\leq{N}}\te_{m, p_n}(\s, X(p_n))\in{A}\r)\]
for any Borel set $A$ on $\CC$. 
By Lemma \ref{lemASConv}, $\sum_{n\leq{N}}\te_{m, p_n}(\s, X(p_n))\to\te_m(\s, X)$ 
in law as $N\to\infty$, i.e.\ $\nu_{\sigma, m, N}\to\mu_{\sigma, m}$ weakly. 
The absoluteness of the convergence follows from \eqref{plemASConv1} and Jessen--Wintner \cite[Theorem 6]{JW1935}. 
\end{proof}

In general, the support of a probability measure $P$ on $\CC$ is defined as
\[\supp(P)
=\Set{z\in\CC}{\text{$P(A)>0$ for any Borel set $A$ with $z\in{A}^i$}}, \]
where $A^i$ is the interior of $A$. 
We know that $\supp(P)$ is a non-empty closed subset of $\CC$. 
Applying Lemma \ref{lemConv}, we study the support of $\mu_{\s, m}$.

\begin{lemma}\label{lemSupp}
Let $\mu_{\s, m}$ be the probability measure defined as \eqref{eqMine10}.
\begin{itemize}
\item[$(\mathrm{i})$] Let $m\geq1$. 
If $1/2\leq\s<1$, then $\supp(\mu_{\s, m})=\CC$. 
If $\s\geq1$, then $\supp(\mu_{\s, m})$ is a compact subset of $\CC$. 
\item[$(\mathrm{ii})$] Let $m=0$. 
If $1/2<\s\leq1$, then $\supp(\mu_{\s, m})=\CC$. 
If $\s>1$, then $\supp(\mu_{\s, m})$ is a compact subset of $\CC$. 
\end{itemize}
\end{lemma}

\begin{proof}
Let $\{A_n\}_{n = 1}^{\infty}$ be a sequence of subsets of $\CC$. 
We denote by $\lim_{n \to \infty}{A_n}$ the set of all points in $\CC$ 
that may be represented in at least one way as the limit of a sequence of points $a_n\in{A}_n$. 
Jessen--Wintner \cite[Theorem 3]{JW1935} proved that 
\begin{align} \label{plemSupp1}
  \supp(P)=\lim_{N\to\infty}\l(\supp(P_1)+\cdots+\supp(P_N)\r)
\end{align}
if the convolution measure $P_1\ast\cdots\ast{P}_N$ converges weakly to $P$ as $N\to\infty$. 
Here, $A + B$ means the set $\set{a + b \in \CC}{a \in A, b \in B}$.
Applying further \cite[Theorem 14]{JW1935} with $P_n=\mu_{\s, m, p_n}$ for $m = 0$, we obtain assertion $(\mathrm{ii})$. 
Then, we consider the case $m\geq1$. 
By Lemma \ref{lemConv} and \eqref{plemSupp1}, we have 
\[\supp(\mu_{\s, m})=\lim_{N\to\infty}\l(\supp(\mu_{\s, m, p_1})+\cdots+\supp(\mu_{\s, m, p_N})\r). \]
Note that the support of every $\mu_{\s, m, p}$ is determined as 
\[\supp(\mu_{\s, m, p})=\Set{ \te_{m, p}(\s, e^{i\theta}) }{ \theta\in[0, 2\pi) }\]
by the definition. 
First, we let $1/2\leq\s<1$. 
In this case we apply \cite[Theorem 5.4]{St2007} to deduce that for any $z\in\CC$, $N_0\geq1$, and $\e>0$, we have 
\[\l|\l(z-\sum_{n<N_0}\frac{\Li_{m+1}(p_n^{-\s})}{(\log{p}_n)^m}\r)-\sum_{N_0<{n}\leq{N}}\frac{p_n^{-\s}e^{i\theta_n}}{(\log{p}_n)^m}\r|<\e\]
with some $N=N(z, N_0, \e)>N_0$ and $\{\theta_n\}_{N_0<{n}\leq{N}}\in[0, 2\pi)^{N-N_0}$. 
We also derive  
\begin{align*}
  &\l|\sum_{N_0<{n}\leq{N}}\frac{\Li_{m+1}(p_n^{-\s}{e}^{i\theta_n})}{(\log{p}_n)^m}
  -\sum_{N_0<{n}\leq{N}}\frac{p_n^{-\s}e^{i\theta_n}}{(\log{p}_n)^m}\r|\\
  &= \l| \sum_{N_{0} < n \leq N} \frac{1}{(\log{p_{n}})^{m}}
  \sum_{k = 2}^{\infty}\frac{e^{ik\theta_{n}}}{k^{m+1}p_{n}^{k\s}} \r|
  \ll \sum_{n > N_{0}}\frac{1}{p^{2\s}(\log{p_{n}})^{m}}
  < \e
\end{align*}
if $N_0$ is sufficiently large. 
These imply $\supp(\mu_{\s, m})=\CC$ for $1/2\leq\s<1$. 
Next, we let $\s\geq1$. 
Then we have 
\[\sum_{n=1}^\infty\te_{m, p_n}(\s, e^{i\theta_n})
\ll\sum_{n=1}^\infty\frac{1}{p_n\log{p}_n}<\infty\]
for any $\{\theta_n\}_{n = 1}^{\infty} \in \prod_{n = 1}^{\infty}[0, 2\pi)$ in this case. 
Hence $\supp(\mu_{\s, m})$ is included in a bounded disk, which completes the proof. 
\end{proof}

\begin{proof}[Proof of Proposition \ref{prop:app}]
Remark that the support of the function $\DF_{\s, m}$ is equal to $\supp(\mu_{\s, m})$ studied in Lemma \ref{lemSupp}. 
We first show assertion $(\mathrm{i})$.
Let $m\geq1$. 
Then the fact that $\DF_{\s, m}$ is compactly supported for $\s\geq1$ is a direct consequence of Lemma \ref{lemSupp}. 
Let $1/2\leq\s<1$. 
To prove the positivity of $\DF_{\s, m}$, we define two probabilistic measures 
\[\nu^\flat_{\s, m, N}=\mu_{\s, m, 2}\ast\mu_{\s, m, p_1^\flat}\ast\cdots\ast\mu_{\s, m, p_N^\flat}
\quad\text{and}\quad 
\nu^\#_{\s, m, N}=\mu_{\s, m, p_1^\#}\ast\cdots\ast\mu_{\s, m, p_N^\#}\]
as analogues of \eqref{eqMine16}, where $p_n^\flat$ is the $n$-th prime number congruent to $1 \pmod{4}$, 
and $p_n^\#$ is the $n$-th prime number congruent to $-1 \pmod{4}$. 
Then it can be proved similarly to Lemma \ref{lemConv} 
that $\nu^\flat_{\s, m, N}$ and $\nu^\#_{\s, m, N}$ converge weakly to some probability measures $\mu_{\s, m}^\flat$ and $\mu_{\s, m}^\#$ 
as $N\to\infty$, respectively. 
One can check that the limit measures $\mu_{\s, m}^\flat$ and $\mu_{\s, m}^\#$ satisfy many of the same properties 
as $\mu_{\s, m}$ described above. 
In particular, we have 
\begin{align} \label{ppropPDF1}
  \supp(\mu_{\s, m}^\flat)=\supp(\mu_{\s, m}^\#)=\CC 
\end{align}
for $1/2\leq\s<1$ along the same line of Lemma \ref{lemSupp}, where we use the prime number theorem in arithmetic progression modulo $4$.
Furthermore, there exist non-negative continuous functions $\DF_{\s, m}^\flat$ and $\DF_{\s, m}^\#$ such that 
\[\mu_{\s, m}^\flat(A)=\int_{A}\DF_{\s, m}^\flat(z)\, |dz| 
\quad\text{and}\quad 
\mu_{\s, m}^\#(A)=\int_{A} \DF_{\s, m}^\#(z)\, |dz| \]
for all Borel set $A$. We see from \eqref{ppropPDF1} 
that the supports of $\DF_{\s, m}^\flat$ and $\DF_{\s, m}^\#$ are equal to $\CC$ if $1/2\leq\s<1$. 
Recall that $\nu_{\s, m, N}^\flat\ast\nu_{\s, m, N}^\#$ converges weakly to $\mu_{\s, m}$ as $N\to\infty$ by Lemma \ref{lemConv}. 
Hence we deduce the equality
\[\DF_{\s, m}(z)=\int_{\CC} \DF^\flat_{\s, m}(z-w)\DF^\#_{\s, m}(w)|dw|\]
for any $z\in\CC$. 
Since the functions $\DF_{\s, m}^\flat$ and $\DF_{\s, m}^\#$ are continuous and are not identically zero on every disk on $\CC$, 
we see that $\DF_{\s, m}(z)>0$ for any $z \in \CC$. 
Hence the proof of assertion $(\mathrm{i})$ is completed. 
We note that assertion $(\mathrm{ii})$ is just a consequence of \cite[Theorem 14]{JW1935}. 
\end{proof}


\begin{ackname}
The authors would like to thank Kohji Matsumoto for his helpful comments.  
The work of the second author was supported by Grant-in-Aid for JSPS Fellows (Grant Number 21J00425).
The work of the third author was supported by Grant-in-Aid for JSPS Fellows (Grant Number 21J00529). 
\end{ackname}

\end{document}